\documentclass[10pt,a4paper,leqno]{article}

\usepackage{amsmath,amscd,amsfonts,amsthm,amssymb,dsfont,bbm,manfnt}
\usepackage{url,color,appendix,eucal,comment,tabu,float}
\usepackage[usenames,dvipsnames, table]{xcolor}
\usepackage{multicol}
\usepackage[utf8]{inputenc}
\usepackage[all]{xy}
\usepackage{lscape,rotating}
\usepackage{graphicx}
\usepackage[nolabel]{showlabels}
\usepackage{tikz}
\usepackage{dynkin-diagrams}
\usepackage{caption}
\usepackage[hypertexnames=false]{hyperref}

\newenvironment{pf}
{\medskip\noindent {\it Proof.  }}
{\hfill\nobreak $\Box$ \par\bigbreak}

\makeatletter

\@addtoreset{table}{section}
\makeatother

\newcommand{\isomo}{\overset{\sim}{\rightarrow}}
\newcommand{\GL}{\mathrm{GL}}
\newcommand{\ps}{\par \smallskip}

\newcommand{\Z}{\mathbb{Z}}
\newcommand{\Q}{\mathbb{Q}}
\newcommand{\R}{\mathbb{R}}

\newcommand{\sff}{\sffamily\selectfont}
\newcommand{\resb}{{\rm res}\,}
\newcommand{\resq}{{\rm qres}\,}

\newcounter{introcounter}
\newtheorem{thm}[subsection]{Theorem}
\newtheorem{prop}[subsection]{Proposition}
\newtheorem{definition}[subsection]{Definition}

\newtheorem{lemma}[subsection]{Lemma}
\newtheorem{remark}[subsection]{Remark}
\newtheorem{cor}[subsection]{Corollary}
\newtheorem{example}[subsection]{Example}

\newtheorem{fact}[subsection]{Fact}

\newtheorem{thmintro}{Theorem}

\usepackage{titlesec}
\titleformat{\section}{\bf \normalsize}{\arabic{section}.}{1 em}{}
\titleformat{\subsection}
{}
{\bf \arabic{section}.\arabic{subsection}.}
{0.5 em}
{\normalsize}
\titleformat{\subsubsection}[runin]
{\small \bf}
{}
{}
{}

\hypersetup{colorlinks=true,linkcolor=black,citecolor=black,urlcolor=Periwinkle,pdfborderstyle={/S/U/W 1}}

\makeatletter

\@addtoreset{equation}{section}
\makeatother

\makeatletter

\@addtoreset{table}{section}
\makeatother

\begin{document}

\title{Unimodular lattices of rank $29$ and related even genera of small determinant}

\author{
Ga\"etan Chenevier\thanks{\texttt{gaetan.chenevier@math.cnrs.fr}, CNRS, \'Ecole Normale Sup\'erieure-PSL, D\'epartement de Math\'ematiques et Application, 45 rue d'Ulm, 75230 Paris Cedex, France. During this work, G. Chenevier was supported by the project ANR-19-CE40-0015-02 COLOSS.}\\
\and
Olivier Ta\"ibi\thanks{\texttt{olivier.taibi@math.cnrs.fr}, 
CNRS, \'Ecole Normale Sup\'erieure de Lyon, Unité de Mathématiques Pures et Appliquées, 46, allée d’Italie 69364 Lyon Cedex 07, France. During this work, O. Ta\"ibi was supported by the project ANR-19-CE40-0015-02 COLOSS.
}}


\maketitle

\begin{abstract} 
We classify the unimodular Euclidean integral lattices of rank $29$ by developing an elementary, yet very efficient, inductive method.
As an application, we determine the isometry classes 
of even lattices of rank $\leq 28$ and prime (half-)determinant $\leq 7$. 
We also provide new isometry invariants allowing for
independent verification of the completeness of our lists,
and we give conceptual explanations of 
some {\it unique orbit phenomena} discovered during our computations. 
Some of the genera classified here are orders of magnitude larger 
than any genus previously classified. In a forthcoming companion paper, 
we use these computations to study the cohomology of ${\rm GL}_n(\Z)$.
\end{abstract}	
\ps
\ps

\section{Introduction}
\label{sect:intro}	

\subsection{The classification of unimodular lattices}
\label{sect:introunimod}
${}^{}$ Let us denote by ${\rm X}_n$ the set of isometry classes of 
{\it unimodular integral Euclidean} lattices of rank $n \geq 1$ (see Sect.~\ref{subsect:notation}). 
The simplest example of an element of ${\rm X}_n$ is the (class of the) {\it standard}, or {\it cubic}, lattice ${\rm I}_n:=\Z^n$.
We know from reduction theory or the geometry of numbers (Hermite, Minkowski) 
that ${\rm X}_n$ is a finite set; we even know its {\it mass} in the sense of Smith-Minkowski-Siegel. 

Determining the exact cardinality and finding representatives of ${\rm X}_n$ 
is, however, a difficult problem and a classical topic in number theory. 
Its origin goes back at least to the works of Lagrange and Gauss on 
counting the number of representations of an integer as a sum of squares, 
a question concerning the single lattice ${\rm I}_n$ a priori 
but intimately linked to the whole of ${\rm X}_n$, e.g. by the {\it Siegel-Weil} formula.
The known values of $|{\rm X}_n|$ are gathered in Table~\ref{tab:tabcardXn} below. \ps

\begin{table}[H]
\centering
\tabcolsep=5pt
\small
\renewcommand{\arraystretch}{1.1}
\begin{tabular}{c|*{11}{c}}
$n$ & $\leq 7$ & $8$ & $9$ & $10$ & $11$ & $12$ & $13$ & $14$ & $15$ & $16$ & $17$ \\
$|{\rm X}_n|$ & $1$ & $2$ & $2$ & $2$ & $2$ & $3$ & $3$ & $4$ & $5$ & $8$ & $9$ \\
\hline
$n$ & $18$ & $19$ & $20$ & $21$ & $22$ & $23$ & $24$ & $25$ & $26$ & $27$ & $28$ \\
$|{\rm X}_n|$ & $13$ & $16$ & $28$ & $40$ & $68$ & $117$ & $297$ & $665$ & $2\,566$ & $17\,059$ & $374\,062$ \\
\end{tabular}
\caption{Size of ${\rm X}_n$ for $n \leq 28$.}
\label{tab:tabcardXn}
\end{table}

These classification results for ${\rm X}_n$ with $n\leq 25$ 
were obtained through a long series of works in the last century 
\cite{mordell,ko,witt,kneser,niemeier,conwaysloaneuni,borcherdsthesis}: 
see \cite{conwaysloane} for some historical perspectives. 
Notable tools include Kneser's neighboring method introduced in \cite{kneser} 
(to settle the cases $10 \leq n \leq 16$), 
root lattices and the gluing method \cite{witt,niemeier,conwaysloaneuni}, 
and the use of Lorentzian lattices in \cite{borcherdsthesis,borcherdsduke} in the case $24 \leq n \leq 25$. 
We also mention that if ${\rm X}_n^R \subset {\rm X}_n$ denotes the subset of lattices 
with root system isomorphic to $R$ (and say, with no norm $1$ vectors), King computed the mass of 
${\rm X}_n^R$ for all $R$ and $n\leq 30$ in \cite{king},  substantially improving the mass formula lower bound 
on $|{\rm X}_n|$ for $26 \leq n \leq 30$.\ps

The remaining cases in Table~\ref{tab:tabcardXn} were obtained in the recent series of works 
\cite{chestat,cheuni} ($26 \leq n \leq 27$) and \cite{cheuni2} ($n=28$), 
using a biased enumeration of the Kneser neighbors of ${\rm I}_n$ 
(developing ideas in \cite{bachervenkov}), 
King's aforementioned work, 
refinements of the Plesken-Souvignier algorithm \cite{pleskensouvignier}, 
and substantial computer calculations. A remarkable, 
though not yet understood, byproduct of \cite{cheuni2} 
is the fact that a certain invariant of vectors of norm $\leq 3$, 
inspired by \cite{bachervenkov} and which we denote by ${\rm BV}$, 
is both fast to compute and  
distinguishes all unimodular lattices of rank $\leq 28$. 
Our first main result is:

\begin{thmintro}
\label{thmintro:uni29}
There are exactly $38\,966\,352$ classes in ${\rm X}_{29}$, 
all distinguished by their {\rm BV} invariants, 
and with Gram matrices given in {\rm \cite{chetaiweb}}. 
\end{thmintro}

King's lower bound $|{\rm X}_{29}| \geq 37\,938\,009$ was thus remarkably close.
Our method to prove Theorem~\ref{thmintro:uni29} is different from those above.
Indeed, the neighbor enumeration and analysis in \cite{cheuni2} for the case $n=28$ 
already required more than $70$ years of CPU time on a single core, 
and would not be reasonable in the much larger case $n=29$. 
The methods in \cite{cheuni,cheuni2} allow in principle to determine ${\rm X}_{29}^R$ for each $R$
(assuming the choice of invariant is fine enough, a serious assumption), 
but each choice of $R$ requires some specific handlings, 
and there are $11 085$ contributing $R$ by \cite{king}.
The set ${\rm X}_{29}^\emptyset$ was actually already determined in \cite{cheuni2}, 
and we use here the same strategy only in the cases 
$R={\bf A}_1$ and $R={\bf A}_2$ (see Proposition~\ref{prop:X29emptyA1A2}). 
\ps

All the remaining lattices contain a pair of orthogonal roots,
and for those we use an entirely different strategy, 
that we develop in Sect. \ref{sect:genunimodwithpairroots}.
It is based on the elementary fact that for any $n\geq 1$, 
there is a natural groupoid equivalence between:
(i) pairs $(L,e)$ with $L$ a rank $n$ unimodular lattice 
and $e \in L/2L$ satisfying $e \cdot e \equiv 2 \bmod 4$, and
(ii) pairs $(U,\{\alpha,\beta\})$ with $U$ a rank $n+2$ unimodular lattice 
and $\{\alpha,\beta\}$ an unordered pair of orthogonal roots of $U$ with $\frac{\alpha+\beta}{2} \notin U$. 
This paves the way for a recursive exhaustion of ${\rm X}_{n+2}$, 
by listing first the orbits of mod $2$ vectors of each element in ${\rm X}_n$. \ps

The main drawback of this method is that it produces 
each rank $n+2$ unimodular lattice $U$ as many times 
as the number of ${\rm O}(U)$-orbits of pairs of orthogonal roots in $U$: 
see Sect.~\ref{sect:genunimodwithpairroots} 
for a few simple techniques to reduce these redundancies. 
The worst case occurs when the root system of a generic $U$ 
is of the form $a {\bf A}_1\, b {\bf A}_2 $ with large $a+b$, and 
unfortunately these lattices constitute a significant proportion 
of the cases in practice when $n$ grows (presumably, 
a manifestation of the ``law of conservation of complexity''). \ps

Nevertheless, despite this issue, we realized that this method, 
combined with {\it unimodular hunting} for empty, ${\bf A}_1$ or ${\bf A}_2$ root systems, 
is by far the most efficient strategy 
to recursively reconstruct all unimodular lattices from scratch. 
For instance, using our current algorithms, it allowed us to recover 
all unimodular lattices of rank $\leq 28$ in only about a week of computations. 
Of course, such computations are always much easier {\it a posteriori}, 
but the superiority of this method over the neighbor method 
(which requires significantly more redundancies) is clear. \ps

In Sect.~\ref{sect:classrk29} we put the theory into practice 
and use this method to prove Theorem~\ref{thmintro:uni29}. 
As in rank $\leq 28$, the ${\rm BV}$ invariant turned out to miraculously 
distinguish all elements of ${\rm X}_{29}$, and 
this is a fundamental ingredient in the proof. 
We also refer to this section for detailed information about the final list 
of rank $29$ unimodular lattices 
and for further details about the computation process. \ps

Our computations were performed 
using the open-source computer algebra system Pari/GP \cite{parigp}. 
For efficiency, we were led to improve or reimplement several key algorithms. 
These important algorithmic contributions, due to the second author, 
are briefly described in Section~\ref{subsect:algolat} 
and will be the subject of a separate paper. 
They include: 
an exact implementation of the Fincke-Pohst algorithm \cite{FP}, 
an improvement of the Plesken-Souvignier algorithm for lattices with roots \cite{taibiateliergp}, 
a faster implementation of the ${\rm BV}$ invariant, 
a probabilistic algorithm for finding ``good'' Gram matrices (improving the one in \cite{cheuni2}), 
and a specific algorithm computing orbits of mod $2$ vectors. 
The source codes are, however, already available in \cite{chetaiweb}. \ps

The entire computation took about $20$ months of CPU time (single core equivalent). 
It was run on a system with $2 \times 64$-core AMD EPYC 7763 CPU (zen3) 
running at 1.5 GHz, with 1024 GB of RAM.\footnote{For reference, 
all CPU times reported in this paper are relative to this system.
While this system allowed efficient parallel computations, 
the algorithms described are equally performant on standard personal computers. }  
The given Gram matrices in \cite{chetaiweb}, together with the invariant BV and the mass formula, allow for an 
independent check that our list is complete, and which only requires about $80$ days of CPU time: see Sect. \ref{sect:classrk29}.
\ps

We conclude this section with a discussion of the case $n= 30$.
King's lower bound indicates that ${\rm X}_{30}$ has more than $20$ billion of classes!
Using the unimodular hunting techniques \cite{cheuni, cheuni2}, 
and the improvements above of our algorithms,
we were able to show the following (see Sect.~\ref{sect:rank30}).

\begin{thmintro}
\label{thmintro:X30}
The size of ${\rm X}_{30}^R$ for $R=\emptyset$, ${\bf A}_1$ or ${\bf A}_2$ is given by the table below.
Neighbor forms for representatives for all of those lattices are given in {\rm \cite{chetaiweb}}. Moreover,
all those lattices are distinguished by their ${\rm BV}$ invariant.
\end{thmintro}

\vspace{-.7cm}
\begin{table}[H]
\tabcolsep=8pt
{\small \renewcommand{\arraystretch}{1.3} \medskip
\begin{center}
\begin{tabular}{c|c|c|c}
$R$ & $\emptyset$ & ${\bf A}_1$ & ${\bf A}_2$ \\
\hline
$|{\rm X}_{30}^R|$ & $82\,323\,107$ & $357\,495\,297$ &  $12\,708\,298$  \\
\end{tabular}
\end{center}
}
\end{table}

Although we have no doubt that we could go further 
and determine ${\rm X}_{30}^R$ for many other root systems $R$,
we do not pursue this direction here. 
Indeed, it seems more promising 
to directly study the more fundamental (and famous) genus
${\rm X}_{32}^{\rm even}$ of rank $32$ even unimodular lattices, 
which should contain ``only'' about $1.2$ billion lattices according to King, 
and from which ${\rm X}_{30}$ can be theoretically deduced. 
\ps

The results established in this section have several interesting consequences 
for the study of ${\rm X}_{32}^{\rm even}$, such as a classification of 
all lattices containing ${\bf A}_3$. To the best of our knowledge, they provide 
the first indication in the literature that a classification of 
${\rm X}_{32}^{\rm even}$ may now be within reach, 
the sizes of ${\rm X}_{29}$ and of the 
${\rm X}_{30}^R$ above being of the same order of magnitude as the expected 
size of ${\rm X}_{32}^{\rm even}$. A key remaining difficulty 
in studying ${\rm X}_{32}^{\rm even}$
is that the natural analogue of the ${\rm BV}$ invariant 
in this case is computationally much more expensive
due to the large number of norm $4$ vectors.\ps

We will study ${\rm X}_{32}^{\rm even}$ 
more thoroughly in a forthcoming work. 
We stress that the study of such large genera is not as futile as it may seem, 
and is not solely motivated by the computational challenge. Indeed,
as shown in a series of recent works by the authors 
(such as \cite{chlannes,chrenard,taibisiegel,mot23,chetai_cuspcohGL}), 
it would also have significant consequences for 
the theory of automorphic forms for ${\rm GL}_n(\Z)$: 
see~Sect.~\ref{subsect:intromot} 
for more information about these aspects.

\subsection{Even lattices of small and prime (half-)determinant $p$}
\label{subsect:hnp}
	For an integer $n\geq 1$ and $p$ an odd prime, we consider the 
genus $\mathcal{G}_{n,p}$ of even Euclidean lattices of rank $n$ 
and determinant $d$, with $d=p$ for $n$ even and $d=2p$ for $n$ odd. 
See Section~\ref{sect:hdiscp} for examples and basic properties of these genera.
The genus $\mathcal{G}_{n,p}$ is nonempty if and only if 
either $n$ is odd or $n+p \equiv 1 \bmod 4$. 
Our main result is the following.

\begin{thmintro}
\label{thmintro:hdiscptable}
Each $(p,n)$-entry given in the Tables~\ref{tab:gnpneven} and~\ref{tab:gnpnodd} 
provides the number of isometry classes in the genus $\mathcal{G}_{n,p}$. 
Moreover, Gram matrices for representatives of these isometry classes are given\footnote{
We did not store the (huge) list in the case $(p,n)=(7,27)$.} 
in~{\rm \cite{chetaiweb}}.
\end{thmintro}
\vspace{-.3cm}
\begin{table}[H]
\tabcolsep=6pt
{\scriptsize \renewcommand{\arraystretch}{1.7} \medskip
\begin{center}
\begin{tabular}{c|c|c|c|c|c|c|c|c|c|c|c|c|c|c}
$p \, \backslash \, n $ & $2$ &  $4$  & $6$ & $ 8$ & $10$ & $12$ & $14$ & $16$ & $18$ & $20$ & $22$ &  $24$ & $26$ & $28$ \\
\hline
$3$ & $1$ & & $1$ & & $1$ & & $2$ & & $6$ & & $31$ & & ${\color{red}678}$ & \\
\hline
$5$ & & $1$ & & $1$ & & $2$ & & $5$ & & $27$ & & ${\color{red}352}$ & & {\color{red} 2\,738\,211} \\
\hline
$7$ & $1$ & & $1$ & & $2$ & & $4$ & & ${\color{red}20}$ & & ${\color{red} 153}$ & &{\color{red} 44\,955}  & \\
\end{tabular}
\end{center}
}
\vspace{-5 mm}
\caption{\small Number of isometry classes of even lattices of even rank $2 \leq n \leq 28 $ and odd prime determinant $p\leq 7$ (zero if blank).}
\label{tab:gnpneven}
\end{table}
\vspace{-.6cm}
\begin{table}[H]
\tabcolsep=5.5pt
{\scriptsize \renewcommand{\arraystretch}{1.7} \medskip
\begin{center}
\begin{tabular}{c|c|c|c|c|c|c|c|c|c|c|c|c|c|c}
$p \, \backslash \, n $ & $1$ &  $3$  & $5$ & $7$ & $9$ & $11$ & $13$ & $15$ & $17$ & $19$ & $21$ &  $23$ & $25$ & $27$ \\
\hline
$3$ & $1$ & $1$ & $1$ & $1$ & $2$ & $2$ & $3$ & $5$ & $10$ & $19$ & $64$ & ${\color{red} 290}$ & ${\color{red} 2\,827}$ & {\color{red} 285\,825}\\
\hline
$5$ & $1$ & $1$ & $1$ & $1$ & $3$ & $3$ & $5$ & $10$ & $21$ & ${\color{red} 55}$ & ${\color{red} 210}$ & ${\color{red} 1\,396}$ & ${\color{red} 38\,749}$ &${\color{red} 24\, 545\,511}$  \\
\hline
$7$ & $1$ & $1$ & $1$ & $2$ & $3$ & $5$ & $8$ & $14$ & ${\color{red} 37}$ & ${\color{red}119}$ & ${\color{red}513}$ &${\color{red} 5\,535}$ & ${\color{red} 341\,798}$ & ${\color{magenta} 659\,641\,434}$ \\
\end{tabular}
\end{center}
}
\vspace{-5 mm}
\caption{\small
Number of isometry classes of even lattices of odd rank $1 \leq n \leq 27$ and determinant $2p$ with $p$ a prime $\leq 7$.}
\label{tab:gnpnodd}
\end{table}

The results for $p=3$ and even $n \leq 14$ in these tables 
go back to Kneser~\cite{kneser}.
Kneser deduced them from his aforementioned classification of 
unimodular lattices of rank $\leq 16$ by studying 
embeddings into unimodular lattices of slightly larger rank.\footnote{
Actually, this idea had already been used by Gauss 
in the last section of his {\it Disquisitiones}, 
in which he relates the representations 
of an integer $n$ as a sum of three squares 
to the isometry classes in a certain genus of 
integral binary forms of determinant $n$.}
Applying similar ideas, along with 
a more systematic use of root lattices and of the gluing method, 
Conway and Sloane determined in \cite{conwaysloanelowdim} 
all the values in the tables above which are $\leq 2$, or with $n\leq 10$.
We also mention the almost complete\footnote{
The few indeterminacies for $(26,3)$ in Borcherds's list 
were settled in \cite{megarbane} (see \cite[\S 9]{cheuni} for a different approach).
}
study in \cite[\S 4.7]{borcherdsduke} of the cases $(n,p)=(25,3)$ and $(26,3)$
by a quite different approach using Lorentzian lattices.\ps

We will prove Theorem~\ref{thmintro:hdiscptable} in Sect.~\ref{subsect:orbitmethod}.
Our method, in the spirit of Kneser's, consists in deducing 
all the genera $\mathcal{G}_{n,p}$ above 
from the classification of unimodular lattices of rank $\leq 29$ and 
successive applications of {\it the orbit method}, that we now discuss.\ps
  
We say that a genus $\mathcal{G}'$ may be deduced 
from another genus $\mathcal{G}$ by {\it the orbit method of type $t$} 
if the map $(L,v) \mapsto L':=v^\perp \cap L$ induces a bijection between 
the isometry classes of pairs $(L,v)$, with $L$ in $\mathcal{G}$ and $v \in L$
of type $t$, and the isometry classes of $L'$ in $\mathcal{G}'$. 
{\it Type $t$ vectors} will typically be the set of 
all primitive vectors of a certain norm satisfying possibly some extra conditions. 
Several natural instances of such triples $(\mathcal{G},\mathcal{G'},t)$
are recalled in Sect.~\ref{subsect:gluing} (see also Proposition \ref{prop:d1d2} for 
a possibly new example). 
If $\mathcal{L}$ is a set of representatives for the isometry classes in $\mathcal{G}$,
finding representatives of the isometry classes in $\mathcal{G}'$ then amounts 
to determining, for each $L \in \mathcal{L}$, 
the ${\rm O}(L)$-orbits of vectors of type $t$ in $L$, hence the terminology.
If the number of type $t$ vectors in these lattices $L$ is not too large,\footnote{
This holds for all cases but two cases: see the end of Sect.~\ref{subsect:orbitmethod}.} and 
if the Plesken-Souvignier algorithm succeeds in finding generators for ${\rm O}(L)$, 
this orbit computation is a simple task for a computer.
\ps

The proof of Theorem~\ref{thmintro:hdiscptable} is fairly long and intricate.
An overview of the techniques we used is 
shown in Figure~\ref{fig:imageliensgenres} below.
In this figure, even genera are pictured in blue, 
and odd ones are in yellow.\footnote{
We focus in this figure on the most important cases $n\geq 23$. 
The cases $n\leq 22$ are only easier: see the information in \cite{chetaiweb}.
}
\begin{figure}[htp]
\centering
\vspace{-3cm}
\includegraphics[scale=1.3]{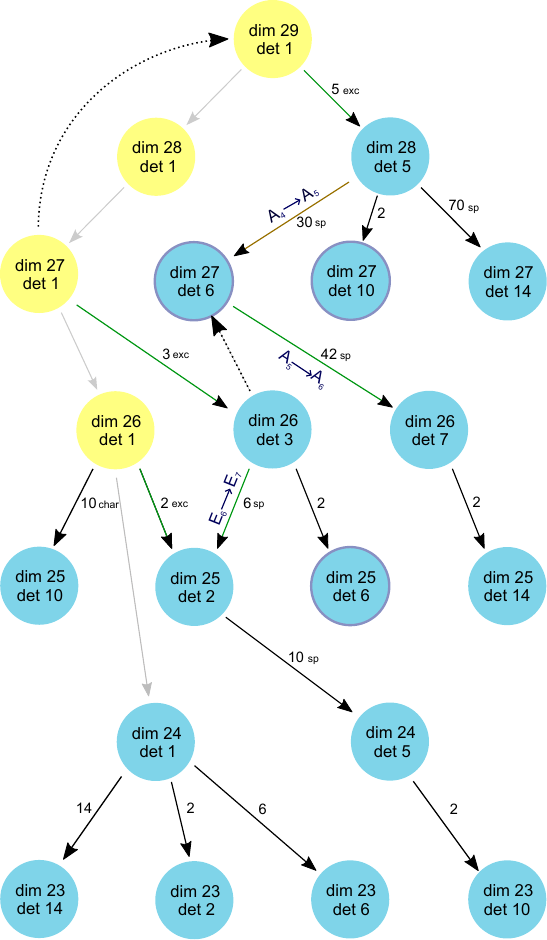}
\caption{Unimodular lattices rule!\\
\small ({\it Leitfaden} for the proof of Theorem~\ref{thmintro:hdiscptable})}
\label{fig:imageliensgenres}
\end{figure}
The presence of an arrow $\mathcal{G} \overset{t}{\rightarrow} \mathcal{G}'$ 
means that $\mathcal{G}'$ can be deduced from $\mathcal{G}$ 
by the orbit method of type $t$. An important advantage 
of the orbit method is that it is unnecessary 
to provide any sharp isometry invariant for the lattices 
in the target genus to produce each isometry class once and only once.
Nevertheless, we found it highly desirable to have such an invariant 
(that is fast to compute), at least for our larger genera. 
For instance, having such invariants allows one to directly check that 
the lattices given in our lists are pairwise non-isomorphic, 
hence form a complete list of representatives,
by applying the known mass formula for 
$\mathcal{G}_{n,p}$ \cite{conwaysloanemass}, 
and the Plesken-Souvignier algorithm. 
It is also useful for potential Hecke operator computations.
This is achieved by the following result.

\begin{thmintro}
\label{thmintro:hdiscptablebvp}
For each $(p,n)$-entry in black in the Tables~\ref{tab:gnpneven} and~\ref{tab:gnpnodd},
the isometry classes in $\mathcal{G}_{n,p}$ are distinguished by their root systems. 
For the entries in red,\footnote{
The case $(p,n)=(7,27)$ could presumably be treated as well, but we did not try.} 
they are distinguished by their ${\rm BV}_{n,p}$-invariant.
\end{thmintro}

We refer to Sect.~\ref{subsect:checkhdiscp} for the proof, and
a (rather case-by-case) definition of ${\rm BV}_{n,p}$. 
These invariants are special cases of new invariants that 
we introduce in Sect.~\ref{subsect:markedBV} and that we call 
the {\it marked} BV {\it invariants}.
The marked ${\rm BV}$ invariants are defined more generally 
for arbitrary isometry classes of pairs $(L,\iota)$, 
where $L$ is a lattice and $\iota$ is an embedding of a fixed lattice $A$ into $L$; 
they refine the ${\rm BV}$ invariant of $L$ by 
taking into account the embedding $\iota$. 
Our strategy is then to identify the groupoid $\mathcal{G}_{n,p}$
with a suitable groupoid of such pairs 
(there are many different ways to do so using gluing constructions), 
and define the invariant of a lattice in $\mathcal{G}_{n,p}$ 
as the marked BV invariant of the corresponding pair. 
The amazing efficiency of these invariants, 
demonstrated by the tens of millions of lattices they distinguish,
is an unexplained miracle to us.\ps

Going back to the statement of Theorem~\ref{thmintro:hdiscptable}, 
we mention that for each lattice in our lists in \cite{chetaiweb}, 
we provide not only a (good) Gram matrix,
but also its root system,
its (reduced) mass, 
generators of its reduced isometry group,
and for red $(p,n)$-entries,
its (hashed) ${\rm BV}_{n,p}$ invariant.
It is then easy to check that our lists are complete, thanks to the mass formula.

\subsection{Unique orbit property for ''exceptional'' vectors}
\label{subsect:uniqueorbintro}
During the numerical applications of 
the orbit method described in \S~\ref{subsect:hnp},
and strikingly often, we observed that the set of vectors of 
a certain type in certain lattices $L$ (or their dual $L^\sharp$) 
forms a single ${\rm O}(L)$-orbit. 
{\it This applies to the set of type $t$ vectors of each $L \in \mathcal{G}$ 
for each green arrow $\mathcal{G} \overset{t}{\rightarrow} \mathcal{G}'$ 
in Figure~\ref{fig:imageliensgenres}}, making the orbit method 
particularly straightforward to implement in these cases. \ps

We naturally tried to find conceptual reasons for this: 
this is the topic of Sect.~\ref{sect:fertile} and Sect.~\ref{sect:excuni}. 
Here is a striking example. Let $L$ be a unimodular lattice. 
Following \cite{bachervenkov}, we denote by ${\rm Exc}\, L$ 
the set of characteristic vectors $v$ of $L$ with norm $v \cdot v <8$ 
(the {\it exceptional} vectors), and call $L$ {\it exceptional} 
if ${\rm Exc}\, L$ is non-empty. The following result is proved 
in Sect.~\ref{sect:excuni}. 

\begin{thmintro} 
\label{thm:excsingleorbituni}
Assume $L$ is an exceptional unimodular lattice of rank $n$ with 
$n \not\equiv 6, 7 \bmod 8$. Then ${\rm O}(L)$ acts transitively on ${\rm Exc}\, L$.
\end{thmintro}

This explains the green arrows of type \texttt{exc} in Figure~\ref{fig:imageliensgenres}.
Known simple descriptions of ${\rm Exc}\, L$ 
for $n \equiv 0,1,2,3,4 \bmod 8$ actually make the theorem 
straightforward to prove in these cases. 
However, the case $n \equiv 5 \bmod 8$, 
of importance here since $29 \equiv 5 \bmod 8$, 
is much harder. 
In this case we prove a stronger result (Theorem~\ref{thm:uniqueorbex}): 
{\it if furthermore $L$ has no norm $1$ vector, 
then we have $|{\rm Exc}\,L| \leq 2n$ and 
the subgroup of ${\rm O}(L)$ generated by $-1$ 
and the Weyl group of $L$ acts transitively on ${\rm Exc}\,L$}. 
For instance, we have $|{\rm Exc}\,L|=2$ if $L$ has no root, 
a previously unexplained observation in \cite{cheuni2}.
Note that the statement of Theorem~\ref{thm:excsingleorbituni} 
does not hold for $n \equiv 6,7 \bmod 8$ 
(but see Remark~\ref{rem:dim6mod8}). \ps

	A general framework in which the number of orbits 
can be controlled is introduced in Sect.~\ref{sect:fertile}. 
Fix a root system $R$, with associated root lattice ${\rm Q}(R)$.
We are interested in an even lattice $M$ in a genus that is {\it opposite} 
to that of ${\rm Q}(R)$, in the sense that there is an isometry 
of finite quadratic spaces
$$\eta : M^\sharp/M \overset{\sim}{\longrightarrow}\,- \, {\rm Q}(R)^\sharp/{\rm Q}(R).$$ 
We say that a class $c$ in ${\rm Q}(R)^\sharp/{\rm Q}(R)$ is {\it fertile} 
if the minimum $\nu(c)$ of $\xi \cdot \xi$, over all $\xi \in {\rm Q}(R)^\sharp$ 
in the class $c$, is $<2$. To each fertile $c$ is attached a canonical root system 
$R_c \supset R$, whose Dynkin diagram is obtained from that of $R$ by adding 
a single node and some edges; each possible such diagram arises 
for one and only one fertile $c$.
Fix a fertile class $c$, as well as $(M,\eta)$ 
as above. We are interested in the following set of (short) vectors in $M^\sharp$
$${\rm Exc}_{c,\eta}\, M \,\,=\,\,\{ v \in M^\sharp\, \, \mid\, \, \eta(v)=-c\,\,{\rm and}\, \, v\cdot v=2- \nu(c)\},$$
that we call the {\it $(c,\eta)$-exceptional vectors} of $M$.\footnote{This notion refines and generalizes that of exceptional vectors in odd unimodular lattices: see Remark \ref{rem:rel_def_exc}.}
A key fact is that these vectors are in natural bijection with the extensions to ${\rm Q}(R_c)$ 
of the natural embedding ${\rm Q}(R) \rightarrow U$, 
where $U \supset M \perp {\rm Q}(R)$ denotes 
the even unimodular lattice associated to the pair $(M,\eta)$ 
by the gluing construction (Proposition~\ref{prop:orbfibreMc}). 
Studying the ${\rm W}(M)$-orbits of $(c,\eta)$-exceptional vectors of $M$ then becomes 
equivalent to studying Weyl group orbits of 
embeddings of root systems into each other, a combinatorial problem
which in favorable cases leads to unique orbit properties.\ps

The proof of Theorem~\ref{thm:excsingleorbituni} in the case $n\equiv 5 \bmod 8$, 
consists in applying these ideas to the even part $M$ of the unimodular lattice $L$ 
with $R \simeq {\bf A}_3$ and to the fertile classes $c$ leading to 
${\rm R}_c \simeq {\bf A}_4$ or ${\bf D}_4$. 
Note that the classification of exceptional unimodular lattices of rank $29$ easily follows from 
Theorem~\ref{thmintro:uni29}: see Theorem~\ref{thm:excuni209} for a detailed study. 
By Theorem~\ref{thm:excsingleorbituni} the number of such lattices is exactly the same as the number of isometry classes in the genus $\mathcal{G}_{28,5}$, explaining the important $(5,28)$-entry in Table~\ref{tab:gnpneven}: see Remark~\ref{ex:g285}. \ps

A second consequence of the general theory above is a
relationship between the genera opposite to the genus of ${\rm Q}(R)$ and 
those opposite to the genus of ${\rm Q}(R_c)$, 
for each fertile class $c$ of $R$ (Corollary~\ref{cor:MRcMRc}). 
When an arrow $\mathcal{G} \overset{t}{\rightarrow} \mathcal{G}'$ 
in Figure~\ref{fig:imageliensgenres} is a special case of this relationship, 
we add the label $R \rightarrow R_c$ to it.
Here is an example occurring twice in the figure. 
For all $n\geq 1$ the pair $({\bf A}_n,{\bf A}_{n+1})$ occurs as some $(R,R_c)$, 
so we deduce a groupoid equivalence between: \par
(i) pairs $(M,e)$, with $M$ an even lattice of determinant $n+1$ 
and $e \in M^\sharp$ a primitive vector satisfying $e\cdot e = \frac{n+2}{n+1}$,  \par
(ii) pairs $(N,w)$, with $N$ an even lattice of determinant $n+2$
and $w$ a generator of the group $N^\sharp/N$ satisfying $w\cdot w \equiv -\frac{n+1}{n+2} \bmod 2\Z$. \par
\noindent A generalization of this statement is given in Proposition~\ref{prop:d1d2}. 
We have not been able to locate such statements elsewhere in the literature. \ps

The unique orbit property for the pair $(R,R_c)=({\bf A}_n,{\bf A}_{n+1})$ is discussed in Proposition~\ref{prop:keyuniqueorbitAn}. 
It explains the green arrow ${\bf A}_5 \rightarrow {\bf A}_6$ in Figure~\ref{fig:imageliensgenres}, as well as the brown arrow ${\bf A}_4 \rightarrow {\bf A}_5$: 
the unique orbit property holds in most cases, with an understood set of exceptions. 
We have essentially explained so far all the decorations in Figure~\ref{fig:imageliensgenres},
except for the purple circles!
\ps

\subsection{Motivations and perspectives}
\label{subsect:intromot}

${}^{}$ As discussed throughout this introduction,
the practical question of classifying integral Euclidean lattices 
of fixed dimension and determinant is a venerable one, 
whose historical developments have generally reflected theoretical and algorithmic 
advances in the study of Euclidean lattices. It is remarkable that 
the sizes of the genera classified in Sect.~\ref{sect:introunimod} and~\ref{subsect:hnp} 
are several orders of magnitude larger than those previously accessible.
However, the motivations for studying such large genera might seem limited to 
the encyclopedic aspect and the computational challenge.
Indeed, while the very first cases historically considered (say, in dimension at most $4$)
were linked to very simple and natural arithmetic questions, 
this can hardly be said to remain true in high dimensions. \ps

Our main motivation for this work is actually quite different;
it stems from the forthcoming companion paper~\cite{chetai_cuspcohGL}, 
in which we study automorphic forms and the rational cohomology of the group 
${\rm GL}_n(\Z)$, building upon a series of recent works 
by the authors~\cite{chlannes,chrenard,taibisiegel,mot23}.
As will be explained in that paper, 
new phenomena occur around dimension $n=27$, 
but uncovering them requires 
the evaluation of certain local orbital integrals for orthogonal groups.
While a direct computation of these integrals seems out of reach, 
they can be determined via a local-global method 
using the characteristic masses (in the sense of~\cite{cheniemeier})
of sufficiently many genera of even lattices 
of dimension $\leq 27$ and prime (half-)determinant. \ps

The specific genera needed for our computations in~\cite{chetai_cuspcohGL} 
are the $\mathcal{G}_{n,p}$ 
for $(n,p)$ in $\{ (25,3), (27,3), (27,5)\}$; 
they are circled in purple in Figure~\ref{fig:imageliensgenres}. 
This explains why, in the proofs of 
Theorems~\ref{thmintro:hdiscptable} and~\ref{thmintro:hdiscptablebvp} 
in Sect.~\ref{sect:hdiscp}, we examine these genera in greater detail 
and sometimes provide alternative classification methods, 
given their primary importance for our applications.
We finally mention that Theorem~\ref{thmintro:uni29} and the methods above 
also enable the determination of $\mathcal{G}_{n,p}$ for many 
other pairs $(n,p)$ (with $p\geq 11$ and $n\leq 27$), but we omit the discussion of these cases here.

\ps \medskip
{\sc Acknowledgments.} 
Experiments presented in this work were carried out using the cluster cinaps 
of the LMO (Universit\'e Paris-Saclay) and the cluster ``Cascade'' of the PSMN (Pôle Scientifique de Modélisation Numérique, ENS de Lyon).
We warmly thank the LMO and PSMN for sharing their machines.

{\small
\setcounter{tocdepth}{1}
\tableofcontents
}

\section{General preliminaries on lattices}
\label{sect:notation}

\subsection{Basic notations, conventions, and terminology}
\label{subsect:notation}
${}^{}$
	If $V$ is a Euclidean space, we usually denote by $x \cdot y$ its inner product.
A {\it lattice} in $V$ is a subgroup $L$ generated by a basis of $V$.
The {\it dual} of $L$ is the lattice $L^\sharp :=\{ v \in V\, \, |\, \, w \cdot v \in \Z ,\,\, \, \forall w \in L\}$.
We say that $L$ is {\it integral} if we have $L \subset L^\sharp$.
It is sometimes better to have a non embedded notion of an Euclidean integral lattice, 
and to think of them as a an abstract abelian group of finite rank $L$ equipped with
a positive definite symmetric $\Z$-bilinear form $L \times L \rightarrow \Z$, $(x,y) \mapsto x.y$, 
the ambient Euclidean space being then $V:=L \otimes_\Z \R$. 
We freely use both points of view. {\bf From now on, unless explicitly stated, 
the term {\it lattice} will always mean {\it integral Euclidean lattice}.}\ps

Let $L$ be a lattice. The {\it norm} of a vector $v \in V$ is $v \cdot v$.
The lattice $L$ is called {\it even} if $v.v \in 2\Z$ for all $v \in L$, and {\it odd} otherwise.
The {\it determinant} of $L$ is the determinant of the {\it Gram matrix} 
${\rm Gram} (e) = (e_i \cdot e_j)_{1 \leq i,j \leq n}$ of any $\Z$-basis 
$e=(e_1,\dots,e_n)$ of $L$. It is denoted $\det L$. We have 
$\det L \in \Z_{\geq 1}$, and we say that $L$ is unimodular if $\det L = 1$.\ps

For $i\geq 0$, the {\it configuration of vectors of $L$ of norm $\leq i$} (resp. of norm $i$)
is the finite metric set
\begin{equation}\label{eq:defri}
{\rm R}_{\leq i}(L)=\{ v \in L\, |\, v \cdot v \leq i\}\,\,\,{\rm and}\, \, \, {\rm R}_{i}(L)=\{ v \in L\, |\, v \cdot v = i\}
\end{equation}
and we also set ${\rm r}_i(L) = |{\rm R}_{i}(L)|$. We denote by ${\rm I}_n$ the standard lattice $\Z^n$,
with $x \cdot y = \sum_{i=1}^n x_iy_i$.
Any lattice $L$ may be uniquely written as $L = L_0 \perp L_1$ 
with $L_0 \simeq {\rm I}_m$ and ${\rm r}_1(L_1)=0$, and we have ${\rm r}_1(L)=2m$.
A {\it root} in $L$ is a vector of norm $2$.
The {\it root system} of $L$ is ${\rm R}_2(L)$.

\begin{example} 
\label{Ex:rootlattices} 
{\rm (Root systems and lattices) }
In this paper, {\it a root system $R$}, in an Euclidean space $V$, 
will always be assumed to have all its roots of norm $2$, 
hence to be of ${\rm ADE}$ type {\rm (}or {\it simply laced}{\rm )}. 
We denote by ${\rm Q}(R)$ the even lattice it generates in $V$ {\rm (}{\it root lattices}{\rm )}. 
We use bold fonts ${\bf A}_n$, ${\bf D}_n$, ${\bf E}_n$ to denote
isomorphism classes of root systems of those names, 
with the conventions ${\bf A}_0={\bf D}_0={\bf D}_1=\emptyset$ and ${\bf D}_2= 2{\bf A}_1$.
We also denote by ${\rm A}_n$, ${\rm D}_n$ and ${\rm E}_n$ the standard corresponding root lattices.
\end{example}

Contrary to ${\rm R}_{\leq 2}(L)$, there seems to be no known classification of the
configurations of vectors of the form ${\rm R}_{\leq 3}(L)$.
We refer to~\cite[\S 4.3]{cheuni} for a discussion of this problem, 
and to Sect.~\ref{subsect:markedBV} below for some important invariants that we shall use.
\ps

The {\it isometry group} of the lattice $L$ is ${\rm O}(L)=\{ \gamma \in {\rm O}(V), \, \, \gamma(L)=L\}$ (a finite group).
The {\it Weyl group} of $L$ is the subgroup ${\rm W}(L) \subset {\rm O}(L)$ 
generated by the orthogonal reflections ${\rm s}_\alpha$ 
about each $\alpha \in {\rm R}_2(L)$. This is a normal subgroup, 
isomorphic to the classical Weyl group ${\rm W}(R)$ of $R:={\rm R}_2(L)$.
We define the {\it reduced isometry group} of $L$ to be 
${\rm O}(L)^{\rm red}:={\rm O}(L)/{\rm W}(L)$.
Let $\rho$ be the Weyl vector of a positive root system in ${\rm R}_2(L)$.
The stabilizer ${\rm O}(L;\rho)$ of $\rho$ in ${\rm O}(L)$ 
satisfies ${\rm O}(L) = {\rm W}(L) \rtimes {\rm O}(L;\rho)$, 
and in particular, is canonically isomorphic to ${\rm O}(L)^{\rm red}$. 
If $w \in {\rm W}(L)$ denotes the unique element satisfying $w(\rho)=-\rho$, 
then we have $-w \in {\rm O}(L;\rho)$. This element is nontrivial, hence has 
order $2$, unless we have $-1 \in {\rm W}(L)$. We also denote by ${\rm W}(L)^{\pm}$ 
the subgroup of ${\rm O}(L)$ generated by $-{\rm 1}$ and ${\rm W}(L)$.

\begin{lemma}
\label{lem:minus1wl} 
Let $L$ be a lattice such that $-1 \in {\rm W}(L)$. 
Then ${\rm R}_2(L)$ has the same rank as $L$,
and the abelian group $L^\sharp/L$ is killed by $2$.
\end{lemma}

\begin{pf} 
The first assertion is clear. 
For the second, note that for $\alpha \in {\rm R}_2(L)$ and $x \in L^\sharp$, 
we have ${\rm s}_\alpha(x) =x- \,(\alpha \cdot x) \alpha \equiv x \bmod L$, 
so ${\rm W}(L)$ acts trivially on $L^\sharp/L$.
\end{pf}

\begin{lemma}
\label{lem:surjWpm}
Let $U$ be a lattice and $V \subset U$ a sublattice.
Let ${\rm O}(U,V)$ be the stabilizer of $V$ in ${\rm O}(U)$, 
and $r : {\rm O}(U,V) \rightarrow {\rm O}(V)$ the restriction morphism.
We have ${\rm W}(V) \subset r({\rm W}(U)\cap {\rm O}(U,V))$ 
and ${\rm W}(V)^{\pm} \subset r({\rm W}(U)^{\pm} \cap {\rm O}(U,V)) $. 
\end{lemma}

\begin{pf} For any root $\alpha$ of $V$ (hence of $U$),
we have ${\rm s}_\alpha \in  {\rm W}(U) \cap {\rm O}(U,V)$ and $r({\rm s}_{\alpha})={\rm s}_\alpha$. 
We conclude since $r(-{\rm id}_U)=-{\rm id}_V$.
\end{pf}

\subsection{Groupoids and masses}
\label{subsect:groupoidsmasses}

${}^{}$ 
The language of {\it groupoids} is particularly well-suited to the theory of lattices,
providing both concise statements and useful points of view.
Recall that a groupoid is a category $\mathcal{G}$ whose morphisms are all isomorphisms.
Two groupoids are said to be {\it equivalent} if they are equivalent as categories.
When the isomorphism classes of objects in a groupoid $\mathcal{G}$ form a set, 
we denote 
it by ${\rm Cl}(\mathcal{G})$. An equivalence of groupoids $\mathcal{G} \rightarrow \mathcal{G'}$ 
induces in particular a natural bijection ${\rm Cl}(\mathcal{G}) \isomo {\rm Cl}(\mathcal{G}')$, whenever defined.
\ps
If $\mathcal{C}$ is any collection of lattices, 
for instance any genus, then $\mathcal{C}$ may be viewed as a groupoid 
whose morphisms are the lattice isometries.
As another example, the pairs $(L,v)$
with $L$ in $\mathcal{C}$ and $v \in L$, form a groupoid in a natural way: an isomorphism
$(L,v) \rightarrow (L',v')$ is a lattice isometry $f : L \rightarrow L'$ with $f(v)=v'$. 
We will consider many variants of these examples in what follows, the groupoid structure 
being usually obvious from the context.

\ps
Let $\mathcal{G}$ be a groupoid with finitely many isomorphism classes of objects, 
and such that each object in $\mathcal{G}$ has a finite automorphism group. 
The {\it mass} of such a $\mathcal{G}$ is 
the rational number ${\rm mass}\, \mathcal{G}
\, =\,\sum_{[x] \in {\rm Cl}(\mathcal{G})} \frac{1}{|{\rm Aut}_{\mathcal{G}}(x)|}.$\ps

	For instance, any collection $\mathcal{C}$ of lattices with bounded determinants 
has a mass, denoted by ${\rm mass}\,\mathcal{C}$. 
For such a $\mathcal{C}$ and any root system $R$, 
we also denote by $\mathcal{C}^R$ the full subgroupoid of lattices $L$ in $\mathcal{C}$ 
with ${\rm r}_1(L)=0$ and root system ${\rm R}_2(L) \simeq R$. 
The {\it reduced mass} of $\mathcal{C}^R$ is defined as 
$|{\rm W}(R)| \, {\rm mass}\, \mathcal{C}^R$. 
If $\mathcal{C}$ consists of the single lattice $L$, 
we simply write ${\rm mass}\, L=1/|{\rm O}(L)|$,
and the reduced mass of $L$ is $|{\rm W}(R)|/|{\rm O}(L)|=1/|{\rm O}(L)^{\rm red}|$
with $R={\rm R}_2(L)$.
	
\subsection{Residues and gluing constructions}
\label{subsect:gluing}
Let $L$ be a lattice. \ps

(R1) The finite abelian group $\resb L := L^\sharp/L$, 
sometimes called the {\it discriminant group} \cite{nikulin}, 
the {\it glue group} \cite{conwaysloane} or the {\it residue}\footnote{
We owe this pleasant notation and terminology to Jean Lannes.
Not only does this choice avoid the overused term {\it discriminant}, 
but it also evokes the construction of Milnor's residue maps 
${\rm W}(\Q) \rightarrow {\rm W}(\mathbb{F}_p)$ in the theory of Witt groups,
as described in the appendix in~\cite{bllv}. } \cite{chlannes}, 
is equipped with a non-degenerate $\Q/\Z$-valued symmetric bilinear form, 
defined by $(x,y) \mapsto x \cdot y \bmod \Z$.
We have $|\resb L| = \det L$.
Any isometry $\sigma : L \rightarrow L'$ between lattices induces 
an isometry of finite bilinear spaces $\resb \sigma :  \resb\,L \rightarrow \resb\,L'$.\ps

(R2) A subgroup $I \subset \resb L$ is called {\it isotropic} 
if we have $x \cdot y  \equiv 0$ for all $x,y \in I$, and a {\it Lagrangian} 
if we have furthermore $|I|^2=|\resb L|$. 
The map $\beta_L : M \mapsto M/L$ defines a bijection between 
the set of integral lattices containing $L$ and the set of isotropic subgroups of $\resb\, L$.
In this bijection we have \(|M/L|^2 \det M = \det L\), 
in particular $M/L$ is a Lagrangian of \(\resb\, L\) if and only if $M$ is unimodular. \ps

(R3) Assume that $L$ is even. 
Then the finite symmetric bilinear space $\resb\, L$ may be equipped
with a canonical quadratic form ${\rm q} : \resb L \rightarrow \Q/\Z$ 
such that ${\rm q}(x+y)-{\rm q}(x)-{\rm q}(y) \equiv x \cdot y$, 
defined by ${\rm q}(x) = \frac{x \cdot x}{2} \bmod \Z$. 
When we view $\resb L$ as equipped with such a structure, 
we rather denote it\footnote{This is important since on occasion we will really
want to consider $\resb\, L$ when $L$ is even, and not $\resq\,L$.}
by $\resq \,L$, for {\it quadratic residue}. In the bijection $\beta_L$ of (R2) above, 
the {\it even} lattices $M$ correspond to the {\it quadratic isotropic} 
subgroups $I \subset \resq L$, {\it i.e.} satisfying ${\rm q}(I)=0$.
\ps

We now state a form of the so-called {\it gluing construction}, 
which is essentially \cite[Prop. 1.5.1]{nikulin}.  
Recall that a subgroup $A$ of a lattice $L$ is called {\it saturated} 
if the abelian group $L/A$ is torsion free,
or equivalently, if $A$ is a direct summand of $L$ as $\Z$-module. 
If $X$ is a bilinear or quadratic space, $-X$ denotes 
the same space but with the opposite bilinear or quadratic form.\ps

\begin{prop}
\label{prop:nikulingroupoids}
Let $A$ be a lattice and $H \subset \resb \, A$ a subgroup 
equipped with the induced $\Q/\Z$-valued bilinear form. 
The following groupoids are naturally equivalent:
\begin{itemize}
\item[(i)] pairs $(L,\iota)$ with $L$ a lattice 
and $\iota : A \rightarrow L$ an isometric embedding with saturated image 
such that the natural map $L \rightarrow \resb \,A$ has image $H$.\ps
\item[(ii)] pairs $(B,\eta)$ with $B$ a lattice 
and $\eta : H \,\rightarrow \,-\resb \,B$ an isometric embedding.
\end{itemize}
In this equivalence, we have $B \simeq L \cap \iota(A)^\perp$ 
and $|H|^2 \,\det L =\, \det A \,\det B$. Assuming furthermore 
$H= \resb \,A$, we also have $\resb B \simeq  - \resb\,A \perp \resb\,L$.
\end{prop}

Let $\mathcal{A}$ and $\mathcal{B}$ be the two respective groupoids in (i) and (ii). 
It is understood that a morphism $(L,\iota) \rightarrow (L',\iota')$ in $\mathcal{A}$ 
is an isometry $\sigma : L \isomo L'$ such that $\sigma \circ \iota = \iota'$,
and similarly, that a morphism $(B,\eta) \rightarrow (B',\eta')$ in $\mathcal{B}$ 
is an isometry $\sigma : B \isomo B'$ with $\resb \sigma\, \circ \eta = \eta '$. 
We also denote respectively by ${\rm O}(L,\iota)$ and ${\rm O}(B,\eta)$ 
the groups ${\rm Aut}_{\mathcal{A}}(L,\iota)$ and ${\rm Aut}_{\mathcal{B}}(B,\eta)$.

\begin{pf} 
Fix $(L,\iota)$ in $\mathcal{A}$ and denote by $B$ the orthogonal of $\iota(A)$ in $L$.
By (R2), the subgroup $I:=L/(\iota(A) \perp B)$ of $\resb \,\iota(A) \perp \resb \,B$ is totally isotropic.
Both projections ${\rm pr}_A : I \rightarrow \resb\,\iota(A)$ 
and ${\rm pr}_B : I \rightarrow \resb\,B$ 
are injective as $\iota(A)$ and $B$ are saturated in $L$. 
The image of ${\rm pr}_A$ is $(\resb\,\iota)(H)$ by assumption.
So the formula $\eta:= {\rm pr}_B \circ {\rm pr}_A^{-1} \circ \resb\,\iota$ 
defines an isometric embedding $H \rightarrow -\resb B$, 
and we have $I={\rm I}(\iota,\eta)$ with
\begin{equation}
\label{iiotaeta}
{\rm I}(\iota, \eta) := \{ (\resb\,\iota)(h) + \eta(h)\, |\, \, h \in H\}.
\end{equation}
This construction defines a natural functor 
$G : \mathcal{A} \rightarrow \mathcal{B}, (L,\iota) \mapsto (B,\eta)$.
An automorphism of $(L,\iota)$ in $\mathcal{A}$ has the form ${\rm id}_{\iota(A)} \times g$ 
with $g \in {\rm O}(B)$, so the definition of ${\rm I}(\iota, \eta)$, 
and the injectivity of $\resb\,\iota$, show that $G$ is fully faithful.
Conversely, fix $(B,\eta)$ in $\mathcal{B}$. 
Define ${\rm I}({\rm id},\eta)$ as above in $\resb \,A \perp \resb \,B$. 
This is an isotropic subspace, 
hence by (R2) of the form $L/(A \perp B)$ for some unique lattice $L$, 
which defines a natural object $(L,{\rm id})$ in $\mathcal{A}$ 
whose image under $G$ is $(B,\eta)$. 
This shows that $G$ is an equivalence. \par
The assertion $|H|^2 \det L\,=\,\det A\, \det B$ follows from $|I|=|H|$.

Assume finally $H=\resb A$. The subspace $\eta(H) \subset \resb B$
is isometric to $-\resb A$, and in particular nondegenerate, 
so we may write $\resb B = \eta(H) \perp S$. 
The subspace ${\rm I}(\iota,\eta) \subset \resb \iota(A) \perp \eta(H)$ is a Lagrangian, 
so we deduce ${\rm I}(\iota,\eta)^\perp = {\rm I}(\iota,\eta) \perp S$, and then
$\resb L \simeq {\rm I}(\iota,\eta)^\perp/{\rm I}(\iota,\eta) \simeq S$.
\end{pf}

\begin{remark}
\label{rem:afternik}
\begin{itemize}
\item[(i)] (unimodular case) In this construction, we have $\det L=1$ $\iff$ $H=\resb A$ 
and $\resb A \simeq - \resb B$. \par
\item[(ii)] (even variant) Assume $A$ is even.
  Pairs $(L,\iota)$ with $L$ even correspond to pairs $(B,\eta)$ with $B$ even and $\eta$ an isometric embedding $H \rightarrow -\resq\,B$, and if \(H=\resq A\) we have \(\resq B \simeq -\resq A \perp \resq L\).\par
\end{itemize}
\end{remark}

We now discuss the case ${\rm rk}\, A=1$.
Let $L$ be a lattice and let $v \in L$ be nonzero.
Recall that $v$ is called {\it primitive} if $\Z v$ is saturated in $L$.
Define the {\it modulus of $v$ in $L$} to be 
the unique integer ${\rm m}(v)\geq 1$ satisfying
\begin{equation}
\label{eq:defmodu}
\{ v \cdot  x\,|\,x \in L\} = {\rm m}(v) \Z.
\end{equation}
This is also the largest integer $m$ with $v \in mL^\sharp$, 
{\it i.e.} such that $v \bmod mL$ is in the kernel of the natural $\Z/m$-valued bilinear form on $L/mL$.
The integer ${\rm m}(v)$ always divides $v \cdot v$, 
as well as $\det L$ if $v$ is primitive 
(consider the Gram matrix of a basis of $L$ containing $v$).

\begin{example}
\label{ex:orthospecial}
Fix an integer $d\geq 1$ and a divisor $m$ of $d$. 
The following natural groupoids are equivalent:
\begin{itemize}
\item[(i)] pairs $(L,v)$ with $L$ even and $v \in L$ primitive 
with $v \cdot v=d$ and ${\rm m}(v)=m$.\ps
\item[(ii)] pairs $(N,w)$ with $N$ even and $w \in \resq\,N$ of order $d/m$ 
satisfying ${\rm q}(w) \equiv -\frac{m^2}{2d} \bmod \Z$.
\end{itemize}
In this equivalence, 
we have $N \simeq  L \cap v^\perp$ and \,\,$m^2 \det N = d \det L$.
\end{example}

For $d \geq 1$ we denote by $\langle d \rangle$ 
the rank $1$ lattice $\Z e$ with $e \cdot e=d$.
It is equivalent to give an isometric embedding $\iota : \langle d \rangle \rightarrow L$ 
(with saturated image) and a (primitive) norm $d$ element of $L$, namely $\iota(e)$.

\begin{pf}
Set $A = \langle d \rangle$. 
For $L$ a lattice and
$\iota : A \rightarrow L$ an isometric embedding with $v=\iota(e)$ primitive, 
the image of the orthogonal (and surjective) projection $L \rightarrow A^\sharp$ 
is $m A^\sharp$ with $m={\rm m}(v)$. 
The subgroup $H = m \,\resb\, A$ of $\resb A$ is 
cyclic of order $d/m$ and generated by $m e/d$, 
whose norm is $m^2/d$. It is thus equivalent to give 
an isometric embedding $\eta : H \rightarrow -\resq N$ 
and the element $\iota(e/d)$, 
which can be any element $w$ of $\resq N$ of order $d/m$ with ${\rm q}(w)= -\frac{m^2}{2d}$.
The statement follows then from Proposition~\ref{prop:nikulingroupoids} 
in the even case (Remark~\ref{rem:afternik} (ii)). 
\end{pf}

\begin{definition}
\label{def:spevec}
A primitive nonzero vector $v$ in a lattice $L$ is called {\rm special} if it satisfies ${\rm m}(v)=\det L$.
\end{definition}

It is equivalent to require $v \in (\det L) L^\sharp$, 
as this implies $\det L\, |\, {\rm m}(v)$ and the opposite divisibility always holds. 
Special vectors will play a role in Sect.~\ref{sect:hdiscp}.

\subsection{Characteristic vectors and even sublattices}
\label{subsect:charvec}

If $L$ is a lattice, a {\it characteristic vector} of $L$ is an element $\xi \in L^\sharp$ such that
for all $x \in L$ we have $x \cdot \xi \equiv x \cdot x \bmod 2$.
Characteristic vectors always exist and form a unique class
${\rm Char}\,L$ in $L^\sharp/2L^\sharp$. 
For $\xi \in {\rm Char}\,L$, the lattice 
$$L^{\rm even}=\{x \in L\, |\, x \cdot \xi \equiv 0 \bmod 2\}=\{ x \in L\, |\, \, x \cdot x\equiv 0 \bmod 2\}$$
is the largest even sublattice of $L$, 
and called the {\it even part} of $L$; 
it has index $2$ whenever $L$ is odd.

\begin{prop}
\label{prop:ortheven}
Assume $(L,\iota)$ and $(B,\eta)$ correspond to each other 
in the equivalence of Proposition~\ref{prop:nikulingroupoids}, with $H=\resb\,A$.
Then $B$ is even if and only if $A$ contains a characteristic vector of $L$.
\end{prop}

\begin{pf}
We may view $L$ as a sublattice of $(A \perp B)^\sharp = A^\sharp \perp B^\sharp$
(hence $\iota$ as an inclusion), and we denote by
${\rm pr}_A : L \rightarrow A^\sharp$ the orthogonal projection,
with kernel $B^\sharp \cap L = B$. By assumption on $H$ we
have ${\rm pr}_A(L)=A^\sharp$, which implies
\begin{equation}
\label{eq:orthmod2AisB} {\rm pr}_A^{-1}(2A^\sharp)=B+2L.
\end{equation}
We have a natural perfect pairing \(L/2L \times L^\sharp/2L^\sharp \to \Z/2\Z\).
For this pairing, the orthogonal of ${\rm Char}\, L$ is $L^{\rm even}/2L$, 
and that of $A+2L^\sharp$ is the orthogonal mod $2$ of $A$ in $L/2L$,
which is also ${\rm pr}_A^{-1}(2A^\sharp)/2L=(B+2L)/2L$ by \eqref{eq:orthmod2AisB}.
So the image of $A$ in $L^\sharp/2L^\sharp$ contains ${\rm Char}\,L$
if, and only if, $B \subset L^{\rm even}$.
\end{pf}

\begin{example}
\label{ex:charequivalencecharvectors}
For every integer $d\geq 1$, 
there is an equivalence between the natural groupoids of:
\begin{itemize}
\item[(i)] pairs $(L,v)$ with 
$L$ an odd unimodular lattice 
and $v \in {\rm Char}\, L$ primitive of norm $d$, \ps
\item[(ii)] pairs $(N,w)$ with 
$N$ an even lattice of determinant $d$, 
and $w \in \resq\, N$ of order $d$ 
satisfying ${\rm q}(w) \equiv  \frac{1}{2}(1-\frac{1}{d}) \bmod \Z$.
\end{itemize}
In this equivalence, we have $N \simeq  L \cap v^\perp$.
\end{example}

\begin{pf}
Apply Proposition~\ref{prop:nikulingroupoids} to 
$A= \langle d \rangle = \Z e$ and $H=\resb\, A$.
As in Example~\ref{ex:orthospecial} isometric embeddings \(\iota: A \to L\) with saturated image 
correspond bijectively to primitive elements of \(L\) of norm \(d\), 
by mapping \(\iota\) to \(v = \iota(e)\).
Any primitive element $v$ of a unimodular lattice satisfies ${\rm m}(v)=1$.
As $H$ is cyclic of order $d$ and generated by $e/d$, 
it is the same to give an isometric embedding $\eta : \resb A \rightarrow - \resb N$ 
and an order $d$ element $w \in \resb N$ with $w \cdot w\equiv -1/d \bmod \Z$ (namely, $w=\eta(e/d)$).
Assume $(L,v)$ and $(N,w)$ correspond to each other.
We have $\det L = 1\,\,\iff \,\,\det N= d$, 
and we assume that this condition holds.
We may also assume $N= L \cap v^\perp$.

Assume that \(L\) is odd and \(v \in {\rm Char}\, L\).
By Proposition~\ref{prop:ortheven} we know that \(N\) is even.
We either have \({\rm q}(w) \equiv -\frac{1}{2d} \bmod \Z\) 
or \({\rm q}(w) \equiv \frac{1}{2}(1-\frac{1}{d}) \bmod \Z\).
The first case is not possible as \(L\) would be even (Example~\ref{ex:orthospecial}).

Conversely assume that \(N\) is even 
and \({\rm q}(w) \equiv \frac{1}{2}(1-\frac{1}{d}) \bmod \Z\).
As $H$ is cyclic generated by the class of $e/d$, 
the Lagrangian of $\resb \iota(A) \perp \resb N$ defining $L$ 
is generated by the class of an element $f \in L$ 
of the form $f = \frac{v}{d} + n$, with $n \in N^\sharp$ and $n \equiv w \bmod N$.
We have \(f \cdot f = \frac{1}{d} + n \cdot n \in 1 + 2\Z\) and so \(L\) is odd.
It then follows from Proposition~\ref{prop:ortheven} that \(v\) belongs to \({\rm Char}\, L\).
\end{pf}

\subsection{Marked ${\rm BV}$ invariant of depth $d$}
\label{subsect:markedBV}
	The ${\rm BV}$ invariant of a lattice $L$ was introduced in \cite[\S 3]{cheuni2}.
It is a variant of some polynomial invariants defined by Bacher and Venkov in \cite{bachervenkov} 
(see Remark 3.3 in~\cite{cheuni2}), 
and is especially simple to implement on a computer. 
This invariant will play a crucial role in the proof of Theorem~\ref{thm:uni29no1}, 
and we will now explain a generalization of it 
that we will use to study non-unimodular lattices in Sect.~\ref{sect:hdiscp}.\ps

	Fix a lattice $A$, as well as a $\Z$-basis $\alpha_1,\dots,\alpha_r$ of $A$. 
We are interested in finding invariants for the isometry classes of 
pairs $(L,\iota)$ with $\iota : A \rightarrow L$ an isometric embedding. 
Fix such a pair $(L,\iota)$ and choose some integer $d$ (the {\it depth}). 
Consider the graph $\mathcal{G}_{\leq d}(L)$ 
whose set $V$ of vertices consists of the unordered pairs $\{ \pm v\}$ 
with $v \in L$ a nonzero vector of norm $\leq d$, 
and with an edge between $\{ \pm v\}$ and $\{ \pm v'\}$ 
if and only if $v \cdot v' \equiv 1 \bmod 2$. 
Let $M$ be the adjacency $V \times V$-matrix of this graph. 
To each $\{ \pm v\} \in V$, we associate the following two objects:
\begin{itemize}
\item[(i)] The multiset ${\rm m}(v)$ of entries in the $v$-th column of $M^2$, \ps
\item[(ii)] The set of $r$-tuples ${\rm a}(v) = \{\pm (a_1,\dots,a_r)\} \subset \Z^r$ with $a_i = \iota(\alpha_i) \cdot  v$.\ps
\end{itemize}
In other words, ${\rm m}(v)$ is the collection of 
{\it numbers of length-$2$ paths in $\mathcal{G}_{\leq d}(L)$ starting at $v$ and ending at each fixed vertex}.

\begin{definition}
\label{def:markedhbv}
Fix a lattice $A$, 
a basis $\alpha=(\alpha_1,\dots,\alpha_r)$ of $A$, 
and an integer $d\geq 1$.
For a pair $(L,\iota)$ with $L$ a lattice 
and $\iota : A \rightarrow L$ an isometric embedding, 
the ${\rm BV}$-invariant of $(L,\iota)$ of depth $d$, and relative to $\alpha$, is the multiset
\[ {\rm BV}^\alpha_d(L,\iota) = \{\!\{\,({\rm m}(v),{\rm a}(v)) \,\,|\,\, v \in V\,\}\!\}. \]
\end{definition}

This is clearly an invariant of the isomorphism class of $(L,\iota)$.
Changing $\alpha$ into $g(\alpha)$ with $g \in {\rm GL}_r(\Z)$ amounts to applying $g$ to each \({\rm a}(v)\).
The exact choice of $\alpha$ is thus not essential, 
and we shall often simply write ${\rm BV}_d(L,\iota)$.
The original, unmarked, ${\rm BV}$-invariant, denoted ${\rm BV}(L)$, 
is the case $A=\{0\}$, $\alpha=()$ and $d=3$. \ps

\begin{remark}
\label{rem:varianthbv} {\rm (Variants)}
{\rm
There are several variants of ${\rm BV}$ which are both finer and more natural.
For instance, we could rather have defined $M$ as 
the $V \times V$-matrix $(|v \cdot w|)_{ (\pm v, \pm w) \in V \times V}$ ({\it absolute variant})
or doubled the size of the set $V$ of vertices of $\mathcal{G}_{\leq d}(L)$ 
by distinguishing a vector and its opposite ({\it signed variant}).
The reason for our choices is purely practical.
Indeed, the computation of ${\rm BV}$ requires\footnote{
We ignore here the asymptotically faster algorithms for matrix multiplication, 
as the matrix sizes occurring in this paper do not warrant using them. }
${\rm O}(n^3)$ operations with $n=|{\rm R}_{\leq d}(L)|/2$,
as the most time consuming part is to compute the square of the adjacency matrix $M$ above, 
which is of size $n$.
The $\pm$ tricks thus allows to divide the computation time by $8$.
Also, the fact that our $M$ has only $\{0,1\}$-coefficients 
allows to substantially speed up the computation of $M^2$ (see~\S\ref{subsect:algolat} (f)).
Fortunately, these choices turned out to be harmless in practice, 
presumably because the graphs $\mathcal{G}_{\leq d}(L)$ we encountered 
are so random or complicated. 
See however Remark~\ref{rem:case235} for a case where 
the absolute variant is needed.}
\end{remark}

These invariants ${\rm BV}^\alpha_d(L,\iota)$ are 
typically efficient in practice when we are interested in 
pairs $(L,\iota)$ for which $d$ is large enough 
so that ${\rm R}_{\leq d}(L)$ generates $L$ over $\Z$. 
The tension is that when ${\rm R}_{\leq d}(L)$ is too large, 
${\rm BV}^\alpha_d(L,\iota)$ is too long to compute.
Our current implementation, discussed in Sect.\ref{subsect:algolat} (f), runs in about $150$ \texttt{ms} for $n \approx 1300$, which is more than $10$ times faster than the previous implementation of ${\rm BV}$ in \cite{cheuni2}.
The most important open question about ${\rm BV}^\alpha_d(L,\iota)$ 
is to provide a conceptual explanation of why it is so sharp 
in the situations occurring in this paper.
This is purely empirical so far.

\subsection{Lattice algorithms}
\label{subsect:algolat} ${}^{}$ 
In our computations, we make extensive use of several classical lattice algorithms. 
We used the open-source computer algebra system \cite{parigp}. 
In order to perform our computations as efficiently as possible, 
we had to refine several of these algorithms (or their implementations in PARI/GP) 
and use specific variants tailored for the lattices we consider (such as unimodular lattices).
This is actually an important aspect of our work, 
although we only briefly discuss it below, 
referring to \cite{chetaiweb} for our scripts and 
a more furnished documentation. 
These algorithms have been mostly developed by the second author 
and will be the subject of an independent publication (see already \cite{taibiateliergp}).\ps

\begin{itemize}
\item[(a)] The Fincke-Pohst algorithm \cite{FP} 
is used to determine the short vectors in a lattice $L$, 
{\it i.e.} ${\rm R}_{\leq i}(L)$, from a Gram matrix $g$ of $L$. 
This algorithm is implemented as $\texttt{qfminim}(g,i)$ in PARI/GP.
As this implementation uses floating-point numbers, hence approximations, 
the second author implemented an exact variant \texttt{eqfminim} of it.\ps
\item[(b)] The root system of a lattice $L$, and its various attached objects 
(Weyl vector, simple roots, irreducible components, isomorphism class...) 
are easily determined from ${\rm R}_2(L)$: 
see e.g. \cite[Remark 4.2]{cheuni}.
\ps
\item[(c)] The Plesken-Souvignier algorithm \cite{pleskensouvignier} 
returns the order and generators of ${\rm O}(L)$ from a Gram matrix of $L$. 
It is possible to exploit our knowledge of root systems to improve this algorithm. 
Indeed, as explained in \cite[Remark 4.4]{cheuni}, 
the extra bilinear forms allowed in Souvignier's code (PARI's \texttt{qfauto}) 
make it possible to directly compute the {\it reduced} isometry group of $L$, 
which is often much faster. 
As most of our lattices have a {\it ``trivial''}\footnote{
By {\it ``trivial''} here we mean  ${\rm O}(L)^{\rm red}=\langle -w \rangle$, 
see Sect.~\ref{subsect:notation}.} 
reduced isometry group, this is an important simplification.
A much more efficient implementation of this idea 
was developed by Ta\"ibi in \cite{taibiateliergp} 
and led to the function \texttt{qfautors} in \cite{chetaiweb}. \ps

\item[(d)] For the Plesken-Souvignier algorithm, 
or its variant above, to be efficient, 
we must first find a $\Z$-basis of the lattice with shortest possible vectors. 
In~\cite[\S 4]{cheuni2}, a simple probabilistic algorithm to find such bases is given. 
The function \texttt{goodbasis}, developed by the second author, 
takes a further step by combining this idea and an LLL reduction. 
All the Gram matrices given in \cite{chetaiweb} have been found by this algorithm.\ps

\item[(e)] The function \texttt{orbmod2} in \cite{chetaiweb} 
takes as input a list of matrices $m_1$, $\dots$, $m_r$ in ${\rm M}_n(\Z)$ 
with odd determinant and returns 
representatives for the orbits of the subgroup 
$\langle m_1,\dots,m_r \rangle \subset {\rm GL}_n(\Z/2)$ 
acting on $(\Z/2)^n$, as well as the cardinality of each orbit.
This is a standard algorithm in computational group theory, 
adapted to the case at hand for greater efficiency 
(``population count'' instructions allow for fast matrix multiplication over \(\Z/2\)).\ps

\item[(f)] We implemented the invariant \({\rm BV}_d^\alpha(L,\iota)\) of Definition \ref{def:markedhbv} 
as a function \texttt{fast\_marked\_HBV}, 
taking as input the Gram matrix of \(L\) and \(\alpha\).
Of course it would be inefficient (in both space and time) 
to compute actual multisets, so we apply hash functions to them 
to obtain an integer between \(0\) and \(2^N-1\) 
(we chose a hash function with \(N=64\)).
After computing the set of vertices \(V\) using the Fincke-Pohst algorithm 
the main step consists of computing the square of the adjacency matrix \(M\).
We take advantage of the fact that \(M\) has coefficients in \(\{0,1\}\) 
to save space when storing \(M\) (using approximately \(|V|^2\) bits) 
and to compute its square efficiently 
(again, using ``population count'' instructions).
Note that it is not necessary to store \(M^2\), 
as we only need the hash of the multiset of its columns.
\end{itemize}

\section{Unimodular lattices having a pair of orthogonal roots}
\label{sect:genunimodwithpairroots}

As explained in the introduction \S~\ref{sect:introunimod},
The starting point of our approach to Theorem~\ref{thmintro:uni29} is
the following general relation between rank $n+2$ unimodular lattices
having a pair of orthogonal roots and rank $n$ unimodular lattices.
We denote by ${\rm Q}_0$ the root lattice ${\rm A}_1 \perp {\rm A}_1$.
	
\begin{prop}
\label{prop:uninplus2n}
For any integer $n\geq 1$, 
there are equivalences between 
the three following natural groupoids:\begin{itemize}
\item[(i)] pairs $(L,e)$ with 
$L$ a unimodular lattice of rank $n$ and 
$e$ an element of $L/2L$ with $e \cdot e \equiv 2 \bmod 4$, \par
\item[(ii)] lattices $M$ of rank $n$ such that 
$\resb\, M$ is isomorphic to $\resb\, {\rm Q}_0$, \par
\item[(iii)] pairs $(U,\,\{\alpha,\beta\})$ 
with $U$ a unimodular lattice of rank $n+2$ and 
$\{\alpha,\beta\}$ a pair of orthogonal roots in $U$ with $\frac{\alpha+\beta}{2} \notin U$.
\end{itemize}
\end{prop}
\ps

If $L$ is a unimodular lattice, 
its index $2$ subgroups are the \begin{equation}
\label{eq:defM2Le}	
{\rm M}_2(L;e):=\{ v \in L \, |\, v \cdot e \equiv 0 \bmod 2\},
\end{equation}
with $e \in L/2L$ nonzero (and uniquely determined).
	
\begin{lemma}
\label{lem:ind2ML} 
\begin{itemize}
\item[(i)] For any lattice $M$ satisfying $\resb\, M \simeq \resb\, {\rm Q}_0$, 
there is a unique unimodular lattice $M \subset L \subset M^\sharp$, 
and it satisfies $L/M \simeq \Z/2$.
\item[(ii)] Let $L$ be a unimodular lattice, 
$e$ a nonzero element in $L /2L$, and set $M={\rm M}_2(L;e)$. 
We have $\resb\, M \simeq \resb\,{\rm Q}_0$ if and only if 
$e \cdot e \equiv 2 \bmod 4$.
\end{itemize}
\end{lemma}

\begin{pf} 
We have $\resb\,{\rm Q}_0 \,=\, \Z/2 \,x \,\perp \,\Z/2\,y$ 
with $x \cdot x \equiv y \cdot y \equiv 1/2$, 
and the third nonzero element $z=x+y$ satisfies $z \cdot z \equiv 0$ 
and is unique in $\resb\,{\rm Q}_0$ with that property. 
This proves assertion (i) by~\S \ref{subsect:gluing} (R2).\ps
We now prove (ii). By unimodularity of $L$, 
there is $f \in L$ with $e \cdot f \equiv 1 \bmod 2$.
The three nonzero elements of $M^\sharp/M$ are thus 
the classes of $f$, $\frac{e}{2}$ and $\frac{e}{2}+f$.
We conclude as we have 
$f \cdot f \equiv 0 \bmod \Z$, 
$\frac{e}{2} \cdot f \equiv 1/2 \bmod \Z$ and 
$(\frac{e}{2}+f) \cdot (\frac{e}{2}+f) \equiv \frac{e}{2} \cdot \frac{e}{2} \equiv  \frac{e \cdot e}{4} \bmod \Z$.
\end{pf}

We now prove Proposition~\ref{prop:uninplus2n}.
	
\begin{pf} 
We respectively denote by $\mathcal{A}_n$, $\mathcal{B}_n$ and $\mathcal{C}_n$ 
the three groupoids defined in (i), (ii) and (iii) of the statement.

We first prove the equivalence between \(\mathcal{B}_n\) and \(\mathcal{C}_n\).
Write ${\rm Q}_0 = \Z \alpha_0 \perp \Z \beta_0$.
Note that the map $\iota \mapsto (\iota(\alpha_0),\iota(\beta_0))$ 
identifies the embeddings $\iota : {\rm Q}_0 \rightarrow U$ with 
the {\it ordered} pair $(\alpha,\beta)$ of orthogonal roots in $U$. 
Also, in this bijection, $\iota({\rm Q}_0)$ is saturated in $U$ 
if and only if we have $\frac{\alpha+\beta}{2} \notin U$.
The equivalence between \(\mathcal{B}_n\) and \(\mathcal{C}_n\) is thus 
a variant of Proposition~\ref{prop:nikulingroupoids} 
in the case $\det L =1$ and $A={\rm Q}_0$ (see Remark~\ref{rem:afternik} (i)).
Fix $(U,\{\alpha,\beta\})$ in $\mathcal{C}_n$, 
we have seen that $D:= \Z \alpha \perp \Z \beta$ is a saturated rank $2$ sublattice of $U$.
Following the analysis in the proof of Proposition~\ref{prop:nikulingroupoids},
for $M:=D^\perp \cap U$ then $U/(D \perp M) \subset \resb\,D \perp \resb\,M$ 
is the graph of an isometry $\resb D \isomo -\resb M$.
As we have $D \simeq {\rm Q}_0$ and $-\resb\, {\rm Q}_0\, \simeq \,\resb \,{\rm Q}_0$, 
we have a natural, essentially surjective, functor 
$\mathcal{C}_n  \rightarrow \mathcal{B}_n, (U,\{\alpha,\beta\}) \mapsto M$.
In order to prove that this is an equivalence, 
it only remains to show that the group morphism 
${\rm O}(U, \{\alpha,\beta\}) \rightarrow {\rm O}(M), \sigma \mapsto \sigma_{|M},$ 
is bijective. But this follows from the fact that \({\rm O}(U,\{\alpha,\beta\})\) is 
the stabilizer of \(U/(D \perp M)\) in \({\rm O}(D, \{\alpha,\beta\}) \times {\rm O}(M)\) 
and the fact that the natural map 
${\rm O}(D, \{\alpha,\beta\}) \rightarrow {\rm O}( \resb\,D)$ is surjective: 
the involution of $D$ exchanging $\alpha$ and $\beta$ induces 
the unique nontrivial element of ${\rm O}( \resb\,D)$.\par
By Lemma~\ref{lem:ind2ML} (ii), 
we have a well-defined functor 
$\mathcal{A}_n \rightarrow \mathcal{B}_n, (L,e) \mapsto {\rm M}_2(L;e)$. 
It is an equivalence by the same lemma, part (i).
\end{pf}

Although we shall not use it, 
it is easy to deduce from description (i) that 
the mass of these three groupoids, say for odd $n$, 
is the mass of the genus of rank $n$ unimodular lattices 
times $\sum_{0\leq i \leq  (n-2)/4} {\binom{n}{2+4i}}$.
Proposition~\ref{prop:uninplus2n} together with the proof of Lemma~\ref{lem:ind2ML} (ii) show the:

\begin{cor}
\label{cor:algounim} 
Let $\mathrm{L}_n$ be a list of representatives of all rank $n$ unimodular lattices.
The following algorithm returns a list $\mathrm{U}_{n+2}$ containing 
representatives of all rank $n+2$ unimodular lattices having a pair of orthogonal roots. 
Start with an empty list ${\rm U}_{n+2}$ and then, for each $L \in {\rm L}_n$ : \ps
{\rm {\sff
(A1) Determine a set $\mathcal{E}(L)$ of representatives of the orbits of ${\rm O}(L)$ acting on $L/2L$, 
and only keep those $e \in \mathcal{E}(L)$ with $e \cdot e \equiv 2 \bmod 4$.\ps

(A2) For each $e \in \mathcal{E}(L)$, 
choose $f \in L$ with $e \cdot f \equiv 1 \bmod 2$, 
and add to $\mathrm{U}_{n+2}$ the lattice
\[ U(L,e):=\,({\rm M}_2(L;e) \perp {\rm Q}_0) \,+ \,\Z \, \frac{e+\alpha_0}{2}\,+ \,\Z \frac{e+2f+\beta_0}{2} \]
with ${\rm Q}_0 = \Z \alpha_0 \perp \Z \beta_0$.
}}
\end{cor}
\ps

We now discuss the redundancies in the list ${\rm U}_{n+2}$ produced by this algorithm.
By construction, the isometry class of each $U$ in ${\rm U}_{n+2}$ 
appears as many times as the number of ${\rm O}(U)$-orbits of 
pairs of orthogonal, and ``saturated'', roots in $U$. 
We now explain a simple way to minimize these redundancies.
Let us totally order the isomorphism classes of irreducible {\rm ADE} root systems, say 
$${\bf A}_1 \prec {\bf A}_2 \prec \dots \prec {\bf A}_n \prec  \dots\prec {\bf D}_4 \prec \dots \prec {\bf D}_n \prec  \dots\prec {\bf E}_6 \prec {\bf E}_7 \prec {\bf E}_8.$$ 
We extend this total ordering $\prec$ on all isotypic root systems $m{\bf X}_n$ 
using as well a lexicographic ordering on $(m,{\bf X}_n)$: 
{\it e.g.} ${\bf E}_6 \prec 2 {\bf D}_5 \prec 3 {\bf A}_1$.  
We denote by ${\rm m}_1(R)$ the number of irreducible components of $R$ occurring with multiplicity $1$.

\begin{definition}
\label{def:relpair}
Let $R$ be a root system and $\{\alpha,\beta\} \subset R$ a pair of orthogonal roots.
We denote by $C_1 \prec C_2 \prec \cdots \prec C_r$ 
the isotypic components of $R$. 
We say that $\{\alpha,\beta\}$ is {\rm relevant} if 
one of the following exclusive assertions\footnote{
We could sharpen a bit this definition when 
$R$ has an irreducible component with multiplicity $2$, 
or in case (ii) with $C_1 \simeq {\bf D}_n$, but this would not affect significantly its applications.}
holds:\begin{itemize}
\item[(i)] ${\rm m}_1(R) \geq 2$, and $\{\alpha,\beta\}$ meets both $C_1$ and $C_2$.\par
\item[(ii)] ${\rm m}_1(R)=1$, and $\{\alpha,\beta\} \subset C_1$.\par
\item[(iii)] ${\rm m}_1(R)=1$, $C_1 \simeq {\bf A}_1$ or ${\bf A}_2$, 
and $\{\alpha,\beta\}$ meets both $C_1$ and $C_2$,\par
\item[(iv)] ${\rm m}_1(R)=0$, and $\alpha$ and $\beta$ are 
in distinct irreducible components of $C_1$.
\end{itemize}
\end{definition}
\par \noindent 
We have the following trivial fact:

\begin{fact} 
\label{fact:relevant} 
If a root system $R$ has a pair of orthogonal roots, 
it also has a relevant such pair.
\end{fact}

For many $R$ there are much less relevant pairs than arbitrary ones.

\begin{example}
\label{rem:m1geq2}
Assume $(L,e)$ and $(U,\{\alpha,\beta\})$ correspond to each other 
as in Proposition~\ref{prop:uninplus2n}, 
as well as ${\rm m}_1(R)\geq 2$ with $R={\rm R}_2(U)$. 
Then $R$ has a unique ${\rm W}(R)$-orbit of relevant pairs of orthogonal roots.
\end{example}

\begin{cor}
\label{cor:algounim2}  
Let ${\rm L}_n$ be as in Corollary~\ref{cor:algounim}. 
Then we produce a list $\mathrm{U}_{n+2}$ containing representatives of 
all rank $n+2$ unimodular lattices having a pair of orthogonal roots and no norm $1$ vectors 
by modifying step (A2) into : \ps
{\rm {\sff
(A2)' only add $U(L;e)$ to the list ${\rm U}_{n+2}$ 
if it has no norm $1$ vector, 
and if $\{\alpha_0,\beta_0\}$ is relevant for the root system of $U(L;e)$. }}
\end{cor}

Indeed, this follows from Fact~\ref{fact:relevant} 
and from the fact that for any pair $\{\alpha,\beta\}$ of orthogonal roots 
the element $\frac{\alpha+\beta}{2}$ has norm $1$.
We end this discussion by observing that Proposition~\ref{prop:uninplus2n} 
also furnishes mass relations between $L$'s and $U$'s 
(hence useful ways to check computations).
Here is an especially simple example:
\begin{example}
\label{rem:m1geq2bis}
In the situation of Example~\ref{rem:m1geq2}, assume $\{\alpha,\beta\}$ is relevant.
Its orbit under ${\rm O}(U)$ is in bijection with $R_\alpha \times R_\beta$, 
where $R_\alpha,R_\beta$ are the irreducible components of $R$ 
containing respectively $\alpha$ and $\beta$. 
Proposition~\ref{prop:uninplus2n} implies
$|{\rm O}(U)| = \frac{|R_\alpha||R_\beta|}{{\rm n}(e)} |{\rm O}(L)|$,
where ${\rm n}(e)$ is the size of the ${\rm O}(L)$-orbit of $e \in L/2L$.
\end{example}

\section{\texorpdfstring{An application: the classification of rank $29$ unimodular lattices}
{An application: the classification of rank 29 unimodular lattices}
}
\label{sect:classrk29}

We now explain the details of our proof of Theorem~\ref{thmintro:uni29}.
By the classification of rank $28$ unimodular lattices in \cite{cheuni2}, 
it is enough to prove the following (see also Corollary~\ref{cor:BVsharpuni} below).

\begin{thm}
\label{thm:uni29no1}
There are $38\,592\,290$ isometry classes of 
unimodular lattices of rank $29$ with no norm $1$ vectors.  
All these are distinguished by their ${\rm BV}$ invariant. 
A list of Gram matrices of representatives is given in {\rm \cite{chetaiweb}}.
\end{thm}

The given list \footnote{Actually, each lattice in our lists 
is given together with its root system, 
its reduced mass and its hashed BV invariant.} 
of Gram matrices allows a simple proof of the theorem: 
check that they are all positive definite, integral, of determinant $1$, 
with minimum $>1$, have different ${\rm BV}$ invariant, 
and that the sum of their masses is the rational given by the mass formula, 
namely (see~\cite[\S 16.2]{conwaysloane} and \cite[\S 6.4]{cheuni}):
\begin{equation*}
\scalebox{.8}{9683137883598841522700149306218386019856601/65188542827444074570459172044800000000}
\end{equation*}
As an indication, the computation of all the BV invariants 
takes about $65$ days of CPU time on a single core 
(about $145$ \texttt{ms} per lattice), 
and that of the reduced masses about $15$ days 
(about $33$ \texttt{ms} per lattice); the other checks are negligible. 
The fact that the ${\rm BV}$ invariant distinguishes all those lattices 
is quite miraculous and only follows from our whole computation.
As in \cite{cheuni2}, much remains to be explained about 
the apparent sharpness of this invariant. \ps

We now explain how we found these Gram matrices.
Recall from Sect.~\ref{sect:introunimod} that, for any root system $R$, 
we denote by ${\rm X}_n^R$
the set of isomorphism classes of rank $n$ odd unimodular lattices $L$
with no norm $1$ vector and root system isomorphic to $R$,
and by ${\rm m}_n(R)$ the reduced mass of this collection of lattices
(see Sect.~\ref{subsect:groupoidsmasses}).
Using~\cite{king}, we know ${\rm m}_{29}(R)$ for each $R$
(see also~\cite[\S 6.4]{cheuni}); 
it is nonzero for exactly $11\,085$ root systems $R$.
The following Table~\ref{tab:mostX29} indicates 
the root systems contributing the most:

\begin{table}[H]
\tabcolsep=3pt
{\scriptsize \renewcommand{\arraystretch}{1.3} \medskip
\begin{center}
\begin{tabular}{c|c|c|c|c|c|c|c|c|c}
$R$ & $7{\bf A}_1 {\bf A}_2$
& $7{\bf A}_1$
& $8{\bf A}_1 {\bf A}_2$
& $6{\bf A}_1$
& $6{\bf A}_1 {\bf A}_2$
& $8{\bf A}_1$
& $7{\bf A}_1 2{\bf A}_2$
& $5{\bf A}_1$
& $9{\bf A}_1 {\bf A}_2$\\
$2\,{\rm m}_{29}(R)$
& $1.64$
& $1.54$
& $1.52$
& $1.45$
& $1.42$
& $1.39$
& $1.21$
& $1.15$
& $1.13$ \\
$|{\rm X}_{29}^R|$
& $1.67$
& $1.57$
& $1.55$
& $1.48$
& $1.45$
& $1.42$
& $1.22$
& $1.18$
& $1.15$\\
\hline
$R$ & $6{\bf A}_1 2{\bf A}_2$
& $9{\bf A}_1$
& $8{\bf A}_1 2{\bf A}_2$
& $5{\bf A}_1 {\bf A}_2$
& $4{\bf A}_1$
& $5{\bf A}_1 2{\bf A}_2$
& $6{\bf A}_1 3{\bf A}_2$
& $7{\bf A}_1 3{\bf A}_2$
& $9{\bf A}_1 2{\bf A}_2$\\
$2\,{\rm m}_{29}(R)$ & $1.06$
& $1.05$
& $1.04$
& $1.01$
& $0.77$
& $0.71$
& $0.68$
& $0.68$
& $0.67$ \\
$| {\rm X}_{29}^R|$
& $1.07$
& $1.07$
& $1.06$
& $1.04$
& $0.80$
& $0.73$
& $0.69$
& $0.69$
& $0.68$
\end{tabular}
\caption{Root systems $R$ listed in the decreasing order of $|{\rm X}_{29}^R|$, 
with the values of $|{\rm X}_{29}^R|$ and $2\,{\rm m}_{29}(R)$ 
given in millions and rounded to $10^{-2}$.}
\label{tab:mostX29}
\end{center}
}
\end{table}
\begin{remark}
\label{rem:roa2}

For any unimodular lattice $L$ of odd rank $n$,
we claim that the reduced isometry group ${\rm O}(L)^{\rm red}$
always has even cardinality $\geq 2$, as its canonical $2$-torsion element $-w$
{\rm (}see Sect.~\ref{subsect:notation}{\rm)} is nontrivial.
Indeed, if $M$ denotes the index $2$ even sublattice of $L$,
we have ${\rm W}(M)={\rm W}(L)$, and
$M$ is in the genus of ${\rm D}_n$.
So $\resb M \simeq \Z/4$ as $n$ is odd, 
and we conclude by Lemma~\ref{lem:minus1wl}.
\end{remark}

We first consider the root systems which do not contain
any pair of orthogonal roots, namely $\emptyset$, ${\bf A}_1$ and ${\bf A}_2$.
We have ${\rm m}_{29}(\emptyset) =$ {\small 49612728929/11136000} $\simeq 4455$,
${\rm m}_{29}({\bf A}_1) = $ {\small 18609637771/673920} $\simeq 27614$
and ${\rm m}_{29}({\bf A}_2)=$ {\small 113241/80} $ \simeq 1415$.
 \ps

\begin{prop}
\label{prop:X29emptyA1A2}
We have $|{\rm X}_{29}^\emptyset|=10\,092$,
$|{\rm X}_{29}^{{\bf A}_1}|=59\,105$ and $|{\rm X}_{29}^{{\bf A}_2}|=3\,714$,
and neighbor forms for representatives of those lattices
are given in {\rm \cite{chetaiweb}}.
\end{prop}

\begin{pf}
The case ${\rm X}_{29}^\emptyset$ was done in~\cite{cheuni2},
with a method lengthily discussed in Sect. 6.7 there.
A similar method works in the other cases,
starting with running the $\texttt{BNE}$ algorithm {\it loc. cit.}
for the {\it visible} root systems ${\bf A}_1$ and ${\bf A}_2$.
\end{pf}

Of course, we deal with the $11\,082$ remaining root systems
using the algorithms described in 
Corollaries~\ref{cor:algounim} and~\ref{cor:algounim2}.
In order to have an idea {\it a priori} of the size of 
the list ${\rm U}_{29}$ provided by Corollary~\ref{cor:algounim},
let us denote by\footnote{
For $R$ irreducible we have ${\rm np}(R)=1$, 
except ${\rm np}({\bf A}_1)={\rm np}({\bf A}_2)=0$ and ${\rm np}({\bf D}_m)=2$ for $m\geq 4$.
For $R=\bigsqcup_{i=1}^h R_i$ with $R_i$ irreducible, 
we have ${\rm np}(R)=\frac{h(h-1)}{2}+\sum_{i=1}^h {\rm np}(R_i)$.}
${\rm np}(R)$ the number of ${\rm W}(R)$-orbits of 
pairs of orthogonal roots in the root system $R$.
Under the heuristic, confirmed by the final computation,
that almost all lattices $U$ in ${\rm X}_{29}^R$ satisfy $|{\rm O}(U)^{\rm red}|=2$,
in which case ${\rm O}(U)$-orbits of pairs of orthogonal roots
coincide with their ${\rm W}(U)$-orbits by Remark~\ref{rem:roa2},
the redundancy of the isomorphism class of $U$
in the list ${\rm U}_{29}$ should be $\simeq {\rm np}({\rm R}_2(U))$,
so we expect
\begin{equation}
\label{eq:nborb}
|\mathrm{U}_{29}| \,\,\simeq \,\,2 \,\,\sum_R\,\, {\rm m}_{29}(R) \,{\rm np}(R)\,\, \simeq \,\,1,28 \cdot 10^{9}.
\end{equation}
This fits our computations: we find about $1.30$ billions\footnote{
To be more precise, as we are only interested in $U$ with ${\rm r}_1(U)=0$, 
we only enumerated pairs $(L,e)$ with 
${\rm R}_1(L) \cap {\rm M}_2(L;e) =\emptyset$, 
but this reduction is innocent.} 
isomorphism classes of pairs $(L,e)$,
using the list $\mathrm{L}_{27}$ (with $17\,059$ elements) given in \cite{cheuni} 
and the algorithm \texttt{orbmod2} (see Sect.~\ref{subsect:algolat}).
This computation takes about $28$ days of CPU time on a single core: 
about $2$ days to compute generators of ${\rm O}(L)$ for each $L \in \mathrm{L}_{27}$, 
and then about $1.72$ \texttt{ms} per orbit.\ps

If we replace ${\rm np}(R)$ in the sum \eqref{eq:nborb} by
the number ${\rm npr}(R)$ of {\it relevant} pairs in $R$ as in Definition~\ref{def:relpair},
we rather find about $309$ millions of orbits, that is $4$ times less!
As it is faster to check if a pair of roots is relevant 
than computing a ${\rm BV}$ invariant,\footnote{
Checking whether a pair or orthogonal roots $\{\alpha,\beta\}$ is 
relevant in ${\rm R}_2(L)$ takes about $9$ \texttt{ms},
which is more than $15$ times faster than computing the BV invariant of $L$.
The selection of all relevant pairs took about 
$4.5$ months of CPU time.}
this shows that it is worth excluding first all pairs $(L,e)$ such that 
$\{\alpha_0,\beta_0\}$ is not relevant for $U(L;e)$, 
as in Corollary~\ref{cor:algounim2}, 
before computing the invariant.
The final list ${\rm U}_{29}$ obtained this way 
has about $312$ millions of elements, 
very close to the estimation above.
As an indication, Table~\ref{tab:mostX29exc} below indicates 
the root systems most affected by the exclusion of non relevant pairs, 
and Table~\ref{tabmaxred} those contributing the most to the final list.

\begin{table}[H]
\tabcolsep=3pt
{\scriptsize \renewcommand{\arraystretch}{1.3} \medskip
\begin{center}
\begin{tabular}{c|c|c|c|c|c|c|c|c}
$R$ &
$8{\bf A}_1 2{\bf A}_2$ &
$8{\bf A}_1 {\bf A}_2$ &
$7{\bf A}_1 2{\bf A}_2$ &
$9{\bf A}_1 {\bf A}_2$ &
$9{\bf A}_1 2{\bf A}_2$ &
$7{\bf A}_1 {\bf A}_2$ &
$10{\bf A}_1 {\bf A}_2$ &
$6{\bf A}_1 2{\bf A}_2$  \\
${\rm exc}(R)$ &
$45.9$ &
$42.6$ &
$42.2$ &
$40.6$ &
$36.2$ &
$34.4$ &
$29.3$ &
$28.6$ \\ \hline
$R$ &
$7{\bf A}_1 3{\bf A}_2$ &
$8{\bf A}_1 3{\bf A}_2$ &
$6{\bf A}_1 3{\bf A}_2$ &
$6{\bf A}_1 {\bf A}_2$ &
$10{\bf A}_1 2{\bf A}_2$ &
$7{\bf A}_1 2{\bf A}_2 {\bf A}_3$ &
$11{\bf A}_1 {\bf A}_2$ &
$9{\bf A}_1 3{\bf A}_2$ \\
${\rm exc}(R)$  &
$28.5$ &
$24.7$ &
$22.4$ &
$21.3$ &
$20.4$ &
$16.1$ &
$15.7$ &
$14.5$
\end{tabular}
{\small
\caption{Root systems $R$ in the decreasing order of the quantity 
${\rm exc}(R)\,=\,2\,{\rm m}_{29}(R)({\rm np}(R)-{\rm npr}(R))$ (counted in millions).}
\label{tab:mostX29exc}
}
\end{center}
}
\end{table}

\begin{table}[H]
\tabcolsep=4pt
{\scriptsize \renewcommand{\arraystretch}{1.3} \medskip
\begin{center}
\begin{tabular}{c|c|c|c|c|c|c|c|c|c}
$R$ &
$8 {\bf A}_1$  & $9 {\bf A}_1$ & $7 {\bf A}_1$ & $10 {\bf A}_1$ & $6 {\bf A}_1$
& $11{\bf A}_1$ & $8{\bf A}_1{\bf A}_2$ & $5{\bf A}_1$  & $7{\bf A}_1{\bf A}_2$   \\
\hline
${\rm npr}(R)$ & $28$ & $36$ & $21$ & $45$ & $15$ & $55$ & $8$ & $10$ & $7$  \\
${\rm red}(R)$ &
$38.9$ & $37.9$ & $32.4$ & $29.8$ & $21.7$ &
$18.5$ & $12.2$ & $11.5$ & $11.5$  \\
\hline
$R$ & $9{\bf A}_1{\bf A}_2$ & $12{\bf A}_1$ &
$6{\bf A}_1 {\bf A}_2$ & $10{\bf A}_1 {\bf A}_2$ & $5{\bf A}_1 {\bf A}_2$ & $4{\bf A}_1$ & $13{\bf A}_1$ & $11{\bf A}_1 {\bf A}_2$ & $4{\bf A}_1 {\bf A}_2$ \\
${\rm npr}(R)$ & $9$ & $66$ & $6$ & $10$ & $5$ & $6$ & $78$ &  $11$ & $4$ \\
${\rm red}(R)$ & $10.1$ & $8.98$ & $8.54$ & $6.51$ & $5.04$ & $4.64$ & $3.30$ &  $3.15$ & $2.32$
\end{tabular}
\caption{Root systems $R$, in the decreasing order of their redundancy ${\rm red}(R)\,=\,2\,{\rm m}_{29}(R)\,{\rm npr}(R)$ (counted in millions).}
\label{tabmaxred}
\end{center}
}

\end{table}
Observe the many root systems $R$ of type $m{\bf A}_1$ in Table~\ref{tabmaxred}. 
For those we have ${\rm npr}(R)=\frac{m(m-1)}{2}$ (all pairs are relevant). 
It is a bit unfortunate that such root systems also tend to have 
a large reduced mass ${\rm m}_{29}(R)$ (see also \cite[Rem. 8 \S 3]{king}). 
The final computation of the ${\rm BV}$ invariants of the elements in ${\rm U}_{29}$ 
takes about $1$ year and a half of CPU time on a single core 
(but of course this step is straightforward to parallelize). 
This is still much less than the computation of ${\rm X}_{28}$ in \cite{cheuni2}, 
proving the clear superiority of this method over the neighbor computations. 
We do find  $38\,592\,290$ different ${\rm BV}$ invariants, 
and then select arbitrarily one lattice for each. 
We computed Gram matrices for these lattices using 
the \texttt{goodbases} algorithm 
(see Sect.~\ref{subsect:algolat}), which takes about $220$ \texttt{h} of CPU time 
on a single core (about $21$ \texttt{ms} per lattice). 
$\square$\ps\ps

We end this section by providing a few additional information 
about the rank $29$ unimodular lattices $L$ with no norm $1$ vectors 
that follow from our computation:\begin{itemize}
\item[(i)] The heuristic average number of roots of $L$, namely 
\scalebox{.8}{$\frac{\sum_R {\rm m}_{29}(R) |R|}{\sum_R {\rm m}_{29}(R)} \simeq 27.1$}, 
is confirmed. Moreover, we have 
${\rm r}_3(L)\, =\, 1856 - 128 \,|{\rm Exc}\,L|\, +\,10 \,{\rm r}_2(L)$ 
by a theta series argument as in~\cite{bachervenkov}.
Here ${\rm Exc}\,L$ is the set of exceptional vectors in $L$: see Sect.~\ref{sect:excuni}. 
We will see in Thm.~\ref{thm:excuni209} that 
$|{\rm Exc}\,L|$ is $0$ for $92.9 \%$ of the lattices, and $0.159$ on average.  
So the average value of ${\rm r}_3(L)$ is about $2107$, 
and the graph appearing in the computation of ${\rm BV}$ has 
$\frac{1}{2}({\rm r}_3(L)+{\rm r}_2(L))\simeq 1067$ vertices on average. \ps
\item[(ii)] Let ${\rm d}(L)$ be the smallest integer $d\geq 2$ 
such that $L$ is generated over $\Z$ by ${\rm R}_{\leq d}(L)$. 
Then $L$ always admits a $\Z$-basis in ${\rm R}_{\leq d}(L)$ for $d={\rm d}(L)$. 
Moreover, we always have ${\rm d}(L) = 3, 4$ or $5$, 
and ${\rm d}(L)=3$ for $38\,590\,862$ lattices, 
${\rm d}(L)=4$ for $1421$ lattices, and
${\rm d}(L)=5$ for $7$ lattices. \ps
\item[(iii)] Some statistics for the order of reduced isometry groups 
are given in Table~\ref{tab:statroa}. 
It is $2$ (resp. $4$) for $95.2\%$ (resp. $4.18 \%$) of the lattices.

\begin{table}[H]
\tabcolsep=1.5pt
{\scriptsize \renewcommand{\arraystretch}{1.3} \medskip
\begin{center}
\begin{tabular}{c|c|c|c|c|c|c|c|c|c|c|c|c|c|c|c}
$\texttt{ord}$ & $2$ & $4$ & $6$ & $8$ & $10$ & $12$ & $14$ & $16$ & $18$ & $20$ & $24$ & $28$ & $32$ & $36$ & $40$ \\
$\#$ & $36\,741\,838$ & $1\,613\,885$ & $7\,942$ & $165\,479$ & $22$ & $12\,147$ & $2$ & $28\,136$ & $19$ & $154$ & $5\,648$ & $13$ & $7\,797$ & $108$ & $239$ \\ \hline
$\texttt{ord}$ & $42$ & $48$ & $56$ & $60$ & $64$ & $72$ & $80$ & $84$ & $96$ & $104$ & $108$ & $120$ & $128$ & $144$ & $160$ \\
$\#$ & $8$ & $2204$ & $3$ & $11$ & $2070$ & $293$ & $94$ & $24$ & $1045$ & $3$ & $3$ & $30$ & $675$ & $213$ & $30$ \\ \hline
$\texttt{ord}$ & $192$ & $216$ & $224$ & $232$ & $240$ & $252$ & $256$ & $288$ & $320$ & $336$ & $384$ & $400$ & $432$ & $480$ & $512$\\
$\#$ & $383$ & $32$ & $1$ & $1$ & $82$ & $3$ & $266$ & $178$ & $9$ & $14$ & $181$ & $1$ & $21$ & $62$ & $90$
 \end{tabular}
\caption{\small Number $\#$ of lattices with reduced isometry group of order $\texttt{ord} \leq 512$.}
\label{tab:statroa}\label{tabstatroa}
\end{center}
}
\end{table}
\item[(iv)] An analysis of the Jordan-H\"older factors using \texttt{GAP} 
and similar to that in \cite[\S 12]{cheuni} 
shows that only $454$ lattices have a non-solvable reduced isometry group. 
Exactly $4$ of those groups have a Jordan-H\"older factor 
not appearing for unimodular lattices of smaller rank: see Table~\ref{tab:X29newsimpl}.

\tabcolsep=3.5pt
\begin{table}[H]
{\scriptsize
\renewcommand{\arraystretch}{1.8} \medskip
\begin{center}
\begin{tabular}{c|c|c|c|c}
${\rm R}_2(L)$ & ${\bf A}_1 {\bf A}_4$ & ${\bf A}_1 {\bf A}_5$ &  ${\bf D}_6$ & ${\bf A}_5$ \\ \hline
$|{\rm O}(L)^{\rm red}|$ & $20\,401\,920$  & $26\,127\,360$  & $52\,254\,720$  &  $3\,592\,512\,000$  \\  \hline
${\rm O}(L)^{\rm red}$ &  \scalebox{.9}{$\Z/2 \times {\rm M}_{23}$} &  \scalebox{.9}{$\Z/2 \times ({\rm U}_4(3) : \Z/2\times \Z/2)$} &  \scalebox{.9}{$\Z/2 \times ({\rm U}_4(3) : {\rm D}_8)$} &  \scalebox{.9}{$\Z/2 \times ({\rm McL} : \Z/2)$}
\end{tabular}
\end{center}
}
\caption{{\small The reduced isometry groups ${\rm O}(L)^{\rm red}$ of 
the $4$ rank $29$ unimodular lattices $L$ with a Jordan-H\"older factor 
not appearing in smaller rank, using \texttt{GAP}'s notation.}}
\label{tab:X29newsimpl}
\end{table}
These $4$ lattices could presumably be directly constructed, and analysed, using the Leech lattice.
\end{itemize}

A computer calculation shows that 
the assertion about ${\rm BV}$ in Theorem~\ref{thm:uni29no1} extends to all unimodular lattices, 
without any restriction on norm $1$ vectors. 
It is thus not even necessary to consider norm $1$ vectors separately.

\begin{cor}
\label{cor:BVsharpuni} 
Let $L$ and $L'$ be two unimodular lattices of same rank $\leq 29$.
Then $L$ and $L'$ are isometric if, and only if, we have ${\rm BV}(L)={\rm BV}(L')$.
\end{cor}

\section{Fertile weights, extensions of root systems and gluing}
\label{sect:fertile}

${}^{}$ If $M$ is an even lattice, its {\it Venkov map} is 
the map $\nu : \resq M \rightarrow \Q_{\geq 0}$ defined by 
$\nu(\overline{x}) \,=\, \min_{y \in x+M} \,\,y \cdot y$.
For any \(x \in M^\sharp\) we have $\nu(\overline{x}) \equiv x \cdot x \bmod 2\Z$. \ps

Assume now $R$ is a root system.
The elements of ${\rm Q}(R)^\sharp$ are usually called {\it weights}.
A weight $\xi \in {\rm Q}(R)^\sharp$ is said {\it minuscule} if 
it satisfies $|\xi \cdot \alpha| \leq 1$ for all $\alpha \in R$.
(Some authors add the condition \(\xi \neq 0\), which we do not impose.)
For each $\xi \in {\rm Q}(R)^\sharp$ the minimal vectors of $\xi+{\rm Q}(R)$ 
coincide with its minuscule weights, 
and form a single ${\rm W}(R)$-orbit (see~\cite[Ch.\ VIII \S 7.3]{bourbaki_Lie78}).
In particular, the Venkov map of ${\rm Q}(R)$ is the norm of the minuscule lifts.

The study of minuscule weights reduces to the case where $R$ is irreducible, 
as $\xi \in {\rm Q}(R)^\sharp$ is minuscule if, and only if, 
so are its orthogonal projections to each of the irreducible components of \(R\).
For \(R\) irreducible, if we choose a positive system for $R$ 
then the non-zero dominant minuscule $\xi$ are 
the fundamental weights $\varpi_i$ associated to the simple roots $\alpha_i$ 
having coefficient $1$ in the highest root $\tilde{\alpha}$.
From the {\it ``Planches''} in~\cite{bourbaki} 
we easily deduce the values of $\varpi_i \cdot \varpi_i$.

\begin{example} 
When $R$ has type ${\bf A}_n$, all the fundamental weights are minuscule 
and their norms are as follows:
\begin{center}
\dynkin[labels={\alpha_1,\alpha_2,\alpha_i,\alpha_{n-1},\alpha_n},edge length=.8cm,
make indefinite edge/.list={2-3,3-4}]{A}{5}
 \, \, \, {\rm with}\, \, \,$\varpi_i  \cdot \varpi_i = \frac{i(n+1-i)}{n+1}$.
\end{center}
Also, there is a group isomorphism $\Z/(n+1) \isomo \resq {\rm Q}(R)$ sending, 
for all $1 \leq i \leq n$, the class of $i$ mod $n+1$ 
to that of $\varpi_i$ mod ${\rm Q}(R)$.
\end{example}

\begin{definition}
Let $\xi \in {\rm Q}(R)^\sharp$.
We say that $\xi$ is {\rm fertile} if we have $\xi \cdot \xi<2$.
Similarly, we say that a class $c \in \resq\, {\rm Q}(R)$ is fertile if we have $\nu(c)<2$.
\end{definition}

Fertile weights are preserved under ${\rm O}({\rm Q}(R))$ and are trivially minuscule.
As an example, the fertile dominant weights in type ${\bf A}_n$ are
$\varpi_1,\varpi_n$, as well as $\varpi_2,\varpi_{n-1}$ for $n\geq 3$,
and $\varpi_3,\varpi_{n-3}$ for $5 \leq n \leq 6$.
Fertile weights will serve in the following simple construction.
Fix a fertile $\varpi \in {\rm Q}(R)^\sharp$.
In the orthogonal direct sum ${\rm Q}(R)^\sharp \perp \Z e_0$
with $e_0 \cdot e_0=2-\varpi \cdot \varpi$ (positive!),  consider the lattice
\begin{equation}
\label{def:QRc}
{\rm Q}(R_\varpi) \,=\, {\rm Q}(R) \,+\, \Z\, (e_0 - \varpi).
\end{equation}
It only depends on the (fertile) class $c$ of $\varpi$ in $\resq {\rm Q}(R)$.
This is a root lattice as $e_0-\varpi$ is a root.
We denote by $R_\varpi$ or $R_c$ its root system,
hence the notation above. By construction,
we have a natural embedding $\upsilon_c : {\rm Q}(R) \rightarrow {\rm Q}(R_c)$.

\begin{remark} {\rm (Dynkin diagram of $R_c$)}
Fix $B=\{\alpha_1,\dots,\alpha_n\}$ a set of simple roots for $R$
and $\varpi$ a fertile weight dominant for $B$.
Let $I \subset \{1,\dots,n\}$ be the subset of $i$ such that $\varpi \cdot \alpha_i=1$,
then $B \cup \{e_0-\varpi\}$ is a basis of $R_c$,
whose Dynkin diagram is obtained from that of $B$
by adding a single node for $e_0-\varpi$
and by connecting this node to each $\alpha_i$ with $i \in I$.
Note that this Dynkin diagram has a marked node,
namely at $e_0-\varpi$.
Any Dynkin diagram with a marked node can be obtained by this construction,
in a unique way.\footnote{
The groupoid $\mathcal{A}$ of triples $(R, B, \varpi)$ with
$R$ a root system, $B$ a basis and $\varpi$ a dominant fertile weight,
is  equivalent to that $\mathcal{B}$ of $(R',B', \alpha)$ with
$R'$ a root system, $B'$ a basis of $R'$ and $\alpha \in B'$.
Indeed, we have defined $\mathcal{A} \rightarrow \mathcal{B},
(R,B,\varpi) \mapsto (R_c,B \cup \{\alpha\},\alpha)$ with $\alpha = e_0-\varpi$.
But we also have $\mathcal{B} \rightarrow \mathcal{A},
(R',B',\alpha) \rightarrow (R,B,\varpi)$ with $B=B' \smallsetminus \{\alpha\}$
and $\varpi$ the orthogonal projection of $-\alpha$ to the space generated by $B$.
These are trivially inverse equivalences.
}
\end{remark}

\begin{example}
\label{ex:fertileex}
\begin{itemize}
\item[(i)] Assume $R \simeq {\bf A}_n$. 
We have $R_{\varpi_1} \simeq R_{\varpi_n} \simeq {\bf A}_{n+1}$, 
$R_{\varpi_2} \simeq R_{\varpi_{n-1}} \simeq {\bf D}_{n+1}$ for $n\geq 3$, 
and $R_{\varpi_3} \simeq R_{\varpi_{n-2}} \simeq {\bf E}_{n+1}$ for $n=5,6,7$.\
\item[(ii)] Assume $R \simeq 3{\bf A}_1$. 
Then $\varpi=\varpi_1+\varpi_2+\varpi_3$ has norm $3/2<2$, hence is fertile, 
and we have $R_\varpi \simeq {\bf D}_4$.
\end{itemize}
\end{example}

We now explain why these apparently irrelevant notions appear in this work.
We first set a few definitions concerning the gluing construction of even unimodular lattices from
${\rm Q}(R)$ (see \S~\ref{subsect:gluing}).

\begin{definition}
\label{def:MRUR}
For $R$ a root system, we denote by:
\begin{itemize}
\item[{\rm (a)}] $\mathcal{M}^{R}$ the groupoid of pairs $(M,\eta)$ with 
\(M\) an even lattice and 
$\eta : \resq M \isomo - \resq {\rm Q}(R)$ an isometry, \par
\item[{\rm (b)}] $\mathcal{U}_R$ the groupoid of pairs $(U,\iota)$ with 
$U$ an even unimodular lattice and 
$\iota : {\rm Q}(R) \rightarrow U$ an isometric embedding with saturated image,\par
\item[{\rm (c)}] $F_R : \mathcal{M}^R \rightarrow \mathcal{U}_R$ 
the natural equivalence given by Proposition \ref{prop:nikulingroupoids} and Remark \ref{rem:afternik} (ii).\footnote{
Recall that for $(M,\eta)$ in $\mathcal{M}^{R}$, 
and $(U,\iota)=F_R(M,\eta)$, then 
$M \perp {\rm Q}(R) \subset U \subset M^\sharp \perp {\rm Q}(R)^\sharp$ 
is defined by the Lagrangian ${\rm I}(\iota,\eta)$ of \eqref{iiotaeta}, 
with $\iota$ the natural inclusion ${\rm Q}(R) \rightarrow U$. }
\end{itemize}
We also write $\mathcal{M}^{{\rm Q}(R)}$, $\mathcal{U}_{{\rm Q}(R)}$  and $F_{{\rm Q}(R)}$ for $\mathcal{M}^{R}$, $\mathcal{U}_R$ and $F_R$.
\end{definition}

If $A$, $B$, $L$ are lattices,
and if we have two isometric embeddings $\iota : A \rightarrow L$ and $\upsilon : A \rightarrow B$,
an {\it extension of $\iota$ to $B$ via $\upsilon$}
is an isometric embedding $\iota' : B \rightarrow L$
such that $\iota' \circ \upsilon = \iota$. We denote by ${\rm Ext}(\iota,\upsilon)$ 
the set of those extensions $\iota'$.
It has a natural action of ${\rm O}(L,\iota)$ given by 
$(g,\iota') \mapsto g \circ \iota'$.
The main observation in this section is the following:
\begin{prop}
\label{prop:orbfibreMc}
Let $(M,\eta)$ in $\mathcal{M}^{R}$ and 
denote by $(U,\iota)$ in $\mathcal{U}_R$ its image under $F_R$.
Fix a fertile class $c \in \resq {\rm Q}(R)$ and 
denote by ${\rm Exc}_{c,\eta}\,M$ the set of elements $e \in M^\sharp$ with 
$e \cdot e=2-\nu(c)$ and $\eta(e) \equiv -c$. 
For any fertile $\varpi$ in $c$, the map
\begin{equation}
\label{eq:defforbfibre}
f : {\rm Ext}(\iota,\upsilon_c) \longrightarrow {\rm Exc}_{c,\eta}\,M, \,\,\,\iota' \mapsto \iota'(e_0-\varpi)+\iota(\varpi),
\end{equation}
is well-defined, bijective and 
independent of the choice of $\varpi$ in $c$. 
Moreover:
\begin{itemize}
\item[(a)] for $\iota' \in {\rm Ext}(\iota,\upsilon_c)$ and $e=f(\iota')$, 
then the subgroup $\iota'({\rm Q}(R_c))$ is saturated in $U$ if, and only if, 
the element $e$ is primitive in $M^\sharp$.\ps
\item[(b)] the map $f$ is equivariant with respect to 
the natural action of ${\rm O}(U,\iota)$ on ${\rm Ext}(\iota,\upsilon_c)$, 
that of ${\rm O}(M,\eta)$ on ${\rm Exc}_{c,\eta}\,M$, and
the natural group isomorphism ${\rm O}(M,\eta) \isomo {\rm O}(U,\iota)$ 
given by $F_R$.
\end{itemize}
\end{prop}

\begin{pf} 
We have $M \perp {\rm Q}(R) \subset U \subset M^\sharp \perp {\rm Q}(R)^\sharp$ 
and $\iota$ is the natural inclusion. 
Choose a fertile $\varpi$ in $c$. As ${\rm Q}(R_c)$ is generated 
by $e_0-\varpi$ and $R$, it is equivalent to give $\iota' \in {\rm Ext}(\iota,\upsilon_c)$ 
and the element $\iota'(e_0-\varpi)$, which can be any element $\beta \in U$ 
satisfying $\beta \cdot \beta=2$ and $\beta \cdot x=-\varpi \cdot x$ 
for all $x \in {\rm Q}(R)$. An element $\beta= e + q$, with $e \in M^\sharp$ 
and $q \in {\rm Q}(R)^\sharp$, has these properties if, and only if, 
we have $\eta(e) \equiv q$, $e \cdot e+q \cdot q=2$ and $q \cdot x=-\varpi \cdot x$ 
for all $x \in {\rm Q}(R)$. But this is equivalent to $q=-\varpi$ and $e \in {\rm Exc}_{c,\eta}\,M$.
This shows that $f$ is well-defined and bijective.
It is obvious that it does not depend on the choice of $\varpi$. \par
For assertion (a), fix $\iota' \in {\rm Ext}(\iota,\upsilon_c)$ and set $e=f(\iota')$.
As $\iota({\rm Q}(R))$ is saturated in $U$, and as we have 
${\rm Q}(R_c) = \Z (e_0-\varpi) \oplus {\rm Q}(R)$, 
the subgroup $\iota'({\rm Q}(R_c))$ is saturated in $U$ if, 
and only if, the element $\iota'(e_0-\varpi)=e-\varpi$ is 
primitive in $U/{\rm Q}(R)$. But the orthogonal projection $U \rightarrow M^\sharp$ 
induces a $\Z$-linear isomorphism $U/{\rm Q}(R) \isomo M^\sharp$ sending $e-\varpi$ to $e$,
concluding the proof of (a). \par
We finally check the (trivial!) assertion (b).
Let ${\rm O}(M,\eta) \isomo {\rm O}(U,\iota), g \mapsto \tilde{g},$ be 
the isomorphism induced by $F_R$.
By definition, $\tilde{g}$ is the unique element $h$ in ${\rm O}(U,\iota)$ with $h_{|M}=g$.
For $\iota'$ in ${\rm Ext}(\iota,\upsilon_c)$ and $g$ in ${\rm O}(M,\eta)$
we have $\tilde{g}(\iota(\varpi))=\iota(\varpi)$, hence
$g(f(\iota')) = \tilde{g}(f(\iota')) = \tilde{g}(\iota'(e_0 - \varpi)) + \iota(\varpi)=f(\tilde{g} \circ \iota')$.
\end{pf}

\begin{cor}
\label{cor:keypropfertile} 
Fix a root system $R$ and $c \in \resq {\rm Q}(R)$ a fertile class. 
Let $\mathcal{M}^{R,c}$ be the groupoid of triples $(M,\eta,e)$ with 
$(M,\eta)$ in $\mathcal{M}^R$ and $e \in M^\sharp$ primitive 
such that $e \cdot e = 2 - \nu(c)$ and $\eta(e) \equiv -c$. 
Then there is an equivalence of groupoids 
$F_{R,c} : \mathcal{M}^{R,c} \rightarrow \mathcal{U}_{R_c}$ such that 
the following diagram commutes:
\begin{equation*}
\xymatrix{
\mathcal{M}^{R,c}  \ar "2,1"^O \ar "1,2"^{F_{R,c}}  & \mathcal{U}_{R_c}\ar "2,2"^{O'} \\
\mathcal{M}^R \ar "2,2"^{F_R}   & \mathcal{U}_R
}
\end{equation*}
where $O : \mathcal{M}^{R,c} \rightarrow \mathcal{M}$ and 
$O' : \mathcal{U}_{R_c} \rightarrow \mathcal{U}_R$ 
are the respective forgetful functors 
$(M,\eta,e) \mapsto (M,\eta)$ and $(U,\iota') \mapsto (U, \iota' \circ \upsilon_c)$.
\end{cor}

\begin{pf} 
The functor $O'$ is well-defined as ${\rm Q}(R)$ is saturated in ${\rm Q}(R_c)$.
Fix $(M,\eta,e)$ in $\mathcal{M}^{R,c}$ and set $(U,\iota)=F_R(M,\eta)$.
By Proposition~\ref{prop:orbfibreMc}, 
there is a unique $\iota' \in {\rm Ext}(\iota,\upsilon_c)$ with 
saturated image and $f(\iota')=e$: define $F_{R,c}(M,\eta,e)$ as 
the object $(U,\iota')$ in $\mathcal{U}_{R_c}$.
This naturally extends to a functor $F_{R,c} : \mathcal{M}^{R,c} \rightarrow \mathcal{U}_{R_c}$.
By the proposition again part (b), 
$F_{R,c}$ is fully faithful. Also, it trivially satisfies $O' \circ F_{R,c} = F_R \circ O$
(and $\mathcal{M}^{R,c} = \mathcal{M}^R \times_{\mathcal{U}_R} \mathcal{U}_{R,c}$).
To check the essential surjectivity, choose $(U,\iota')$ in $\mathcal{U}_{R_c}$.
As $F_R$ is essentially surjective, 
we may assume $(U,\iota' \circ \upsilon_c)=F_R(M,\eta)$ for some $(M,\eta)$ in $\mathcal{M}^R$. 
By surjectivity of the map \eqref{eq:defforbfibre}, 
there is $e \in M^\sharp$ such that 
$F_{R,c}(M,\eta,e)=(U,\iota')$, 
and we are done.
\end{pf}

By composing $F_{R,c}$ with the natural inverse of $F_{R_c}$ (Sect.~\ref{subsect:gluing}), 
we deduce:

\begin{cor}
\label{cor:MRcMRc} 
We have a natural equivalence $\mathcal{M}^{R,c} \rightarrow \mathcal{M}^{R_c}$,
sending any $(M,\eta,e)$ in $\mathcal{M}^{R,c}$ 
to some $(M',\eta')$ in $\mathcal{M}^{R_c}$ with $M'=M \cap e^\perp$.
\end{cor}

As an example we consider the case $R \simeq {\bf A}_n$ and $c \equiv \varpi_1$.
In this case we have seen $R_c \simeq {\bf A}_{n+1}$,
$\nu(c)=\frac{n}{n+1}$ and so $2-\nu(c)=\frac{n+2}{n+1}$.

\begin{cor}
\label{cor:manpi}
$\mathcal{M}^{{\rm A}_n,\varpi_1}$ is naturally equivalent to
the following groupoids:
\begin{itemize}
\item[(i)] pairs $(M,e)$ with 
$M$ an even lattice such that $\resq M \simeq - \resq {\rm A}_n$, and
$e \in M^\sharp$ a primitive element with $e \cdot e = \frac{n+2}{n+1}$,\ps
\item[(ii)] pairs $(M',\eta')$ with 
$M'$ an even lattice and 
$\eta' : \resq M' \rightarrow - \resq {\rm A}_{n+1}$ an isometry.
\end{itemize}
In this correspondence, 
we have $M' = M \cap e^\perp$, hence ${\rm rk} \,M'\, = \,{\rm rk} M -1$.
\end{cor}

As will follow from the proof, we may also replace the condition $
\resq M \simeq - \resq {\rm A}_n$ in (i) by 
the (only apparently) weaker condition $\det M = n+1$.

\begin{pf} 
The equivalence with (ii) is Corollary~\ref{cor:MRcMRc}. 
For that with (i), we claim that for $(M,e)$ as in (i) there is
a unique isometry $\eta : \resq M \isomo - \resq {\rm A}_n$
with $\eta(e) \equiv - \varpi_1$. 
The uniqueness is obvious as $\varpi_1$ generates $\resq {\rm A}_n$. 
For the existence, we have a unique group morphism
\(- \resq \mathrm{A}_n \to \resq M\) mapping the generator \(-\varpi_1\) to \(e\). 
It is clearly an isometry, in particular it is injective 
as $\resq \mathrm{A}_n$ is nondegenerate. 
As source and target have the same cardinality, it is bijective. 
This constructs \(\eta^{-1}\).
\end{pf}
 
%

Observe that even unimodular lattices entirely disappeared 
in the statement of Corollary~\ref{cor:MRcMRc}!
Also, the statement of Corollary~\ref{cor:manpi} establishes 
a rather surprising equivalence between the two genera (i) and (ii), 
and we may wish to have a more direct proof of it.
This led us to discover the following result, 
which does include that equivalence and generalizes it in another direction.
Let us stress however that the use of the approach above, 
and more precisely Proposition~\ref{prop:orbfibreMc}, 
will still be crucial to prove some subsequent results, such as
Theorem~\ref{thm:uniqueorbex} and Proposition~\ref{prop:keyuniqueorbitAn}.
See Definition~\ref{def:spevec} for the notion of special vector.

\begin{prop}
\label{prop:d1d2}
Let $n,d_1,d_2$ be integers $\geq 1$ with 
${\rm gcd}(d_1,d_2)=1$ and $d_1d_2$ even.
The following natural groupoids are equivalent:
\begin{itemize}
\item[(i)] pairs $(L,u)$ where 
$L$ is a rank $n$ even lattice of determinant $d_1$ and 
$u$ is a primitive special vector in $L$ with $u \cdot u=d_1d_2$,\ps
\item[(ii)] pairs $(L,v)$ where 
$L$ is a rank $n$ even lattice of determinant $d_1$ and 
$v$ is a primitive vector in $L^\sharp$ with $v \cdot v=d_2/d_1$, \ps
\item[(iii)] pairs $(N,w)$ where 
$N$ is a rank $n-1$ even lattice of determinant $d_2$ and 
$w$ a generator of $\resq\,N$ with ${\rm q}(w) \,\equiv \,- \frac{d_1}{2d_2} \bmod \Z$.
\end{itemize}
In these correspondences, we have 
$u=d_1 v$ and $N=L \cap u^\perp = L \cap v^\perp$.
 \end{prop}

As we shall see during the proof, for all pairs $(L,v)$ as in (ii), 
the group $\resq\, L$ is cyclic of order $d$, generated by the class of $v$, 
which satisfies by definition ${\rm q}(v) \equiv \frac{d_2}{2d_1} \bmod \Z$.
This shows that the genus of $L$ is uniquely determined, and of course, 
that of $N$ is also determined by (iii).\footnote{
It would be possible to give the precise conditions on $(n,d_1,d_2)$ 
such that those genera are non empty, but we shall not do it.}

\begin{pf}  
Let $L$ be an even lattice with $\det L=d_1$.
We have $d_1 L^\sharp \subset L$ and the map $v \mapsto d_1 v$ 
defines a bijection between 
the set of $v \in L^\sharp$ with $v \cdot v=d_2/d_1$, 
and that of $u \in d_1L^\sharp$ with $u \cdot u=d_1d_2$.
Assume we have $v \in L^\sharp$ with 
$v \cdot v = d_2/d_1$ and ${\rm gcd}(d_1,d_2)=1$.
Then $L^\sharp/L$ is cyclic of order $d_1$, 
generated by the class $\overline{v}$ of $v$.
Indeed, for $k \in \Z$ with $kv \in L$ we have 
$k v \cdot v \in \Z$ hence $k \equiv 0 \bmod d_1$.
In particular, the group $d_1L^\sharp/d_1L \simeq \Z/d_1$ is 
generated by $d_1 v$.
It is clear that if \(d_1 v\) is primitive in \(L\) 
then \(v\) is primitive in \(L^\sharp\).
Let us check the converse.
Assume that \(v\) is primitive in \(L^\sharp\), denote \(u = d_1 v\) and 
assume \(u = m u'\) with \(m \in \Z\) and \(u' \in L\).
Since the image of \(u\) generates \(d_1 L^\sharp / d_1 L\) 
we have that \(m\) is coprime to \(d_1\).
Writing \(1 = am + bd_1\) with \(a,b \in \Z\) 
we deduce \(v = m u'/d_1\) with \(u'/d_1 = av + bu' \in L^\sharp\) and so \(m = \pm 1\).
We have shown that $(L,v) \mapsto (L, d_1v)$ is 
an equivalence between the groupoids in (i) and (ii).\par
The equivalence between (i) and (iii) is 
a consequence of Example~\ref{ex:orthospecial} 
in the special case $\det L=d_1$, $d=d_1d_2$ and ${\rm m}(v)=d_1$ 
($v$ is a primitive special vector of $L$), and $\det N=d_2$.
\end{pf}

\begin{remark} 
\label{rem:d1d2noneven}
The statement also holds if we remove 
all four occurrences of the word {\rm even} in it, and 
replace $\resq\,N$ by $\resb\,N$, and
${\rm q}(w) \equiv - \frac{d_1}{2d_2}$ by $w \cdot w \equiv - \frac{d_1}{d_2}$.
\end{remark}

\section{Exceptional vectors in odd unimodular lattices}
\label{sect:excuni}
${}^{}$
Let $L$ be a unimodular lattice of rank $n$.
Recall that if $L$ is odd, then 
$L$ is in the genus of ${\rm I}_n$, so 
the norm of any characteristic vector of $L$ is\footnote{
The characteristic vectors of ${\rm I}_n$ are 
the $(x_i) \in \Z^n$ with $x_i \equiv 1 \bmod 2$.} 
$\equiv n \bmod 8$.
We denote by ${\rm Exc}\, L$ the set of characteristic vectors of $L$ with norm $<8$, and
following~\cite{bachervenkov} we introduce the following definition:

\begin{definition}
\label{def:except}
A unimodular lattice $L$ is called {\it exceptional} if ${\rm Exc}\, L \neq \emptyset$.
\end{definition}

If $L$ is even, then we have ${\rm Exc}\, L=\{0\}$, 
so $L$ is exceptional and we now focus on the odd case. 
We are interested here in the set ${\rm Exc}\,L$, together with its natural action of ${\rm O}(L)$.

\begin{remark} \label{rem:rel_def_exc}
  The sets \(\mathrm{Exc}_{c,\eta}\, M\) introduced in Section \ref{sect:fertile} (and in Proposition \ref{prop:orbfibreMc}) 
  are related to the sets  \(\mathrm{Exc} L\) as follows, which justifies our terminology for the former.
  Indeed, let \(L\) be an odd unimodular lattice of rank \(n \not\equiv 0 \bmod 8\) and let \(k\) be the integer satisfying \(0 < k < 8\) and \(n + k \equiv 0 \bmod 8\).
  Let \(M = L^\mathrm{even}\), so that \(M\) is in the genus of \(\mathrm{D}_n\).
  Realize \(\mathrm{D}_k\) as \(\mathrm{I}_k^\mathrm{even}\).
  There are exactly two isometries \(\eta_1,\eta_2 : \resq M \rightarrow  - \resq \mathrm{D}_k\) mapping \(L/M\) to \(\mathrm{I}_k/\mathrm{D}_k\).
  Set \(c = \tfrac{1}{2} (1, \dots, 1) \in \resq \mathrm{D}_k\).
  Then \(\tfrac{1}{2} \mathrm{Exc}\, L\) is equal to \(\mathrm{Exc}_{c,\eta_1}\, M \sqcup \mathrm{Exc}_{c,\eta_2}\, M\): for \(e \in \mathrm{Exc}\, L\) we have \(e/2 \in M^\sharp \smallsetminus L\) and exactly one of \(\eta_1\), \(\eta_2\) maps \(e/2 + M\) to \(-c\).
\end{remark}

Using ${\rm Char}(A \perp B) = \{ a+b\, |\, a \in {\rm Char}(A), b \in {\rm Char}(B)\}$,
the study of exceptional unimodular lattices is easily reduced to 
that of those with no norm $1$ vectors (see e.g. \cite[Prop. 9.2]{cheuni}):

\begin{lemma}
\label{red:excno1}
Let $L$ be a unimodular lattice of rank $8k+r$ with ${\rm r}_1(L)=0$ and $0\leq r <8$. 
Then $L \perp {\rm I}_s$ is exceptional if, and only if, so is $L$ and $r+s<8$.
\end{lemma}

Assume $L$ is odd exceptional with ${\rm r}_1(L)=0$. 
This clearly forces $n \not\equiv 0,1 \bmod 8$.
We first recall some results in \cite[\S 9.1]{cheuni}.

\begin{itemize}
\item[(i)] For $n \equiv 2, 3 \bmod 8$, we have $|{\rm Exc}\,L|=2$.\ps
\item[(ii)] For $n \equiv 4 \bmod 8$, we have $2 \leq |{\rm Exc}\, L| \leq 2n$. \ps
\end{itemize}

The situation in case (ii) is more interesting. Indeed, setting $M=L^{\rm even}$,
there are exactly $3$ unimodular lattices in $M^\sharp$ containing $M$, 
namely $L$ and two others called the {\it companions} of $L$. 
Those two companions are odd, so have even part $M$ as well.
If $L$ is exceptional with ${\rm r}_1(L)=0$, 
exactly one of its two companions $L'$ satisfies ${\rm r}_1(L') \neq 0$.
It is non exceptional, satisfies ${\rm O}(L')={\rm O}(M)$ and 
the map $v \mapsto v/2$ induces a natural ${\rm O}(M)$-equivariant bijection 
${\rm Exc}\,L \isomo {\rm R}_1(L')$ 
(in particular, ${\rm O}(M)$ preserves ${\rm Exc}\, L$). 
As we have ${\rm W}(L)={\rm W}(L')={\rm W}(M)$, 
and ${\rm W}(D_r)^{\pm}$ trivially acts transitively on ${\rm R}_1({\rm I}_r)$, 
these descriptions show:

\begin{prop}
\label{prop:orbex1} 
Assume $L$ is an exceptional odd unimodular lattice of rank $n$ 
with $n \equiv 2,3,4 \bmod 8$ and ${\rm r}_1(L)=0$. 
Then ${\rm W}(L)^{\pm}$ acts transitively on ${\rm Exc}\,L$.
\end{prop}

Our first aim now is to pursue this study 
in the quite more subtle case $n \equiv 5 \bmod 8$, 
which is the situation of interest for $n=29$.

\begin{thm}
\label{thm:uniqueorbex} 
Assume $L$ is an exceptional odd unimodular lattice of rank $n$ with 
$n \equiv 5 \bmod 8$ and ${\rm r}_1(L)=0$. Then we have 
$|{\rm Exc}\,L|= 2m$ for $1 \leq m\leq n$, 
and ${\rm W}(L)^{\pm}$ acts transitively on ${\rm Exc}\,L$.
\end{thm}

An immediate consequence of this theorem is the following result,
that was empirically observed for $n=29$ in~\cite{cheuni2}.

\begin{cor}
\label{cor:exnorootdim5mod8}
If $L$ is an exceptional unimodular lattice of rank $n \equiv 5 \bmod 8$ 
and with no nonzero vector of norm $\leq 2$,
then we have $|{\rm Exc}\,L|=2$.
\end{cor}

Note that Lemma~\ref{red:excno1}, together with Proposition~\ref{prop:orbex1} 
and Theorem~\ref{thm:uniqueorbex}, 
imply Theorem~\ref{thm:excsingleorbituni} of the introduction.
As a preliminary, we shall consider certain embeddings of root systems.

\begin{lemma}
\label{lem:orblemmaA3} 
Let $Q$ be an irreducible root lattice. 
Then there is a unique ${\rm O}(Q)$-orbit 
of isometric embeddings $f: {\rm A}_3 \rightarrow Q$, 
unless we have $Q \simeq {\rm D}_m$ with $m\geq 5$ 
and in which case there are two.
\end{lemma}

\begin{pf} 
See Table 4 in \cite{king} (note however that the line $S={\rm A}_3$, $T={\rm D}_j$ of that table 
is incorrect in the case $j=4$, as there is a unique orbit of size $12$) or
the companion paper \cite{chetaiapp}. 
\end{pf}

We use the standard description of the ${\rm A}_n$ lattice, 
with simple roots $\alpha_i= \epsilon_{i+1}-\epsilon_i$ for $1 \leq i \leq n$. 
Here are examples of both types of embeddings ${\rm A}_3 \rightarrow {\rm D}_m$:

\begin{table}[H]
\tabcolsep=3pt
{\scriptsize \renewcommand{\arraystretch}{1.3} \medskip
\begin{minipage}{.45\linewidth}
\centering
Type I\,\,
\dynkin[labels={,,\alpha_2,\alpha_1,\alpha_3},edge length=.8cm,make indefinite edge={1-2}]{D}{5}
\end{minipage}
\begin{minipage}{.53\linewidth}
\centering
Type II\,\,
\dynkin[labels={,\alpha_1,\alpha_2,\alpha_3,},edge length=.8cm,make indefinite edge={1-2}]{D}{5}

\end{minipage}
}
\end{table}
The orthogonal of such an ${\rm A}_3$ in ${\rm D}_m$ is isometric to ${\rm D}_{m-3}$ in Type I, and to ${\rm D}_1 \perp {\rm D}_{m-4}$ in type II (setting ${\rm D}_n= {\rm I}_n^{\rm even}$ as well for $n=1,2,3$).

\begin{lemma}
\label{lem:embA3}
Fix embeddings $f : {\rm A}_3 \rightarrow Q$, 
with $Q$ an irreducible root lattice, 
$\mu :  {\rm A}_3 \rightarrow {\rm A}_4$ and $\nu : {\rm A}_3 \rightarrow {\rm D}_4$. 
Let $\Phi$ denote the set of embeddings $f' : {\rm A}_4 \rightarrow Q$ {\rm extending $f$ via $\mu$},
that is, with $f' \circ \mu = f$. 
Then $\Phi$ has a natural action of the subgroup $W$ of ${\rm W}(Q)$ fixing pointwise $f({\rm A}_3)$, and
\begin{itemize}
\item[(i)] either there is an embedding $f'' : {\rm D}_4 \rightarrow Q$ extending $f$ via $\nu$, \ps
\item[(ii)] or we have $Q \simeq {\rm A}_m$, $|\Phi|=m-3$ and $W$ acts transitively on $\Phi$.
\end{itemize}
\end{lemma}

\begin{pf} 
Observe that neither the hypotheses, nor the conclusions, 
are affected by replacing $f$, $\mu$, $\nu$, by $g\circ f$, $h \circ \mu$ and $k \circ \nu$, 
with $g \in {\rm O}(Q)$, $h \in {\rm O}({\rm A}_4)$ and $k \in {\rm O}({\rm D}_4)$: 
use the bijections $f' \mapsto g \circ f' \circ h^{-1}$ and $f'' \mapsto g \circ f'' \circ k^{-1}$. \par
Assume $Q \simeq {\rm A}_m$. 
By this observation and Lemma~\ref{lem:orblemmaA3}, 
we may assume $Q={\rm A}_m$ and that $f$ and $\mu$ send 
$\alpha_i$ to $\alpha_i$ for $1 \leq i \leq 3$. 
It is the same to give $f' \in \Phi$ and $f'(\alpha_4)$, 
which is any root $\alpha=\epsilon_i - \epsilon_j$ in ${\rm A}_m$ 
with $\alpha \cdot \alpha_1=\alpha \cdot \alpha_2=0$ and $\alpha  \cdot \alpha_3=-1$. 
This is equivalent to $i=4$ and $5 \leq j  \leq m+1$. This shows (ii).\par
Assume $Q \simeq {\rm D}_m$ with $m\geq 5$.
Again we may assume \(Q = \mathrm{D}_m\).
By Lemma~\ref{lem:orblemmaA3} we may assume that 
\(f\) is as in the pictured examples above, 
and we see that (i) holds for both types I and II of $f$.
Assume finally $Q$ has type ${\bf E}$ or ${\bf D}_4$.
Choose any embedding $f'' : {\rm D}_4 \rightarrow Q$.
Then $f'' \circ \nu$ and $f$ are in the same ${\rm O}(Q)$-orbit 
by Lemma~\ref{lem:orblemmaA3}, so (i) holds again.
\end{pf}

We are finally able to prove Theorem~\ref{thm:uniqueorbex}.

\begin{pf}
Let $L$ be a unimodular lattice with odd rank $n \equiv 5 \bmod 8$.
Then $L$ is the unique proper overlattice of its even part $M:=L^{\rm even}$,
and $M$ is in the genus of ${\rm D}_n$. 
As we have $\resq {\rm D}_n \simeq - \resq {\rm A}_3$,
we may choose an isometry $\eta : \resq M \rightarrow  - \resq {\rm A}_3$.
In the notations of Sect.~\ref{sect:fertile}, 
we have $(M,\eta) \in \mathcal{M}^{{\rm A}_3}$, and we may consider
the associated $(U,\iota)=F_{{\rm A}_3}(M,\eta)$, 
with $U$ a rank $n+3$ even unimodular lattice
and $\iota : {\rm A}_3 \rightarrow U$ an embedding with saturated image. 
The Dynkin diagram of $R:={\rm R}_2({\rm A}_3)$, 
labelled by norms of fundamental weights, is
$$\dynkin[labels={\alpha_1,\alpha_2,\alpha_3},labels*={3/4,1,3/4},edge length=.8cm]{A}{3}$$
We are going to apply Proposition~\ref{prop:orbfibreMc} 
to both the fertile classes of $\varpi_1$ and $\varpi_2$,
which satisfy $R_{\varpi_1} \simeq {\bf A}_4$ and $R_{\varpi_2} \simeq {\bf D}_4$.  \ps
(a) We have $2-\varpi_2 \cdot \varpi_2=1$. 
As any element $e \in M^\sharp \smallsetminus L$ satisfies 
$e \cdot e \equiv  \frac{1}{4} \bmod \Z$, 
each norm $1$ vector $e \in M^\sharp$ belongs to $L$ and satisfies 
$\eta(e) \equiv -\varpi_2$. By Proposition~\ref{prop:orbfibreMc}, 
the extensions of $\iota$ to ${\rm Q}(R_{\varpi_2}) \simeq {\rm D}_4$ via $\upsilon_{\varpi_2}$ are thus 
in bijection with norm $1$ vectors in $L$.
As we have ${\rm r}_1(L)=0$ by assumption on $L$, 
there are no such extensions. \ps
(b) We have $2-\varpi_1 \cdot  \varpi_1=5/4$. 
Let $E$ denote the set of norm $5/4$ vectors in $M^\sharp$.
Observe that the ${\rm O}(M)$-equivariant map $E \rightarrow L$, $e \mapsto 2e$, 
induces a bijection $E \isomo {\rm Exc}\,L$. For $e \in E$ we have 
either $\eta(e) \equiv -\varpi_1$ or $\eta(e) \equiv -\varpi_3$, 
and we denote by $E=E_1 \sqcup E_3$ the corresponding partition of $E$. 
It satisfies $E_1=-E_3$. By Proposition~\ref{prop:orbfibreMc}, the set $E_1$ is 
in natural ${\rm O}(M,\eta)$-equivariant bijection with 
the set $\Psi$ of extensions of $\iota$ to $R_{\varpi_1} \simeq {\bf A}_4$
via $\upsilon_{\varpi_1}$. \ps

Let $Q$ be the lattice generated by 
the unique irreducible component of $U$ containing $\iota({\rm A}_3)$.
We apply Lemma~\ref{lem:embA3} to the embedding $f : {\rm A}_3 \rightarrow Q$ defined by $\iota$, 
as well as $\mu = \upsilon_{\varpi_1}$ and $\nu=\upsilon_{\varpi_2}$. 
We have $\Psi=\Phi$. By (a) above, we are not in case (i) of this Lemma, 
so its assertion (ii) is satisfied. 
But the subgroup of ${\rm W}(U)$ fixing pointwise $\iota({\rm A}_3)$ coincides with ${\rm W}(L)$ by 
\cite[Ch. V, \S  3.3, Prop. 2]{bourbaki} and the equalities 
$M=U \cap \iota({\rm A}_3)^\perp$ and ${\rm W}(M)={\rm W}(L)$.
So $E_1$ consists of a single ${\rm W}(L)$-orbit, 
and we have $Q \simeq {\rm A}_{m+3}$ with $m=|E_1|=|E_3|$.
\end{pf}

\begin{cor}
\label{cor:ex2mUmplus3}
Let $m,n \geq 0$ be integers with $n \equiv 5 \bmod 8$.
The set of isometry classes of rank $n$ unimodular lattices
satisfying ${\rm r}_1(L)=0$ and $|{\rm Exc}\, L|=2m$ is in natural bijection
with that of pairs  $(U,C)$ with $U$ a rank $n+3$ even unimodular lattice
and $C$ an irreducible component of ${\rm R}_2(U)$ of type ${\bf A}_{m+3}$.
\end{cor}

\begin{pf}
Given the proof of Theorem~\ref{thm:uniqueorbex},
this follows from the equivalence ${\rm F}_{{\rm A}_3}$ and the two
following facts. First, we have ${\rm O}(\resq {\rm A}_3) = \langle - {\rm id} \rangle$
and for $(M,\eta)$ in $\mathcal{M}^{\rm A}_3$ we always have $(M,\eta) \simeq (M,-\eta)$.
Second, if we have $(U,i_1)$ and $(U,i_2)$ in $\mathcal{U}_{{\rm A}_3}$ such that 
$i_1({\rm A}_3)$ and $i_2({\rm A}_3)$ lie 
in a same irreducible root lattice $Q \subset U$ of type ${\rm A}_{m+3}$, 
there is $g \in {\rm O}(U)$ with $g \circ i_1 = i_2$.
Indeed, there is $h \in {\rm O}(Q)$ with $h \circ i_1 = i_2$ by Lemma~\ref{lem:orblemmaA3},
and we conclude as ${\rm O}(Q)={\rm W}(Q)^{\pm}$
and the restriction map ${\rm W}(L)^\pm \rightarrow {\rm W}(Q)^\pm$ is trivially surjective.
\end{pf}

For example, we easily deduce from this and the classification of Niemeier lattices\footnote{
Recall that the isometry group of a Niemeier lattice $N$ 
permutes transitively the irreducible components of ${\rm R}_2(N)$ of same type.},
the number of exceptional vectors of 
the $12$ rank $21$ unimodular lattices with no norm $1$ vectors.
We finally consider the rank $n=29$. 
It is trivial to compute $|{\rm Exc}\, L|$ for each lattice $L$ 
in the list given by Thm.~\ref{thm:uni29no1}. We find:

\begin{thm}
\label{thm:excuni209}
There are $2\,721\,152$ exceptional unimodular lattices of rank $29$ with no norm $1$ vectors. 
The number $\#$ of those lattices having {\rm \texttt{exc}} exceptional vectors is given in Table~\ref{tab:statexc}.
\end{thm}
\vspace{-.5cm}
\begin{table}[H]
\tabcolsep=1.7pt
{\scriptsize \renewcommand{\arraystretch}{1.3} \medskip
\begin{center}
\begin{tabular}{c|c|c|c|c|c|c|c|c|c|c|c|c|c|c|c|c|c|c|c|c|c|c|c}
$\texttt{exc}$
 & \scalebox{.8}{$2$}
 & \scalebox{.8}{$4$}
 & \scalebox{.8}{$6$}
 & \scalebox{.8}{$8$}
 & \scalebox{.8}{$10$}
 & \scalebox{.8}{$12$}
 & \scalebox{.8}{$14$}
 & \scalebox{.8}{$16$}
 & \scalebox{.8}{$18$}
 & \scalebox{.8}{$20$}
 & \scalebox{.8}{$22$}
 & \scalebox{.8}{$24$}
 & \scalebox{.8}{$26$}
 & \scalebox{.8}{$28$}
 & \scalebox{.8}{$30$}
 & \scalebox{.8}{$32$}
 & \scalebox{.8}{$34$}
 & \scalebox{.8}{$36$}
 & \scalebox{.8}{$40$}
 & \scalebox{.8}{$42$}
 & \scalebox{.8}{$46$}
 & \scalebox{.8}{$56$}
 & other \\
$\#$
 & \scalebox{.8}{$2439727$}
 & \scalebox{.8}{$237232$}
 & \scalebox{.8}{$33400$}
 & \scalebox{.8}{$7509$}
 & \scalebox{.8}{$1966$}
 & \scalebox{.8}{$734$}
 & \scalebox{.8}{$257$}
 & \scalebox{.8}{$165$}
 & \scalebox{.8}{$58$}
 & \scalebox{.8}{$40$}
 & \scalebox{.8}{$18$}
 & \scalebox{.8}{$19$}
 & \scalebox{.8}{$7$}
 & \scalebox{.8}{$7$}
 & \scalebox{.8}{$2$}
 & \scalebox{.8}{$3$}
 & \scalebox{.8}{$2$}
 & \scalebox{.8}{$1$}
 & \scalebox{.8}{$2$}
 & \scalebox{.8}{$1$}
 & \scalebox{.8}{$1$}
 & \scalebox{.8}{$1$}
  & \scalebox{.8}{$0$} \end{tabular}
\caption{\small Number $\#$ of exceptional unimodular lattices of rank $29$ 
having $\texttt{exc}$ exceptional vectors.}
\label{tab:statexc}
\end{center}
}
\end{table}
\vspace{-.5cm}
This table is coherent with the computation in \cite{king} 
of the mass ${\rm m}_{32}^{{\rm II}}(R)$ of the rank $32$ even unimodular lattices 
with root system $\simeq R$. Indeed, if ${\rm N}_m(R)$ denotes 
the number of irreducible components of $R$ of type ${\bf A}_m$, 
Corollary~\ref{cor:ex2mUmplus3} shows that 
$f_{2m} = 2\,\sum_{R} {\rm N}_{m+3}(R)\, |{\rm W}(R)|\, {\rm m}_{32}^{\rm II}(R)$ 
is a (heuristically close) lower bound for the number of isomorphism classes of $L$ 
with ${\rm r}_1(L)=0$ and $|{\rm Exc}\,L|=2m>0$.
We numerically find $\lceil f_0 \rceil = 35\,026\,757$, and for $m\geq 1$:
\vspace{-.4cm}
\begin{table}[H]
\tabcolsep=1.7pt
{\scriptsize \renewcommand{\arraystretch}{1.3} \medskip
\begin{center}
\begin{tabular}{c|c|c|c|c|c|c|c|c|c|c|c|c|c|c|c|c|c|c|c|c|c|c|c}
$2m$
 & \scalebox{.8}{$2$}
 & \scalebox{.8}{$4$}
 & \scalebox{.8}{$6$}
 & \scalebox{.8}{$8$}
 & \scalebox{.8}{$10$}
 & \scalebox{.8}{$12$}
 & \scalebox{.8}{$14$}
 & \scalebox{.8}{$16$}
 & \scalebox{.8}{$18$}
 & \scalebox{.8}{$20$}
 & \scalebox{.8}{$22$}
 & \scalebox{.8}{$24$}
 & \scalebox{.8}{$26$}
 & \scalebox{.8}{$28$}
 & \scalebox{.8}{$30$}
 & \scalebox{.8}{$32$}
 & \scalebox{.8}{$34$}
 & \scalebox{.8}{$36$}
 & \scalebox{.8}{$40$}
 & \scalebox{.8}{$42$}
 & \scalebox{.8}{$46$}
 & \scalebox{.8}{$56$}
 & other \\
$\lceil f_{2m} \rceil $
 & \scalebox{.8}{$2323793$}
 & \scalebox{.8}{$211670$}
 & \scalebox{.8}{$28918$}
 & \scalebox{.8}{$5847$}
 & \scalebox{.8}{$1558$}
 & \scalebox{.8}{$551$}
 & \scalebox{.8}{$210$}
 & \scalebox{.8}{$120$}
 & \scalebox{.8}{$46$}
 & \scalebox{.8}{$32$}
 & \scalebox{.8}{$16$}
 & \scalebox{.8}{$16$}
 & \scalebox{.8}{$6$}
 & \scalebox{.8}{$7$}
 & \scalebox{.8}{$2$}
 & \scalebox{.8}{$3$}
 & \scalebox{.8}{$2$}
 & \scalebox{.8}{$1$}
 & \scalebox{.8}{$2$}
 & \scalebox{.8}{$1$}
 & \scalebox{.8}{$1$}
 & \scalebox{.8}{$1$}
  & \scalebox{.8}{$0$} \end{tabular}
\end{center}
}
\end{table}
\vspace{-.5cm}
\noindent This is consistent with 
the table of Theorem~\ref{thm:excuni209} and explains all $0$ values there.

\begin{remark}
\label{rem:dim6mod8} 
{\rm (Case $n\equiv 6 \bmod 8$) } 
Assume $L$ is a unimodular lattice of rank $n \equiv 6 \bmod 8$ with ${\rm r}_1(L)=0$.
Then ${\rm Exc}\,L$ can be studied using similar ideas.
We can prove that each of the two fibers of the natural map 
$f: {\rm Exc}\, L \rightarrow {\rm Char}(L)/2L^{\rm even}$ 
is a single ${\rm W}(L)^{\pm}$-orbit, or empty.
To be slightly more precise, let us fix an isometry 
$\eta : \resq L^{\rm even} \rightarrow - \resq ({\rm A}_1 \perp {\rm A}_1)$ 
and denote by $(U,\beta_1,\beta_2)$ the rank $n+2$ even unimodular lattice 
equipped with two orthogonal roots
associated to $(L^{\rm even},\eta)$. 
Denote by $C_i$ the irreducible component of ${\rm R}_2(U)$ containing $\beta_i$. 
We can show $C_1 \neq C_2$, and that 
the fibers of $f$ are naturally indexed by $\{1,2\}$ and have size $2{\rm h}(C_i)-4$, 
where ${\rm h}(R)$ denotes the Coxeter number of $R$.
\end{remark}

\section{Even lattices of prime (half-)determinant}
\label{sect:hdiscp} 

Our main aim in this section is to prove Theorems~\ref{thmintro:hdiscptable}
and~\ref{thmintro:hdiscptablebvp} of the introduction. 
We start with some information about the genera $\mathcal{G}_{n,p}$ briefly introduced in Sect.~\ref{subsect:hnp}.

\subsection{The genera $\mathcal{G}_{n,p}$}
\label{subsect:residuegnp} 
Let $n\geq 1$ be an integer and $p$ an odd prime.
We are interested in this section in 
the even lattices of rank $n$ and determinant $d=p$ or $2p$.
A mod $2$ inspection shows that the $d=p$ (resp.\ $d=2p$) case 
is only possible for $n$ even (resp.\ odd). 
Moreover, as is well-known, these lattices form a single genus 
that we will denote by $\mathcal{G}_{n,p}$, 
and which satisfies the following properties: \ps

-- For $n$ even, we have 
$\mathcal{G}_{n,p} \neq \emptyset \,\iff\,n+p \equiv 1 \bmod 4$. 
For $L \in \mathcal{G}_{n,p}$, the nonzero values of 
$x \mapsto x \cdot x \bmod \Z$ on $\resq L \simeq \Z/p$ 
are the $\frac{a}{p}$ such that the Legendre symbol $(\frac{a}{p})$ is {\scriptsize $(-1)^{(n+p-1)/4}$}. 
We denote below by ${\rm R}_{n,p}$ the isometry class of this quadratic space. 
Note that {\scriptsize $(-1)^{(n+p-1)/4}$}$(\frac{2}{p})$ only depends on $p \bmod 4$.  \ps

-- For $n$ odd, we always have 
$\mathcal{G}_{n,p} \neq \emptyset$, and for $L \in \mathcal{G}_{n,p}$, 
we have $\resq L \simeq {\rm R}_{n-1,p} \perp \resq {\rm A}_1$ for $n+p \equiv 2 \bmod 4$, 
and $\resq L \simeq {\rm R}_{n+1,p} \perp -\resq {\rm A}_1$ otherwise. \ps

This information is gathered in Table~\ref{tab:generap2p} below: 
the sign $\pm$ is the Legendre symbol $(\frac{a}{p})$ if 
the residual {\it quadratic} form ${\rm q}(x)=\frac{x \cdot x}{2}\bmod \Z$ 
takes the nonzero value $a/p \bmod \Z$ on $\resq L$, and 
we write it in blue if we have $n$ odd and if ${\rm q}$ takes the value $-1/4$ 
(rather than $1/4$).

\begin{table}[H]
\tabcolsep=6pt
{\scriptsize \renewcommand{\arraystretch}{1.7} \medskip
\begin{center}
\begin{tabular}{c|c|c|c|c|c|c|c|c}
$p \bmod 4 \, \backslash \, n \bmod 8$ & $0$ & $1$ & $2$ & $3$ & $4$ & $5$ & $6$ & $7$ \\
\hline
$1$ & $+$ & $+$ & & ${\color{blue} -}$ & $-$ & $-$ & &  ${\color{blue} +}$ \\
\hline
$3$ & & ${\color{blue} +}$ & $+$ & $+$ & & ${\color{blue} -}$  &  $-$ &  $-$
\end{tabular}
\end{center}
}
\caption{\small The genera $\mathcal{G}_{n,p}$.}
\label{tab:generap2p}
\end{table}

In the important cases $p\leq 7$ for us, Table~\ref{tab:generap2p357} gives 
an example of a lattice in $\mathcal{G}_{n,p}$, up to adding copies of ${\rm E}_8$. 
In this table, $L'$ denotes the orthogonal of a root in $L$, 
${\rm F}_8$ denotes the unique even lattice of determinant \(5\) 
containing ${\rm E}_7 \perp \langle 10 \rangle$ and 
${\rm S}_2$ denotes a rank $2$ lattice with Gram matrix 
{\tiny $\left[\begin{array}{cc} 2 & -1 \\ -1 & 4\end{array}\right]$}.

\begin{table}[H]
\tabcolsep=6pt
{\scriptsize \renewcommand{\arraystretch}{1.7} \medskip
\begin{center}
\begin{tabular}{c|c|c|c|c|c|c|c|c}
$p  \, \backslash \, n$  & $1$ & $2$ & $3$ & $4$ & $5$ & $6$ & $7$ & $8$ \\
\hline
$3$ & $\langle 6 \rangle$ & ${\rm A}_2$ & ${\rm A}_1 \perp {\rm A}_2$ & & ${\rm A}_5$  &  ${\rm E}_6$ &  ${\rm A}_1 \perp {\rm E}_6$ & \\
\hline
$5$ & $\langle 10 \rangle$ & & ${\rm A}'_4$ & ${\rm A}_4$ & ${\rm A}_1 \perp {\rm A}_4$ & &  ${\rm F}'_8$ & ${\rm F}_8$ \\
\hline
$7$ & $\langle 14 \rangle$ & ${\rm S}_2$ & ${\rm A}_1 \perp {\rm S}_2$ & & ${\rm A}'_6$  &  ${\rm A}_6$ &  ${\rm A}_1 \perp {\rm A}_6$ &
\end{tabular}
\end{center}
}
\caption{\small Some lattices in $\mathcal{G}_{n,p}$, for $n \leq 8$ and $p \leq 7$.}
\label{tab:generap2p357}
\end{table}

We conclude with a table relevant for the application of Proposition~\ref{prop:d1d2}.
In this table, it is convenient to include the genus $\mathcal{G}_{n,1}$ of rank $n$ even lattices 
with (half-)determinant $1$.

\begin{table}[H]
\tabcolsep=4pt
{\scriptsize \renewcommand{\arraystretch}{1.7} \medskip
\begin{center}
\begin{tabular}{c|c|c|c|c|c|c|c|c}
$p \, \backslash \, n \bmod 8$ & $1$ & $2$ & $3$ & $4$ & $5$ & $6$ & $7$ & $8$  \\
\hline
$1$ & $\frac{1}{2}, \frac{5}{2}, \frac{9}{2}, \frac{13}{2}$ & & & & & & $\frac{3}{2}, \frac{7}{2}, \frac{11}{2}$ & \\
\hline
$3$ & $\frac{1}{6}, \frac{13}{6}$ & $\frac{2}{3}, \frac{8}{3}, \frac{14}{3}$ & $\frac{7}{6}$ & & $\frac{5}{6}$ & $\frac{4}{3}, \frac{10}{3}$ & $\frac{11}{6}$ & \\
\hline
$5$ & $\frac{1}{10}, \frac{9}{10}$  & & $\frac{3}{10}, \frac{7}{10}$ & $\frac{4}{5},\, \frac{6}{5}, \frac{14}{5}$ & $\frac{13}{10}$ & & $\frac{11}{10}$ & $\frac{2}{5}, \, \frac{8}{5},\, \frac{12}{5}$  \\
\hline
$7$ & $\frac{1}{14}, \frac{9}{14}$ & $\frac{2}{7}, \frac{4}{7}, \frac{8}{7}$ & $\frac{11}{14}$ & & $\frac{5}{14},\,\frac{13}{14}$ & $\frac{6}{7}, \frac{10}{7}, \frac{12}{7}$ &  $\frac{3}{14}$ & 
\end{tabular}
\end{center}
}
\vspace{-5 mm}
\caption{\small Rationals $\frac{m}{d}$ with $m\leq 14$ and ${\rm gcd}(m,d)=1$ satisfying $\lambda \equiv x.x \bmod 2\Z$ for some nonzero $x$ in the residue of the lattices in $\mathcal{G}_{n,p}$, with $d=p$ or $2p$.}
\label{tab:tabd1d2}
\end{table}

\subsection{Proof of Theorem~\ref{thmintro:hdiscptable}: the orbit method}
\label{subsect:orbitmethod}
We focus on the most important cases $n\geq 23$, referring the reader to \cite{chetaiweb} for the easier cases $n\leq 22$.
As already explained in Sect.~\ref{subsect:hnp}, 
our strategy is to deduce everything from 
the classification of unimodular lattices of rank $\leq 29$, 
following the diagram in Figure~\ref{fig:imageliensgenres} 
that we now decipher. 
\begin{itemize}\ps
\item[(a)] Each disc in Figure~\ref{fig:imageliensgenres} 
represents the genus $\mathcal{G}$ of lattices of the given dimension, 
determinant and parity, where even genera are blue, and odd in yellow. 
Each blue genus is either of the form $\mathcal{G}_{n,p}$, or equal to 
the (already known) genus\footnote{We may think of $\mathcal{G}_n$ as $\mathcal{G}_{n,1}$.} of even Euclidean lattices of determinant $1$ (case $n$ even) or $2$ (case $n$ odd). \ps

\item[(b)] Each arrow $\mathcal{G} \overset{t}{\rightarrow} \mathcal{G}'$ in Figure~\ref{fig:imageliensgenres} 
means that we shall deduce, 
from a set $\mathcal{L}$ of representatives for the isometry classes in $\mathcal{G}$, 
a set $\mathcal{L}'$ of representatives of that in $\mathcal{G}'$, 
using the vectors in each $L \in \mathcal{L}$ of a certain type determined by the label $t$. \ps
\item[(c)] Each label $t$ as in (b) consists of an integer $\nu$ and an attribute $\chi$ 
which is either $\texttt{char}$, $\texttt{exc}$, $\texttt{sp}$ or empty. 
A vector of {\it type $t$} in a lattice $L$ in $\mathcal{G}$ is then defined as 
a primitive\footnote{
As $\nu$ is always square free, the primitive condition is automatic.} 
vector of $L$ of norm $\nu$, which is respectively 
characteristic (\S \ref{subsect:charvec}), exceptional (Definition~\ref{def:except}), 
special (Definition~\ref{def:spevec}), or with no extra property.
We also write $t=(\nu,\chi)$.
\end{itemize}

The general algorithm is as follows. 
Start with an arrow $\mathcal{G} \overset{t}{\rightarrow} \mathcal{G}'$ in Figure~\ref{fig:imageliensgenres} 
and a subset $\mathcal{L}$ of $\mathcal{G}$ as in (b). 
Set $\mathcal{L}'=\emptyset$, and for each $L \in \mathcal{L}$, 
do the following: 
\begin{itemize}
{\sff
\item[A1.] Compute the set $S$ of all vectors of type $t$ in $L$,\par
\item[A2.] Compute a set of generators of ${\rm O}(L)$,\par
\item[A3.] Determine representatives $v_1,\dots,v_n$ for the action of ${\rm O}(L)$ on $S$,\par
\item[A4.]  For $i=1,\dots,n$, compute a Gram matrix for $v_i^\perp \cap L$, and add it to $\mathcal{L}'$.
}
\end{itemize}

\begin{prop}
\label{prop:algoLLprime}
For each $\mathcal{G} \overset{t}{\rightarrow} \mathcal{G}'$ and $\mathcal{L}$ as above, 
the algorithm above returns 
a set $\mathcal{L}'$ of representatives for the isometry classes in $\mathcal{G}'$.
\end{prop}

\begin{pf}
In order to prove this Proposition, we apply: \begin{itemize}
\item[(i)]  Example~\ref{ex:orthospecial} with $(d,m)=(\nu,1)$ 
in the case $t=(\nu,\emptyset)$ (in all cases we have that \(\det \mathcal{G}\) and \(\nu\) are coprime 
so the modulus condition is automatically satisfied), \par
\item[(ii)] Example~\ref{ex:charequivalencecharvectors} with $d=\nu$ 
for $t=(\nu,\texttt{char})$ or $t=(\nu,\texttt{exc})$, \par
\item[(iii)] Proposition~\ref{prop:d1d2} with $d_1= \det \mathcal{G}$ and $d_2=\det \mathcal{G}'$ 
in the case $t=(d_1d_2,\texttt{sp})$. \par
\end{itemize}
Here we have denoted by $\det \mathcal{G}$ 
the common determinant of any $L$ in $\mathcal{G}$. 
Each of these statements provides an equivalence between 
the groupoids $\mathcal{A}$ of pairs $(L,v)$ 
with $L$ in $\mathcal{G}$ and $v$ in $L$ of type $t$, 
and that $\mathcal{B}$ of pairs $(N,w)$ 
with $N \in \mathcal{G}'$ and $w \in \resq N$ having a certain order $o>1$ 
and with a certain $q={\rm q}(w) \in \Q/\Z$ determined by $t$. 
Those $o$ and $q$ are as follows, 
with $d_1=\det \mathcal{G}$ and $d_2=\det \mathcal{G}'$:

\begin{table}[H]

\tabcolsep=5pt
{\scriptsize \renewcommand{\arraystretch}{1.8} \medskip
\begin{center}
\begin{tabular}{cc|cc|cc|cc}
\multicolumn{2}{c|}{case} & \multicolumn{2}{c|}{(i)} & \multicolumn{2}{c|}{(ii)} & \multicolumn{2}{c}{(iii)}\\
\hline
$o$ & $q$ & $\nu$ & $-\frac{1}{2\nu}$ & $d_2$ & $\frac{1}{2}\frac{d_2-1}{d_2}$ & $d_2$ & $-\frac{d_1}{2d_2}$
\end{tabular}
\end{center}
\caption{Values of $o$ and $q$ in each case}
\label{tab:valuesoq}
}
\end{table}
 We claim that for any given $N \in \mathcal{G}'$, 
there is a $w \in \resq N$, unique up to sign, 
such that $(N,w) \in \mathcal{B}$. 
As the isometry $-1$ of $N$ induces an isomorphism $(N,w) \isomo (N,-w)$, 
it shows that the forgetful functor $\mathcal{B} \rightarrow \mathcal{G}', (N,w) \mapsto N,$ 
induces a bijection on isomorphism classes, concluding the proof.\par
Let us now check the claim.
The existence part would follow for instance case-by-case 
by comparing Tables~\ref{tab:tabd1d2} and~\ref{tab:valuesoq}.
This is a bit tedious, and actually unnecessary, 
since the fact that we did find in each case at least one $(L,v)$ in $\mathcal{A}$, 
hence one $(N,w)$ in $\mathcal{B}$, concludes, 
as all elements in $\mathcal{G'}$ have isomorphic quadratic residues.
For uniqueness, assume $w,w' \in \resq N$ have 
the same order $o$ and satisfy ${\rm q}(w)={\rm q}(w')$. 
As $\resb N$ is cyclic we must have $w'= \lambda w$ for some $\lambda \in \Z$, 
hence $(1-\lambda^2){\rm q}(w) \equiv 0 \bmod \Z$. 
In all cases we have ${\rm q}(w) \equiv \frac{x}{2o} \bmod \Z$ 
for some $x \in \Z$ with ${\rm gcd}(x,2o)=1$, 
hence $\lambda^2 \equiv 1 \bmod 2o$, 
which forces $\lambda \equiv \pm 1 \bmod o$ 
as we have either $o=2$ or $2p$ with $p$ prime, 
hence $w'=\pm w$.
\end{pf}


\begin{remark} 
{\rm (}On Step {\rm {\sff A1}} of the algorithm{\rm )}
\begin{itemize}
\item[--] {\rm (Case (ii))} In order to enumerate 
the characteristic vectors of norm $\nu$ in a unimodular lattice $L$, 
we may choose $\xi \in {\rm Char}\, L$ and
enumerate the norm $\nu$ vectors in the lattice $\Z \xi + 2L$ not belonging to $2L$.
This is especially simple when $\nu<8$ (exceptional case) and ${\rm r}_1(L)=0$, 
since all nonzero vectors of $2L$ have norm $\geq 8$. \par
\item[--] {\rm (Case (iii))}  We may enumerate the special vectors 
by enumerating first the vectors of norm $d_2/d_1$ in $L^\sharp$, 
and apply $x \mapsto d_1 x$ (see Proposition~\ref{prop:d1d2}).\par
\end{itemize}
\end{remark}

In order to perform Step {\sff A3}, 
we may use standard orbit algorithms, such as the one implemented as $\texttt{qforbits}$ in \texttt{PARI/GP}. 
However, we claim that {\it this is unnecessary for all green arrows in Figure~\ref{fig:imageliensgenres} !} 
More precisely:

\begin{prop}
\label{prop:propuniqueorb}
For each green arrow 
$\mathcal{G} \overset{t}{\rightarrow} \mathcal{G}'$ in Figure~\ref{fig:imageliensgenres},
and each $L \in \mathcal{G}$, 
there is at most one ${\rm O}(L)$-orbit of type $t$ vectors of $L$.
\end{prop}

We stress that this proposition, although interesting in itself (see~\S\ref{subsect:uniqueorbintro}), is 
not essential to the proof of Theorem~\ref{thmintro:hdiscptable}.
Indeed, as we shall see below, in those green situations 
the number of type $t$ vectors does not exceed a few thousands, 
so that Step {\sff A3} is actually straightforward for the computer. 
The proof of Proposition~\ref{prop:propuniqueorb} will use 
the methods developed in section~\ref{sect:fertile}.

\begin{example}
\label{ex:g285}
{\rm (}Determination of $\mathcal{G}_{28,5}${\rm )}
{\rm The numbers of exceptional unimodular lattices with no norm $1$ vectors of rank 
$24$, $25$, $26$, $27$, $28$ and $29$ 
are respectively $24$, $0$, $97$, $557$ 
(Borcherds, M\'egarban\'e, \cite[\S 9]{cheuni}), 
$16381$ (\cite{cheuni2}) 
and $2\,721\,152$ (Theorem~\ref{thm:excuni209}). 
By Lemma~\ref{red:excno1}, there are thus exactly
$$24+97+557+16381+2721152=2738211$$
exceptional unimodular lattices of rank $29$, 
hence as many classes in $\mathcal{G}_{28,5}$ by Proposition~\ref{prop:propuniqueorb}, 
providing a satisfactory explanation of this large number in Table~\ref{tab:gnpneven}.
Similarly, $\mathcal{G}_{26,3}$ has $557+97+24=678$ isometry classes, 
and $\mathcal{G}_{25,2}$ has $24+97=121$. 
It is easy to see\footnote{Let $L$ be an odd exceptional unimodular lattice, 
$v \in {\rm Exc}\,L$ with norm $r$ 
and set $M = L \cap v^\perp$. 
The orthogonal projection $L \rightarrow M^\sharp$ induces 
a map ${\rm R}_1(L) \rightarrow {\rm R}_{\frac{r-1}{r}}(M^\sharp)$ which is bijective for $r>2$, 
and a $2 : 1$ surjection for $r=2$. } 
that for $p=5, 3, 2$ and a given $M$ in $\mathcal{G}_{23+p,p}$, 
the number of norm $1$ vectors of the exceptional lattice $L$ 
from which $M$ comes coincides with the number of norm $\frac{p-1}{p}$ vectors in $M^\sharp$, 
multiplied by $2$ for $p=2$. }
\end{example}

\begin{pf} 
(Of Proposition~\ref{prop:propuniqueorb})
For the three arrows in exceptional cases, 
the proposition follows from Theorem~\ref{thm:excsingleorbituni}. 
For the two arrows 
$\mathcal{G} \overset{t}{\rightarrow} \mathcal{G}'$ in Figure~\ref{fig:imageliensgenres} 
labelled by ${\rm A}_n \rightarrow {\rm A}_{n+1}$, 
$\mathcal{G}$ is the genus of even lattices $M$ 
with $\resq M \simeq - \resq {\rm A}_n$, 
and we are in the situation of Corollary~\ref{cor:manpi}.
The result follows in this case from Proposition~\ref{prop:keyuniqueorbitAn} 
below. The green arrow ${\rm A}_5 \rightarrow {\rm A}_6$ is the special case $n=5$ (odd).
The arrow ${\rm E}_6 \rightarrow {\rm E}_7$ could also be dealt
with as a special case of Corollary~\ref{cor:MRcMRc} 
for $R \simeq {\bf E}_6$ and $c \equiv \varpi_1$ (but this case is not new).
\end{pf}

\begin{prop}
\label{prop:keyuniqueorbitAn}
Let $M$ be an even lattice with $\resq M \simeq - \resq {\rm A}_n$ for $n\geq 4$. 
Assume that the natural map ${\rm O}(M) \rightarrow {\rm O}(\resq M)$ is surjective.\footnote{
This is automatic if we have ${\rm O}(\resq {\rm A}_n)=\{ \pm 1\}$,
or equivalently, if we have $n+1=p^r$, $2p^r$, or $2^r$ for some odd prime $p$ and $r\geq 1$.
This holds for $n=4,5$.}
Let $\mathcal{E}$ be the set of primitive $e \in M^\sharp$ with $e \cdot e=${\small $\frac{n+2}{n+1}$}. 
Then ${\rm O}(M)$ has at most two orbits on $\mathcal{E}$, 
and at most one if either $n$ is odd or $M^\sharp$ has no vector of norm {\small $\frac{4}{n+1}$}.\ps
\end{prop}

\begin{pf} 
Fix a positive system on ${\rm A}_n$, consider the associated (fertile) fundamental weight 
$\varpi_1$ and choose an isometry $\eta : \resq M \rightarrow - \resq {\rm A}_n$.
The isometry group of $\resq {\rm A}_n$ acts transitively on the vectors  $w \in \resq {\rm A}_n$
with $w \cdot w$ $ \equiv \varpi_1 \cdot \varpi_1 \equiv${\small $\,\frac{n}{n+1}$}\,$\bmod \,2\Z$.
By our assumption on $M$,
we deduce $\mathcal{E} \,=\, {\rm O}(M) \,\cdot\, {\rm Exc}_{\varpi_1,\eta}\,M$.
It is thus enough to determine the number $r$ of 
${\rm O}(M,\eta)$-orbits on ${\rm Exc}_{\varpi_1,\eta}\, M$.\par
Let $(U,\iota)$ be the pair corresponding to $(M,\eta)$ under the gluing construction;
so $U$ is an even (unimodular) lattice and $\iota : {\rm A}_n \rightarrow U$ is a saturated embedding.
Let $\upsilon$ be the embedding $\upsilon_{\varpi_1} : {\rm A}_n \rightarrow {\rm A}_{n+1}$, 
let $\mathcal{E}'$ be the set of saturated embeddings $\iota' : {\rm A}_{n+1} \rightarrow U$ 
with $\iota' \circ \upsilon = \iota$. By Proposition~\ref{prop:orbfibreMc}, ${\rm O}(U,\iota)$ has $r$
orbits on $\mathcal{E}'$. \par

Let $S$ be the irreducible component of the root system of $U$ containing $\iota({\rm A}_n)$,
and denote by $s$ the number of ${\rm W}(S)$-orbits of saturated sublattices of ${\rm Q}(S)$
isometric to ${\rm A}_{n+1}$. By Lemma~\ref{lem:emban1} (applied to $m=n+1$), we have $r \leq s$. 
By Lemma~\ref{lem:emban2} applied to $(R,m)=(S,n+1)$, we either have (a) $s\leq 1$,
or (b) $s=2$, $n$ is even, and either $S \simeq {\bf D}_{n+2}$ or $n=4$ and $S \simeq {\bf E}_7$.
In case (a) we are done, so we may and do assume $r=s=2$ and that we are in case (b). \par
There is a sublattice $\iota({\rm A}_n) \subset Q \subset U$ with $Q \simeq {\rm D}_{n+2}$.
Indeed, this is obvious for $S \simeq {\bf D}_{n+2}$, and it follows from Lemma~\ref{lem:emban2} (ii) in the case $S \simeq {\bf E}_7$, since the orthogonal of a root in ${\rm E}_7$ is isometric to ${\rm D}_6$.
Let $\upsilon' : {\rm A}_n \rightarrow {\rm D}_{n+1}$ denote the isometric embedding associated to any of the fertile weights $\varpi_2,\varpi_{n-2}$ of ${\rm A}_n$: see Example~\ref{ex:fertileex}.
By Lemma \ref{lem:emban3} there exists an isometric embedding ${\rm D}_{n+1} \rightarrow U$ extending $\iota$ via $\upsilon'$.
We have $\varpi_2 \cdot \varpi_2 = \varpi_{n-2} \cdot \varpi_{n-2} = \frac{2(n-1)}{n+1}$.
By Proposition~\ref{prop:orbfibreMc} again (applied to $R={\bf A}_n$, $\eta$, and $c \equiv \varpi_2$ or $\varpi_{n-2}$), there is an element $f \in M^\sharp$ with $f \cdot f = 2 - \frac{2(n-1)}{n+1} = \frac{4}{n+1}$, concluding the proof.
\end{pf}

\begin{lemma} 
\label{lem:emban1}
Let $U$ be a lattice and $m\geq 1$. Two isometric embeddings ${\rm A}_m \rightarrow U$ 
are in the same ${\rm W}(U)^{\pm}$-orbit if and only if their images are.
\end{lemma}

\begin{pf} 
Fix isometric embeddings $f,g :  {\rm A}_m \rightarrow U$ 
and $w \in {\rm W}(U)$ with $w({\rm Im \,g}) = {\rm Im\,f}$. 
We only have to show that there is $w' \in {\rm W}(U)^{\pm}$ such that $w' \circ g = f$. 
Replacing $g$ with $w \circ g$, we may assume ${\rm Im}\,g = {\rm Im}\,f$.
By Lemma~\ref{lem:surjWpm}, we may even assume $f$ and $g$ surjective and $U \simeq {\rm A}_m$.
The result follows since we have $f=w' \circ g$ for some $w' \in {\rm O}(U)= {\rm W}(U)^{\pm}$.
\end{pf}

We have used the following lemma about embeddings of root systems, for which we omit the details: see~\cite[Table 4]{king} for some information and \cite{chetaiapp} for the full details.

\begin{lemma} 
\label{lem:emban2}
Let $R$ be an irreducible root system and $m\geq 4$. 
Then either there is at most one ${\rm W}(R)^{\pm}$-orbit of 
embeddings ${\rm A}_m \rightarrow {\rm Q}(R)$, or there are two and 
we are in one of the following cases:\begin{itemize} 
\item[(i)] $R \simeq {\bf D}_{m+1}$ with $m$ odd, 
and the two orbits are permuted by ${\rm O}({\rm Q}(R))$, 
\item[(ii)] $m=5$ and $R \simeq {\bf E}_7$ 
{\rm (}one orbit with orthogonal ${\rm A}_2$, 
the other ${\rm A}_1 \perp \langle 6 \rangle${\rm )},
\item[(iii)] $m=7$ and $R \simeq {\bf E}_8$, and one of the two orbits is not saturated.
\end{itemize}
\end{lemma}

\begin{lemma} \label{lem:emban3}
  Let \(n \geq 4\).
  Let \(\iota: \mathrm{A}_n \to \mathrm{D}_{n+2}\) and \(\upsilon: \mathrm{A}_n \to \mathrm{D}_{n+1}\) be isometric embeddings.
  Then there exists an extension \(\mathrm{D}_{n+1} \to \mathrm{D}_{n+2}\) of \(\iota\) via \(\upsilon\).
\end{lemma}
\begin{proof}
  This follows from the existence of an isometric embedding \(\mathrm{D}_{n+1} \to \mathrm{D}_{n+2}\) (obvious) and the fact that the set of isometric embeddings \(\mathrm{A}_n \to \mathrm{D}_{n+2}\) consists of a single \(\mathrm{O}(\mathrm{D}_{n+2})\)-orbit (Lemma \ref{lem:emban2}).
\end{proof}

\begin{remark}
\label{rem:brownarrow}
{\rm (}The brown arrow{\rm )} 
{\rm The arrow ${\rm A}_4 \rightarrow {\rm A}_{5}$ is 
the case $n=4$ of Proposition~\ref{prop:keyuniqueorbitAn}. 
We have thus the unique orbit property for type $t$ vectors 
of most lattices in $\mathcal{G}_{28,5}$ by Example~\ref{ex:g285}, 
namely all the $2\,721\,152$ ones coming from rank $29$ exceptional lattices 
with no norm $1$ vectors (note $4/(n+1)=(5-1)/5$). }
\end{remark}

As promised above, we provide in Table~\ref{tab:nbormtvec} below 
some information about the maximal/average number of type $t$ vectors 
of the lattices in the genera at the source of an arrow of Figure~\ref{fig:imageliensgenres}.
We exclude the case of exceptional vectors (detailed in Sect.\ref{sect:excuni}), 
the case of Niemeier lattices  (discussed below), 
and that of roots: an inspection of the Coxeter numbers of root systems 
shows that the number of roots in any lattice of rank $n\geq 16$ is $\leq 2n(n-1)$, 
which is always $\leq 1512$ for $n\leq 28$. 
The table shows that the maximum number of type $t$ vectors 
never exceeds a few thousands, 
showing that Step {\sff A3} is a straightforward task for a computer.

\begin{table}[H]

\tabcolsep=4pt
{\scriptsize \renewcommand{\arraystretch}{1.6} \medskip
\begin{center}
\begin{tabular}{cc|cc|c|c!{\vrule width 1pt}cc|cc|c|c}
$\texttt{dim} $ & $\texttt{det} $ &  \multicolumn{2}{c|}{\texttt{t} } & \(\texttt{avg}\) & $\texttt{max}$
&
$\texttt{dim}  $ & $\texttt{det}  $ &  \multicolumn{2}{c|}{\texttt{t} } & \(\texttt{avg}\) & $\texttt{max}$ \\ \hline
$28$ & $5$ &  $70$ & ${\rm sp}$  & $1958.6$ & $11\,100$ &
$28$ & $5$ &  $30$ & ${\rm sp}$ & $2.5$ & $108$
\\ \hline
$27$ & $6$ & $42$ & ${\rm sp}$ & $2.8$ & $104$ &
$26$ & $1$ &  $10$ & ${\rm char}$ & $1002.6$ & $4424$
\\ \hline
$26$ & $3$ &  $6$ & ${\rm sp}$ & $2.8$ & $6$  &
$25$ & $2$ &  $10$ & ${\rm sp}$ &  $899$ & $2208$
\end{tabular}
\end{center}
\caption{Average \(\texttt{avg}\) and maximum $\texttt{max}$, of the {\it nonzero} numbers of type $\texttt{t} $ vectors over the lattices of rank $\texttt{dim} $ and determinant $\texttt{det} $ in Figure~\ref{fig:imageliensgenres}.}
\label{tab:nbormtvec}
}
\end{table}
The situation is quite different for 
vectors of norm $14$ (resp.\ $6$) in Niemeier lattices: 
there are about respectively $187$ billions (resp.\ $17$ millions) in each Niemeier lattice!
Of course, the isometry groups of a Niemeier lattice $L$ is huge, 
so we expect far fewer orbits, 
but enumerating them with the usual algorithm is not reasonable.
One idea to circumvent this problem is to divide the computation in two steps: 
first determine the orbits of ${\rm O}(L)$ in $L/2L$, and then, 
for a representative $\xi \in L/2L$ of such an orbit, 
determine the ${\rm O}(L; \xi)$-orbits of norm $14$ (resp.\ $6$) vectors in $\xi+2L$. 
For this purpose, the following straightforward variant of Lemma~\ref{lem:ind2ML} 
(see also Formula~\eqref{eq:defM2Le}) is useful: \ps

\begin{lemma}
\label{lem:variantind2ML}
The groupoid of pairs $(L,e)$, 
with $L$ a rank $n$ even unimodular lattice 
and $e \in L/2L$ with $e \cdot e \equiv 2 \bmod 4$, 
is equivalent to that of rank $n$ even lattices $M$ 
with $\resq M\, \simeq \,\resq {\rm A}_1 \perp  - \resq {\rm A}_1$, 
via $(L,e) \mapsto M:={\rm M}_2(L; e)$. In this correspondence, 
we have $L= \{ v \in M^\sharp\,| \, \, v \cdot v \in \Z\}$ 
and $\frac{1}{2}(e+2L)=M^\sharp \smallsetminus M$.
\end{lemma}

The isometry classes of $M$'s above of rank $24$ are easily determined 
by first computing orbits of mod $2$ vectors in Niemeier lattices 
(see \S~\ref{subsect:algolat} (e)): we find only $339$ classes. 
In order to conclude, it only remains to determine, for each such $M$, 
representatives for the ${\rm O}(M)$-orbits of vectors of norm 
$14/4=7/2$ (resp.\ $6/4=3/2$) in $M^\sharp$, and then take 
their orthogonal in $M$. The maximum number of such vectors in $M^\sharp$ 
is now only $32\,384$ (resp.\ $88$), which makes the orbit computation feasible. 
This ends the proof of Theorem~\ref{thmintro:hdiscptable}. $\square$

\begin{remark}
\label{rem:freemasses}
{\rm (}Free masses{\rm )} 
{\rm Assume $\mathcal{G} \overset{t}{\rightarrow} \mathcal{G}'$ 
appears in Figure~\ref{fig:imageliensgenres}. 
Assume also that the lattice $N$ in $\mathcal{G}'$ is obtained as $L \cap v^\perp$ 
with $L$ in $\mathcal{G}$ and $v \in L$ of type $t$. 
Let $s$ be the size of the ${\rm O}(L)$-orbit of $v$ in $L$, 
a quantity usually given by the orbit algorithm in A3. 
Set $e=1$ if $o=2$, $e=2$ otherwise (see Table~\ref{tab:valuesoq}). 
Then the proof of Proposition~\ref{prop:algoLLprime} shows
$\frac{1}{e}|{\rm O}(N)|=|{\rm O}(N,w)|=|{\rm O}(L,v)|=\frac{1}{s} |{\rm O}(L)|$. 
In particular, we obtain for free the masses of all the classes in $\mathcal{G}'$. A
s the Plesken-Souvignier algorithm also allows to find generators of ${\rm O}(L,v)$, 
this also provides an alternative method to determine ${\rm O}(N)$ 
which is sometimes more efficient than the direct one 
({\it e.g.} when $L$ is unimodular).}
\end{remark}

\subsection{The invariant ${\rm BV}_{n,p}$ and other methods}
\label{subsect:checkhdiscp}

	Our aim now is to prove Theorem~\ref{thmintro:hdiscptablebvp}.
The assertion about the black entries of Tables~\ref{tab:gnpneven} and~\ref{tab:gnpnodd} follows
from a straightforward computation.
We shall discuss here some new isometry invariants for the lattices in $\mathcal{G}_{n,p}$.
Let us emphasize that the sharpness of these invariants, 
as for ${\rm BV}$ itself, 
is quite empirical and based on computer calculations. 
We apologize for this, and leave as an important open problem to understand why they work.\ps

The lattices in $\mathcal{G}_{n,p}$ being even,
a natural idea is to consider their depth $4$ invariant ${\rm BV}_4$ of~\S\ref{subsect:markedBV}.
Unfortunately, they typically have
more than $80\,000$ norm $4$ vectors in our ranges for $(n,p)$,
so that the computation of ${\rm BV}_4$ for such a lattice 
is about $40^3=64\,000$ times slower than that of ${\rm BV}_3$ 
for a rank $29$ unimodular lattices for instance. 
An alternative solution would be to apply (variants of) ${\rm BV}$ to short vectors in the dual lattices. 
We focus here on another (but related) method, 
whose idea is to first glue to an odd unimodular lattice of rank $\leq 29$, 
for which we already know that ${\rm BV}_3$ is sharp by Corollary~\ref{cor:BVsharpuni}, 
and then apply a marked ${\rm BV}$ invariant of depth $3$ 
introduced in Definition~\ref{def:markedhbv}.  \ps
	
	Consider for this the following Table~\ref{tab:generap2p357}, 
in which we have set $S_5=$\scalebox{.5}{$\left[\begin{array}{cc} 2 & 1 \\ 1 & 3\end{array}\right]$} 
and $S_7=$\scalebox{.5}{$\left[\begin{array}{ccc} 2 & 1 & 1 \\ 1 & 2 & 1 \\ 1 & 1 & 3 \end{array} \right]$} 
(two odd lattices with respective determinant $5$ and $7$).

\begin{table}[H]
\tabcolsep=6pt
{\scriptsize \renewcommand{\arraystretch}{1.7} \medskip
\begin{center}
\begin{tabular}{c|c|c|c|c|c|c|c|c}
$p\, \, \, \backslash\,\,\, n \bmod 8$  & $1$ & $2$ & $3$ & $4$ & $5$ & $6$ & $7$ & $8$ \\
\hline
$3$ & $\langle 2 \rangle \perp \langle 3 \rangle$ & $\langle 3 \rangle$ & $\langle 2 \rangle \perp \langle 3 \rangle$ & & $\langle 6 \rangle$  & \cellcolor{gray!30} ${\rm A}_2$ & $\langle 2 \rangle \perp {\rm A}_2 $ & \\
\hline
$5$ & $\langle 10 \rangle$ & & $\langle 2 \rangle \perp \langle 5 \rangle$ & $\langle 5 \rangle$ & $\langle 2 \rangle \perp \langle 5 \rangle$ & & $\langle 2 \rangle \perp S_5$ & $S_5$ \\
\hline
$7$ & $\langle 2 \rangle \perp S_7 $  & $S_7$ & $\langle 2 \rangle \perp S_7 $ & & $\langle 14 \rangle$ & $\langle 7 \rangle$ & $\langle 2 \rangle \perp \langle 7 \rangle$ &

\end{tabular}
\end{center}
}
\caption{\small Some lattices whose bilinear residue is {\it opposite} to that of $\mathcal{G}_{n,p}$}
\label{tab:oppres357}
\end{table}

\begin{fact}
Each $(n,p)$-entry of Table~\ref{tab:oppres357} is
a lattice $A$ (of rank $\leq 3$) satisfying $\resb A \simeq - \resb L$
for any $L \in \mathcal{G}_{n,p}$.
\end{fact}

\begin{pf}
This is a simple case-by-case verification
using the discussion of~\S\ref{subsect:residuegnp}.
It is useful to observe that for $n$ even and $\mathcal{G}_{n,p} \neq \emptyset$,
then the lattices in $\mathcal{G}_{n-1,p}$ and $\mathcal{G}_{n+1,p}$
have isomorphic {\it bilinear} residues $\resb L \perp \resb {\rm A}_1$,
with $L \in \mathcal{G}_{n,p}$.
\end{pf}

Fix $(n,p)$ and $A$ as above, and choose $L$ in $\mathcal{G}_{n,p}$.
By the gluing construction (Proposition~\ref{prop:nikulingroupoids}),
the choice of an isometry $\eta : -\resb A \isomo \resb L$
defines a unique unimodular overlattice $U$ of $L \perp A$.
The isomorphism class of the pair $(L,\eta)$
uniquely determines that of $(U,\iota)$,
with $\iota : A \rightarrow U$ the natural inclusion.
Since the natural morphism
${\rm O}(L) \rightarrow {\rm O}(\resb L)$ is surjective
by the trivial equality ${\rm O}(\resb L)=\{ \pm 1\}$,
the isomorphism class of $(U,\iota)$ actually
only depends on the isometry class of $L$: we denote it by $\widetilde{L}$.
\begin{definition}
\label{eqdefBV1np}
For $L \in \mathcal{G}_{n,p}$ we set ${\rm BV}^1_{n,p}(L)={\rm BV}_3(U,\iota)$, 
with $(U,\iota)=\widetilde{L}$.
\end{definition}

Note that unless $p=3$ and $n \equiv 6 \bmod 8$ 
(the grey cell in Table~\ref{tab:oppres357}), the unimodular lattice $U$ is odd,
since its rank is $\not \equiv 0 \bmod 8$.
If ${\rm rank}\, U \leq 29$, we know that ${\rm BV}_3(U)$
is a sharp (and fast to compute) invariant of $U$ by Corollary~\ref{cor:BVsharpuni},
so it is tempting to hope that ${\rm BV}^1_{n,p}(L)$ 
is a sharp invariant\footnote{
The marked invariant is only marginally slower to compute than the unmarked one, 
so we refer to \cite{cheuni2} and \S \ref{sect:classrk29} for indications about computation times.}
of $L$ as well. Note that there is no red entry $(n,3)$ 
with $n\equiv 6 \bmod 8$ in Tables~\ref{tab:gnpneven} \&~\ref{tab:gnpnodd}.

\begin{fact}
\label{fact:checkinv1}
The invariant ${\rm BV}^1_{n,p}$ is sharp on $\mathcal{G}_{n,p}$ for all
red $(n,p)$ in Table~\ref{tab:gnpneven} (so $n$ even), 
and all red $(n,3)$ in Table~\ref{tab:gnpnodd}.
\end{fact}

\begin{pf} 
It follows from a direct computer calculation. 
See~Example~\ref{ex:check273} for a few CPU time indications.
\end{pf}

For the case \((n,p)=(18,7)\) this invariant is overkill: there are only two isometry classes with the same isomorphism class of root system (namely \(\mathbf{A}_{17}\)), and they are distinguished by their numbers of vectors of norm \(4\).

\begin{example} 
{\rm ({\it Case $(n,p)=(28,5)$})
For $L \in \mathcal{G}_{28,5}$ and $\widetilde{L}=(U,\iota)$, 
the lattice $U$ is the exceptional unimodular lattice of rank $29$ associated to $L$, 
so we did expect ${\rm BV}^1_{28,5}$ to be sharp on $\mathcal{G}_{28,5}$ 
by Theorem~\ref{thm:excsingleorbituni}. }
\end{example}


Although a computation shows that the invariant ${\rm BV}^1_{n,p}$ 
is also sharp for some other cases of $(n,p)$, 
such as $(19,5)$, $(21,5)$, $(19,7)$, $(25,7)$ and $(27,5)$, 
it fails to be so in all cases. For instance, 
the $1396$ lattices in $\mathcal{G}_{23,5}$ 
only have $1370$ distinct ${\rm BV}^1$ invariants. 
Our aim now is to define two variants ${\rm BV}^2$ and ${\rm BV}^3$ of ${\rm BV}^1$, 
actually faster to compute,\footnote{
Having a faster invariant can be useful 
for instance if anyone wants to compute a Hecke operator on $\mathcal{G}_{n,p}$.}, 
the combination of which will eventually allow 
to provide sharp invariants in all cases except $(23,5)$. \ps

Assume $n$ is odd, $p \leq 7$ and set $m=n+1$ or $m=n-1$ 
so that $m+p \equiv 1 \bmod 4$.
Instead of gluing $L \in \mathcal{G}_{n,p}$ 
with the lattice $A$ in the $(n,p)$-entry of Table~\ref{tab:oppres357},
we may rather glue it either with the lattice in the $(m,p)$-entry of this table, 
or with the ${\rm A}_1$ lattice.
In the first (resp.\ second) case, we denote by $\widetilde{L}^{'}$ (resp.\ $\widetilde{L}^{''}$) 
the resulting pair $(V,\iota)$.
The lattice $V$ has determinant $2$ (resp.\ $p$). 
For $p=5,7$, $V$ is always odd in the first case, 
as well as in the second case for $n \equiv p \bmod 4$.

\begin{definition}
\label{eqdefBV23np}
For $n$ odd, $p \leq 7$ and $L \in \mathcal{G}_{n,p}$, we set 
${\rm BV}^2_{n,p}(L)={\rm BV}_3(\widetilde{L}^{'})$ 
and ${\rm BV}^3_{n,p}(L)={\rm BV}_3(\widetilde{L}^{''})$.
\end{definition}

For $i=(i_1,\dots,i_k) \in \{1,2,3\}^k$ and $L \in \mathcal{G}_{n,p}$ 
we also denote by ${\rm BV}_{n,p}^i(L)$ the $k$-uple
$({\rm BV}^{i_1}_{n,p}(L), \dots,{\rm BV}^{i_k}_{n,p}(L))$. 
A computer calculation shows then:

\begin{fact}
\label{fact:checkinvi}
For each black (resp.\ red) entry $i$ in the box $(n,p)$ of Table~\ref{tab:bvnpnodd}, 
the invariant ${\rm BV}^i_{n,p}$ is sharp (resp.\ not sharp) on $\mathcal{G}_{n,p}$.
\end{fact}
\vspace{-.5cm}
\begin{table}[H]
\tabcolsep=6pt
{\scriptsize \renewcommand{\arraystretch}{1.7} \medskip
\begin{center}
\begin{tabular}{c|c|c|c|c|c}
$p \, \backslash \, n $ & $19$ & $21$ &  $23$ & $25$ & $27$ \\
\hline
$3$ & & & $1,{\color{red} 2},3$ & $1,2,{\color{red} 3}$ & $1, 2, {\color{red} 3}$ \\
\hline
$5$ & $1, 2, 3$ & $1, 2, 3$ & ${\color{red} (1,2,3)}$ & ${\color{red} 1}, {\color{red} 2}, {\color{red} 3}, (1,2), (1,3), (2,3)$ & $1, {\color{red} 2},{\color{red} 3}, (2,3)$  \\
\hline
$7$ & $1, 2, {\color{red} 3}$ & $2, {\color{red} (1, 3)}$ &${\color{red} 1}, 2, {\color{red} 3}, (1, 3)$ & $1, {\color{red} 2}, {\color{red} 3}, (2,3)$ &  \\
\end{tabular}
\end{center}
}
\caption{\small
Sharpness of the invariants ${\rm BV}^i_{n,p}$ 
for the red entries $(n,p)$ in Table~\ref{tab:gnpnodd}}
\label{tab:bvnpnodd}
\end{table}

The case \((n,p)=(17,7)\) is not included in this table because as for \((18,7)\) it is much easier: there are only two isometry classes in this genus having the same isomorphism class of root system (namely \(\mathbf{A}_{15}\)), and they are distinguished by their number of vectors of norm \(4\).

\begin{example} 
{\rm ({\it Case $(n,p)=(25,5)$}) 
\label{ex:case255}
The invariant ${\rm BV}^1$ falls short of being sharp in this case: 
there are $38\,749$ lattices but $38\,746$ different ${\rm BV}^1$ invariants, 
hence exactly $3$ pairs of ambiguous lattices. 
For only one of these pairs the two lattices have the same root system, 
namely ${\bf A}_1\, {\bf A}_2 \,2{\bf A}_3\, {\bf A}_9 \,{\bf D}_4$! 
For each of these $3$ pairs, the two lattices are distinguished 
both by ${\rm BV}^2$ and ${\rm BV}^3$. 
Similar exceptional behaviors do repeat for some other values of $(n,p)$, 
which shows that we may have been quite lucky 
that the assertion about ${\rm BV}$ holds in Theorem~\ref{thm:uni29no1}.
Note also that for $(25,5)$, 
the average computation time of ${\rm BV}^1$, ${\rm BV}^2$ and ${\rm BV}^3$ 
are respectively $330\,\texttt{ms}$, 27\,\texttt{ms}, $17\,\texttt{ms}$.}
\end{example}

At this point, the only remaining genera 
for which we do not have provided any isometry invariant 
is $\mathcal{G}_{23,5}$, since ${\rm BV}^{(1,2,3)}$ is not sharp in this case by Fact~\ref{fact:checkinvi}.
We now treat this case in a ad hoc way.

\begin{remark}
\label{rem:case235}
 {\rm ({\it Case $(n,p)=(23,5)$})
A computation shows that ${\rm BV}^2$ is close to be sharp here: the
$1\,396$ lattices have $1\,394$ distinct ${\rm BV}^2$ invariants. 
Unfortunately, the two pairs with the same ${\rm BV}^2$ 
have a common root system ($4{\bf A}_1\, 3{\bf A}_2\, 2{\bf A}_3$ for one pair 
and $2{\bf A}_1 \,{\bf A}_2 \,2{\bf A}_3\, {\bf A}_4 \,{\bf A}_5$ for the other), 
as well as the same ${\rm BV}^1$ and ${\rm BV}^3$. 
We may distinguish those last lattices in the following ad hoc way. 
For $L$ in $\mathcal{G}_{23,5}$, the average number of vectors of norm $n \leq 30$ 
of the {\it rescaled dual lattice} $L^\flat:=\sqrt{ \det L}\, L^\sharp$ 
is given by Table~\ref{tab:nbvecscaldual235}. We checked that 
the invariant $f(L):={\rm BV}_{24}(L^\flat)$ does distinguish the two lattices in each pair, 
{\it if we compute ${\rm BV}$ using the absolute variant} described in Remark~\ref{rem:varianthbv}.}
\begin{table}[H]
\tabcolsep=6pt
{\scriptsize \renewcommand{\arraystretch}{1.7} \medskip
\begin{center}
\begin{tabular}{c|c|c|c|c|c|c|c|c|c}
$n $ & $4$ & $11$ &  $15$ & $16$ & $19$ & $20$ & $21$ & $24$ & {\rm other $\leq 30$} \\
\hline
$\# $ & $0.1$ & $2.4$ &  $8.0$ & $24.1$ & $85.4$ & $134.4$ & $21$ & $670.4$ & $0$
\end{tabular}
\end{center}
}
\caption{\small
Average number $\#$ of norm $n$ elements in $\sqrt{\det L}\,\,L^\sharp$ for $L \in \mathcal{G}_{23,5}$.
}
\label{tab:nbvecscaldual235}
\end{table}
\end{remark}

We finally define ${\rm BV}_{n,p}$.
Assume the $(n,p)$-box is red in Tables~\ref{tab:gnpneven} or~\ref{tab:gnpnodd}.
If \(p=7\) and \(n \in \{17,18\}\) simply define \(\mathrm{BV}_{n,p}(L)\) as the pair formed of the isomorphism class of the root system of \(L\) and \(\mathrm{r}_4(L)\).
Now assume that \((n,p) \not\in \{(17,7),(18,7)\}\).
If $n$ is even, set ${\rm BV}_{n,p}={\rm BV}^1_{n,p}$.
If $n \geq 19$ is odd and $(n,p) \neq (23,5)$, define ${\rm BV}_{n,p}$ as ${\rm BV}^i_{n,p}$
for any black entry $i$ in the $(n,p)$-box of Table~\ref{tab:bvnpnodd}.
If we write $a \prec b$ to mean that $a$ is faster than $b$, 
we usually have ${\rm BV}^3 \prec {\rm BV}^2 \prec {\rm BV}^{(2,3)} \prec {\rm BV}^1$, 
hence a best choice for $i$ given by Table~\ref{tab:bvnpnodd}.
Finally, for $(n,p)= (23,5)$, set ${\rm BV}_{n,p}=({\rm BV}^2,f)$ as in Remark~\ref{rem:case235}.
From Facts~\ref{fact:checkinv1} \& \ref{fact:checkinvi} and that remark, we deduce:

\begin{cor}
Theorem~\ref{thmintro:hdiscptablebvp} holds for the definition above of ${\rm BV}_{n,p}$.
\end{cor}

	The first main application of the invariants ${\rm BV}_{n,p}$ is that 
they allow an independent verification that 
our lists of lattices in Theorem~\ref{thmintro:hdiscptable} are complete, 
using the mass formula.
	
\begin{example}
\label{ex:check273}
{\rm ({\it Independent check that the list of $285\,825$ representatives 
of $\mathcal{G}_{27,3}$ given in {\rm \cite{chetaiweb}} is complete})
The computation of all ${\rm BV}^1_{27,3}$ invariants takes about $6\,\texttt{h}\,45\,\texttt{min}$, 
{\it i.e.} about $85\,\texttt{ms}$ per lattice (and they are all distinct). 
Moreover, applying to the given Gram matrices the variant \texttt{qfautors} 
of the Plesken-Souvignier algorithm mentioned in~\S\ref{subsect:algolat} (c), 
or PARI's \texttt{qfauto} for the $9$ lattices with no roots,\footnote{
The order of the isometry groups of the $9$ lattices with no roots are 
$18\,720\,000$, $5\,760$, $4\,608$, $1\,152$, $240$, $120$, $48$, $48$, $48$. 
Their direct computation is a bit lengthy ($1\,\texttt{h}$). 
However, it only takes a few seconds if we rather use 
the associated rank $29$ unimodular lattices obtained 
by gluing with $\langle 2 \rangle \perp \langle 3 \rangle$ ({\it i.e.} 
if we rather compute $|{\rm O}(\widetilde{L})|=|{\rm O}(L)|/2$).  }  
the computation of the order of all isometry groups takes about $8\,\texttt{h}\,45\,\texttt{min}$, 
{\it i.e.} about $110\,\texttt{ms}$ per lattice 
(and actually, about $85\,\texttt{ms}$ for about $99.8$\% of the lattices). 
Statistics for reduced isometry groups of order $\leq 256$ are 
given in Table~\ref{tab:statroadim27det6} 
(only $509$ lattices have a larger reduced isometry group). 
The total mass of our lattices coincides with the mass formula of $\mathcal{G}_{27,3}$.\footnote{
It is
{\tiny
$184361388591800313635423567792726086296697/6214940800321288874910535133429760000000$.}
}\begin{table}[H]
\tabcolsep=2pt
{\scriptsize \renewcommand{\arraystretch}{1.3} \medskip
\begin{center}
\begin{tabular}{c|c|c|c|c|c|c|c|c|c|c|c|c|c|c|c}
$\texttt{ord}$ & $2$ & $4$ & $6$ & $8$ & $10$ & $12$ & $16$ & $18$ & $20$ & $24$ & $32$ & $36$ & $40$ & $42$ & $48$ \\
$\#$ & $225\,451$ & $44\,125$ & $707$ & $8\,775$ & $11$ & $1\,482$ & $1\,922$ & $5$ & $31$ & $896$ & $673$ & $19$ & $50$ & $3$ & $398$ \\ \hline
$\texttt{ord}$ & $64$ & $72$ & $80$ & $84$ & $96$ & $104$ & $108$ & $120$ & $128$ &  $144$ & $160$ & $192$ & $216$ & $240$ & $256$ \\
$\#$ & $135$ & $58$ & $24$ & $10$ & $261$ & $1$ & $2$ & $11$ & $64$ & $57$ & $8$ & $72$ & $19$ & $32$ & $14$
 \end{tabular}
\caption{\small Number $\#$ of lattices in $\mathcal{G}_{27,3}$ with reduced isometry group of order $\texttt{ord} \leq 256$.}
\label{tab:statroadim27det6}
\end{center}
}
\end{table}
}
\end{example}
The invariants ${\rm BV}_{n,p}$ also allow 
to use alternative methods to compute some genera $\mathcal{G}_{n,p}$. 
For instance, the case $(23,3)$ can be easily dealt with 
using Kneser neighbors and ${\rm BV}_{23,3}$, 
and this is how we first determined it.
\begin{example} 
{\rm ({\it Alternative determination of $\mathcal{G}_{27,3}$})
We initially determined $\mathcal{G}_{27,3}$ using a ``backward''  method, 
represented by the dotted arrow $\mathcal{G}_{26,3} \dashrightarrow \mathcal{G}_{27,3}$ 
in Figure~\ref{fig:imageliensgenres}, 
in the spirit of the method used in \S\ref{sect:classrk29} to deduce the rank $29$ unimodular lattices 
from rank $27$ ones. Indeed, a variant of Proposition~\ref{prop:uninplus2n} shows that
{\it the groupoid of pairs $(M,\alpha)$, 
with $M \in \mathcal{G}_{27,3}$ and $\alpha$ a root of $M$ with ${\rm m}(\alpha)=1$, 
is equivalent to that of pairs $(N,e)$ 
with $N \in \mathcal{G}_{26,3}$ and $e \in N/2N$ with $e \cdot e\,\equiv \,2\, \bmod \,4$}. 
Using \texttt{orbmod2}, the known classification of $\mathcal{G}_{26,3}$, and ${\rm BV}_{26,3}$, 
this allows one to determine all isometry classes in $\mathcal{G}_{27,3}$ with nonempty root system.
The lattices $M \in \mathcal{G}_{27,3}$ having a root $\alpha$ with ${\rm m}(\alpha)=2$ 
are precisely those of the form
${\rm A}_1 \perp N$ with $\mathcal{G}_{26,3}$.
The $9$ remaining lattices without root in $\mathcal{G}_{27,3}$ 
were found using Kneser neighbors
(their total mass is $951709/12480000$ by \cite[Prop. 6.5]{cheuni} and \cite{king}).}
\end{example}

\section{Rank $30$ unimodular lattices with few roots}
\label{sect:rank30}

${}^{}$ Our aim in this last section is to discuss the proof of Theorem~\ref{thmintro:X30}.
This is a massive computation, which required more than $100$ years\footnote{Precise CPU time is 
difficult to estimate (and may have been significantly higher than stated), as our parallel implementation was not fully optimized.} of CPU time (single core equivalent).
Using the observed sharpness of the BV invariant and the Plesken-Souvignier algorithm, 
the completeness of these lists can be easily checked 
independently of the way we found them;
given our current implementation, it takes less than $3.5$ years (see Remark~\ref{rem:X30CPUtime}). 
Indeed, for each root system $R$, we know from the work of King \cite{king} 
(see also~\cite[\S 6.4]{cheuni}) the reduced mass ${\rm m}_{30}(R)$ of ${\rm X}_{30}^R$: see Table~\ref{tab:redmassX30}.
Again, King's lower bounds for $2 \,{\rm m}_{30}(R)$ were 
not too far from the actual size of ${\rm X}_{30}^R$ in the three 
cases considered above. 
\vspace{-.4cm}
\begin{table}[H]
{\scriptsize \renewcommand{\arraystretch}{1.5} \medskip
\begin{center}
\begin{tabular}{c|c|c}
$R$ & ${\rm m}_{30}(R)$ & $\approx 2 {\rm m}_{30}(R) \cdot 10^{-6}$ \\
\hline
$\emptyset$ & $7180069576834562839/175111372800$ & $82.01$ \\
${\bf A}_1$ & $9242148948311/51840$ & $356.56$ \\
${\bf A}_2$ & $25436628608581/4043520$ & $12.58$
\end{tabular}
\caption{\small The reduced mass ${\rm m}_{30}(R)$, and $2 \,{\rm m}_{30}(R)$ in millions (rounded to $10^{-2}$).}
\label{tab:redmassX30}
\end{center}
}
\end{table}
\vspace{-.5cm}

We will not give many details about how we found the 
lists in Theorem~\ref{thmintro:X30}, as the method is close 
to the one described in details in~\cite{cheuni2} for the 
classification of ${\rm X}_{29}^\emptyset$. We will content 
ourselves with giving an overview of the main steps, assuming 
the reader is familiar with \cite{cheuni,cheuni2}, and 
to emphasize some novel difficulties we encountered 
in the two cases $R=\emptyset$ and $R = {\bf A}_1$.
The various improvements in lattice algorithms described in \S~\ref{subsect:algolat},
as well as the notion of visible isometries explained in \cite[\S 7]{cheuni},
were of great help in these new computations. \ps

\begin{remark}
\label{rem:X30CPUtime} 
Assume $L \in {\rm X}_{30}$ has no norm $1$ vectors. 
We have\footnote{It follows from similar arguments as in {\rm \cite[\S 4]{bachervenkov}}, 
and from Remark~\ref{rem:rel_def_exc}.}
${\rm r}_3(L)\,=\,1520 \,+\, 12 \,{\rm r}_2(L) \,-\, 64\, |{\rm Exc}\,L|$,
the possible values of $|{\rm  Exc}\, L|$ being given by Remark~\ref{rem:dim6mod8}. 
In the case ${\rm R}_2(L)=\emptyset$ {\rm (}resp. ${\bf A}_1$, ${\bf A}_2${\rm )}, it follows that
the number of vertices of the graph $\mathcal{G}_{\leq 3}(L)$
{\rm (}see~\S~\ref{subsect:markedBV}{\rm )} is bounded above by $760$ {\rm (}resp. $773$, $799${\rm )}.
Perhaps surprisingly, these quantities are slightly smaller than 
their $29$-dimensional analogues discussed at the end of Sect.~\ref{sect:classrk29} 
{\rm (}the second column of Table 1 in {\rm \cite{nebevenkov_shadow}} gives an idea 
of how these quantities vary with the rank{\rm )}. 
This makes the sharpness of ${\rm BV}(L)$ even more remarkable here, 
and its computation faster: it runs in about $63$\,{\rm \texttt{ms}}.
By comparison, it takes about $30$\,{\rm \texttt{ms}} 
to find a good Gram matrix for such a lattice $L$,
and then about $250$\,{\rm \texttt{ms}} {\rm (}resp. $120$\,{\rm \texttt{ms}}, $82$\,{\rm \texttt{ms}}{\rm )}
to compute $|{\rm O}(L)^{\rm red}|$.
\end{remark}

\subsection{Case $R= \emptyset$}
\label{subsect:X30vide}
We started with an exploration of the $d$-neighbors 
of ${\rm I}_{30}$ having an empty visible root system for $d$ ranging from $61$ to $147$,
following the \texttt{BNE} algorithm described in \cite[\S 5.4]{cheuni2}.
The number of new lattices we found for each $d$ is indicated 
in Table~\ref{tab:nbneiX30empty} below. Up to $d=83$, we enumerated {\it all} 
the $d$-neighbors\footnote{More precisely, all those defined by 
an isotropic line with some coordinate coprime to $d$.} of ${\rm I}_{30}$, 
but we stopped doing so from $d=84$ for efficiency reasons, 
preferring to increase $d$ when the algorithm 
started to yield fewer new lattices. From $d=96$ to $d=147$, 
we only selected about $10^7$ isotropic vectors 
(and ceased selecting minimal vectors 
in a line as in~\cite[Rem. 5.9]{cheuni2}, which is ineffective here).

\begin{table}[h]
\centering
\tiny
\setlength{\tabcolsep}{2pt} 
\renewcommand{\arraystretch}{1.1} 
\begin{tabular}{l r | l r | l r | l r | l r | l r | l r | l r| l r}
\texttt{d} & \#  & \texttt{d} & \# & \texttt{d} & \#  & \texttt{d} & \#  & \texttt{d} & \#  & \texttt{d} & \# & \texttt{d} & \# & \texttt{d} & \# & \texttt{d} & \# \\
 & & & & & & & & & & & & & & & & & \\
$61$ & $1$& $71$ & $130$& $81$ & $260\,291$& $91$ & $9\,124\,548$& $101$ & $289\,484$& $111$ & $174\,694$& $121$ & $58\,073$& $131$ & $18\,067$& $141$ & $1\,192$ \\
$62$ & $0$& $72$ & $1\,177$& $82$ & $907\,179$& $92$ & $8\,053\,421$& $102$ & $323\,373$& $112$ & $124\,628$& $122$ & $51\,277$& $132$ & $10\,797$& $142$ & $911$\\
$63$ & $0$& $73$ & $752$& $83$ & $638\,350$& $93$ & $4\,271\,277$& $103$ & $369\,491$& $113$ & $143\,898$& $123$ & $51\,319$& $133$ & $11\,486$& $143$ & $877$\\
$64$ & $0$& $74$ & $4\,986$& $84$ & $3\,378\,682$& $94$ & $2\,068\,536$& $104$ & $377\,261$& $114$ & $84\,183$& $124$ & $46\,145$& $134$ & $6\,407$& $144$ & $623$\\
$65$ & $4$& $75$ & $5\,678$& $85$ & $1\,846\,743$& $95$ & $1\,524\,120$& $105$ & $369\,945$& $115$ & $53\,617$& $125$ & $45\,977$& $135$ & $5\,786$& $145$ & $580$\\
$66$ & $4$& $76$ & $20\,249$& $86$ & $6\,085\,993$& $96$ & $137\,059$& $106$ & $343\,827$& $116$ & $69\,709$& $126$ & $35\,771$& $136$ & $3\,613$& $146$ & $526$\\
$67$ & $3$& $77$ & $18\,940$& $87$ & $5\,008\,910$ & $97$ & $145\,590$& $107$ & $384\,739$& $117$ & $67\,457$& $127$ & $37\,791$& $137$ & $4\,035$& $147$ & $243$\\
$68$ & $27$& $78$ & $103\,979$& $88$ & $12\,014\,052$& $98$ & $181\,599$& $108$ & $267\,665$& $118$ & $62\,403$& $128$ & $27\,466$& $138$ & $2\,268$\\
$69$ & $30$& $79$ & $57\,901$& $89$ & $6\,972\,717$& $99$ & $190\,195$& $109$ & $301\,591$& $119$ & $64\,000$& $129$ & $26\,221$& $139$ & $2\,329$\\
$70$ & $282$& $80$ & $320\,713$& $90$ & $14\,040\,632$& $100$ & $256\,585$& $110$ & $180\,120$& $120$ & $53\,835$& $130$ & $17\,898$& $140$ & $1\,181$
 \end{tabular}
\caption{{\scriptsize Number \# of new lattices found in ${\rm X}_{30}^\emptyset$ as $d$-neighbors of ${\rm I}_{30}$}}
\label{tab:nbneiX30empty}
\end{table}
 
After this massive computation, about $10^5$ lattices remained to be found, but 
it did not seem reasonable to pursue this strategy further.
Note that a specific search for exceptional lattices using the method described in~\cite[\S 9.3]{cheuni}
led to the discovery of only about $4\,000$ new lattices, most candidates having already been found.
At this point, the remaining mass was 
{\small
$$1593528554589611/M\,\,\,{\rm with}\,\,\,M=35022274560=29 \cdot 13 \cdot 7 \cdot 5 \cdot 3^4 \cdot 2^{15}.$$
}\par
\noindent In order to "clean" the denominator, we then searched for neighbors 
having an isometry of prime order $p \, |\, M$ and a prescribed characteristic polynomial, 
using the method of visible isometries described in \cite[\S 7]{cheuni} 
(see also \cite[\S 6.7]{cheuni2} for an example). We stopped after finding 
about $7\,600$ new lattices, leaving a remaining mass of $13033918217/M'$ 
with $M'=3^2 \cdot 2^{15}$. For instance, we found a lattice with mass $1/232$ for $p=29$, 
and two lattices with masses $1/8736$ and $1/134784$ for $p=13$ 
(with the characteristic polynomial $\Phi_{13}^2 \Phi_1^{6}$). \ps
 
To complete the classification, we then computed the $2$-neighbors of 
a suitably chosen collection $\mathcal{C}$ of already found lattices in ${\rm X}_{30}$.
As explained in \cite[\S 7.5]{cheuni}, the $2$-neighbors of a lattice with a large isometry group 
are good candidates for having a non-trivial (or large) isometry group; conversely, those of a 
lattice with a trivial isometry group are random enough to be useful
in the search for the (many) missing lattices with mass $1/2$. 
We therefore used both kinds of lattices in our choice of $\mathcal{C}$. \ps
More precisely, we first included $40$ lattices in $\mathcal{C}$, each given as a $d$-neighbor 
of ${\rm I}_{30}$ for odd $d$ between $65$ and $147$, and we computed 
{\it all} the $2$-neighbors of those lattices. We also performed a partial computation 
of the $2$-neighbors of the lattice with mass $1/96$ found for $d=65$, 
and of $4$ lattices with mass $1/32$ found for $d=81$ and $85$, 
in order to start hunting the powers of $2$ and $3$ in $M'$ 
(but this turned out to be ineffective). These massive neighbor computations 
allowed us to find more than $95\,000$ new lattices, 
and left a remaining mass of $4289033/294912 \approx 14.5$. 
A posteriori, we know that after this step, 
only $115$ lattices were actually missing, with masses given by Table~\ref{tab:last300lat} 
(note that all lattices but one have an isometry group which is a $2$-group).\par

\begin{table}[h]
\centering
\scriptsize
\setlength{\tabcolsep}{5pt} 
\renewcommand{\arraystretch}{1.5} 
\begin{tabular}{l|cccccccccc}
\texttt{mass} & $1/2$ & $1/4$ & $1/8$ & $1/16$ & $1/32$ & $1/64$ & $1/128$ & $1/512$ & $1/2304$ & $1/32768$ \\
\hline
\# & $1$ & $29$ & $38$ & $25$ & $13$ & $4$ & $1$ & $2$ & $1$ & $1$
\end{tabular}
\caption{{\small Number $\#$ of lattices with mass \texttt{mass} in the last $115$ lattices found}}
\label{tab:last300lat}
\end{table}

We then added to $\mathcal{C}$ about $200$ lattices of the form 
${\rm I}_1 \perp L$ with $L \in {\rm X}_{29}^\emptyset$ and mass $\leq 1/16$,
as well as a few lattices with mass $1/512$, $1/256$ and $1/16$ 
obtained using the visible isometry method.
We found the remaining lattices by computing, for each of them, ``only'' about $10^7$ $2$-neighbors.
For instance, the two lattices in ${\rm X}_{30}^\emptyset$ with mass $1/32768=1/2^{15}$ and $1/2304$
were discovered as $2$-neighbors of the same lattice ${\rm I}_1 \perp L$, 
where $L$ is the unique class in ${\rm X}_{29}^\emptyset$ with mass $1/18432=1/(2^{11} \cdot 3^{2})$.
As an anecdote, the last lattice we found has mass $1/16$.
The final statistics for the isometry groups 
in ${\rm X}_{30}^\emptyset$ are given in Table~\ref{tab:sizeautX30empty}.

\begin{table}[h]
\centering
\tiny
\setlength{\tabcolsep}{2.8pt} 
\renewcommand{\arraystretch}{1.1} 
\begin{tabular}{l r r | l r r | l r r | l r r | l r r | l r r}
\texttt{mass} & \# & \texttt{e} & \texttt{mass} & \# & \texttt{e} & \texttt{mass} & \# & \texttt{e} & \texttt{mass} & \# & \texttt{e} & \texttt{mass} & \# & \texttt{e} & \texttt{mass} & \# & \texttt{e} \\
 & & & & & & & & & & & & & & & & & \\
$1/2$ & $81706477$ & $4429936$ & $1/40$ & $33$ & $5$ & $1/128$ & $107$ & $20$ & $1/576$ & $2$ & $2$ & $1/2880$ & $1$ & $0$ & $1/20160$ & $1$ & $1$ \\
$1/4$ & $583827$ & $85387$ & $1/48$ & $144$ & $63$ & $1/144$ & $13$ & $10$ & $1/600$ & $2$ & $2$ & $1/3072$ & $4$ & $2$ & $1/30720$ & $1$ & $1$ \\
$1/6$ & $688$ & $195$ & $1/54$ & $1$ & $0$ & $1/160$ & $1$ & $1$ & $1/672$ & $1$ & $0$ & $1/3600$ & $1$ & $1$ & $1/32768$ & $1$ & $0$ \\
$1/8$ & $25837$ & $5127$ & $1/56$ & $9$ & $2$ & $1/192$ & $35$ & $13$ & $1/768$ & $17$ & $7$ & $1/3840$ & $1$ & $0$ & $1/57600$ & $1$ & $1$ \\
 $1/10$ & $23$ & $0$ & $1/60$ & $7$ & $1$ & $1/232$ & $1$ & $1$ & $1/864$ & $1$ & $1$ & $1/4032$ & $2$ & $2$ & $1/82944$ & $1$ & $1$ \\
 $1/12$ & $791$ & $312$ & $1/64$ & $229$ & $53$ & $1/240$ & $5$ & $2$ & $1/960$ & $2$ & $0$ & $1/4096$ & $2$ & $0$ & $1/134784$ & $1$ & $1$ \\ 
 $1/16$ & $3429$ & $850$ & $1/72$ & $15$ & $8$ & $1/256$ & $24$ & $4$ & $1/1024$ & $7$ & $2$ & $1/4608$ & $4$ & $2$ & $1/161280$ & $2$ & $1$ \\
 $1/18$ & $6$ & $1$ & $1/80$ & $7$ & $4$ & $1/288$ & $7$ & $5$ & $1/1152$ & $3$ & $2$ & $1/5760$ & $2$ & $1$ & $1/184320$ & $1$ & $0$ \\
$1/20$ & $34$ & $3$ & $1/84$ & $2$ & $2$ & $1/320$ & $1$ & $0$ & $1/1296$ & $1$ & $0$ & $1/6144$ & $3$ & $0$ & $1/688128$ & $1$ & $0$ \\
$1/24$ & $408$ & $170$ & $1/96$ & $79$ & $31$ & $1/384$ & $26$ & $9$ & $1/1536$ & $4$ & $2$ & $1/7168$ & $1$ & $0$ & $1/1179648$ & $1$ & $1$ \\
$1/28$ & $3$ & $0$ & $1/100$ & $1$ & $0$ & $1/448$ & $1$ & $0$ & $1/2048$ & $3$ & $1$ & $1/8736$ & $1$ & $1$ & $1/2419200$ & $1$ & $1$ \\ 
$1/32$ & $717$ & $159$ & $1/108$ & $2$ & $1$ & $1/480$ & $5$ & $1$ & $1/2304$ & $2$ & $2$ & $1/9216$ & $1$ & $1$ & $1/41287680$ & $1$ & $0$ \\ 
$1/36$ & $6$ & $3$ & $1/120$ & $5$ & $1$ & $1/512$ & $17$ & $4$ & $1/2688$ & $1$ & $0$ & $1/18432$ & $2$ & $1$
 \end{tabular}
\caption{{\scriptsize Number \# of classes (resp. \texttt{e} of exceptional classes) in ${\rm X}_{30}^\emptyset$ with mass \texttt{mass}}}
\label{tab:sizeautX30empty}
\end{table}

\ps

\subsection{Cases $R= {\bf A}_1$ and $R={\bf A}_2$}
\label{subsect:X30A1}
${}^{}$
We applied a strategy similar to the one described above, 
using the visible root system $R$ itself.
The ${\bf A}_2$ case presented no particular surprises, so we will not 
say anything about it. The ${\bf A}_1$ case was especially challenging.
Indeed, despite multiple enumerations of neighbors, 
it ceased to yield new lattices close to the end,
leaving a remaining mass of $3/4$. 
At this stage, finding the missing lattices by the neighbor method 
amounts to searching for a needle in a haystack.
Instead, we used an exceptional degree $3$ correspondence 
on ${\rm X}_{30}$, which we call the {\it triplication method}, 
and which we now briefly explain.\ps

Let $Q$ be the finite quadratic space $ - \resq \,({\rm A}_1 \perp {\rm A}_1 \perp {\rm A}_1)$,
and denote by $\mathcal{H}_n$ the genus of 
even lattices $H$ of rank $n$ satisfying $\resq H \simeq Q$. 

\begin{prop} 
\label{prop:finalgroupoideq}
Assume $n \equiv 6 \bmod 8$. There is a natural equivalence of groupoids between:
\begin{itemize}
\item[(i)] pairs $(L,\alpha)$ with 
$L$ a rank $n$ unimodular lattice and $\alpha$ a root of $L$. \ps
\item[(ii)] pairs $(H,w)$ with $H$ a lattice in $\mathcal{H}_{n-1}$ and 
$w \in \resq H$ such that ${\rm q}(w) \equiv 3/4 \bmod \Z$.
\end{itemize}
In this equivalence, we have $H \simeq L^{\rm even} \cap \alpha^\perp$ and 
$L^{\rm even} \simeq (H \perp \Z \alpha) + \Z (w+\alpha/2)$. 
\end{prop}

\begin{pf} 
We only sketch the proof.
The groupoid in (i) is naturally equivalent to that of pairs $(M,\alpha)$
with $M$ an even lattice in the genus of ${\rm D}_n$ and $\alpha$ a root of $M$, 
via $(L,\alpha) \mapsto (L^{\rm even},\alpha)$. Any such $M$ may be obtained as 
the orthogonal of some ${\rm D}_2 \simeq {\rm A}_1 \perp {\rm A}_1$
inside an even unimodular lattice $U$ of rank $n+2$. So $H:=M \cap \alpha^\perp$
is the orthogonal in $U$ of some ${\rm A}_1 \perp {\rm A}_1 \perp {\rm A}_1$,
showing $H \in \mathcal{H}_{n-1}$. The equivalence between (i) and (ii)
is then a consequence of Proposition~\ref{prop:nikulingroupoids} in the even context, 
with $A={\rm A}_1$ and $H=\resq A$.
\end{pf}

We have $Q = \Z/2 e_1 \perp \Z/2 e_2 \perp \Z/2 e_3$ 
with ${\rm q}(e_i) \equiv 3/4 \bmod \Z$ for $1 \leq i \leq 3$. 
For each $H$ in $\mathcal{H}_{n-1}$, there are thus exactly 
$3$ elements $w \in \resq H$ satisfying ${\rm q}(w) \equiv 3/4$, hence
$3$ corresponding pairs $(L,\alpha)$, explicitly given by 
the last formula in the proposition. 
The relation $H \simeq L \cap \alpha^\perp$ shows
$${\rm R}_2(L) \simeq {\bf A}_1\, \, {\rm or}\, \, {\bf A}_2 \,\iff\,{\rm R}_2(H) =\emptyset.$$ 
The construction above thus associates to an isometry class \([L] \in \mathrm{X}_n^{\mathbf{A}_1} \sqcup \mathrm{X}_n^{\mathbf{A}_2}\) a 3-element multiset of classes in \(\mathrm{X}_n^{\mathbf{A}_1} \sqcup \mathrm{X}_n^{\mathbf{A}_2}\) containing $[L]$.
This triple is easily computed.
Applied to our list of found lattices in ${\rm X}_{30}^{{\bf A}_1}$, this method allowed us to 
produce a list of rank $30$ unimodular lattices with root system ${\bf A}_1$ or ${\bf A}_2$ that is three 
times larger (but of course, with much redundancy).
The computation of the ${\rm BV}$ invariants of all the new classes 
happily led us to discover the $3$ 
remaining lattices in ${\rm X}_{30}^{{\bf A}_1}$, each having the mass $1/4$. 
This concludes the proof, up to the fact that these last $3$ lattices are not 
yet given as $d$-neighbors of ${\rm I}_{30}$. 
For this last step we use the following:

\begin{lemma} 
\label{lem:tripl2nei}
Let $H$ be an even lattice $\mathcal{H}_{n-1}$ with $n \equiv 6 \bmod 8$. 
Let $W$ be the $3$-element set of $w \in \resq \, H$ with ${\rm q}(w) \equiv 3/4 \bmod \Z$,
and for $w$ in $W$, let $L_w$ be the rank $n$ unimodular lattice associated to $(H,w)$
under the equivalence of Proposition~\ref{prop:finalgroupoideq}. 
Then $L_{w'}$ is a $2$-neighbor of $L_{w}$ for any $w' \neq w$ in $W$.
\end{lemma}

\begin{pf} 
Set $N=H \perp {\rm A}_1$ and write ${\rm A}_1 = \Z \alpha$.
For $w \in W$ we have $L_w^{\rm even} = N + \Z (w+\alpha/2)$, 
and we easily check $L_w \,=\, L_w^{\rm even} \,+\, \Z\, (w'+w'')$ where $\{w,w',w''\}=W$.
We have thus $L_w \cap L_{w'} \,=\, N + \Z \,(\alpha/2+w+w'+w'')$, and 
$L_w \cap L_{w'}$ has index $2$ in $L_w$ and $L_w'$.
\end{pf}

A neighbor form for the last $3$ lattices was finally obtained as follows.
For each such lattice $L$, we computed $2$-neighbors of $L$ 
until we found a lattice $L'$ belonging to our list and represented  
as a $d$-neighbor of ${\rm I}_{30}$ for some odd $d$ (this is very fast). 
We then computed $2$-neighbors of $L'$ with root system ${\bf A}_1$,
which can easily be done on neighbor forms using~\cite[Lemma 11.2]{cheuni}, 
until we found one with the same BV invariant as $L$.

\appendix

\end{document}